\documentclass[11pt,english]{article}
\usepackage{amssymb}
\usepackage{color}
\usepackage{tikz}
\usepackage[latin1]{inputenc}
\usepackage[english]{babel}
\usepackage{graphicx}
\usepackage{amsmath}
\usepackage{amssymb}
\usepackage{amsthm}
\usepackage{latexsym}
\usepackage{multirow}
\usetikzlibrary{matrix,positioning}
\usetikzlibrary{decorations.pathreplacing}
\usetikzlibrary{matrix,positioning}
\usetikzlibrary{decorations}
\usepackage{color}
\usepackage{tikz}
\usepackage[latin1]{inputenc}
\usepackage[english]{babel}
\usepackage{graphicx}
\usepackage{amsmath}
\usepackage{amssymb}
\usepackage{amsthm}
\usepackage{latexsym}
\usepackage{multirow}
\usetikzlibrary{matrix,positioning}
\usetikzlibrary{decorations}
\newtheorem{thm}{Theorem}
\newtheorem{prop}{Proposition}
\newtheorem{ex}{Example}
\newtheorem{re}{Remark}
\newtheorem{lem}{Lemma}

\newtheorem{cor}{Corollary}
\title{An analogue of the Robinson-Schensted-Knuth correspondence
and non-symmetric Cauchy kernels for truncated staircases}
\author{Olga Azenhas and Aram Emami}
\date{}
\begin{document}
\maketitle

\begin{abstract}
We prove a restriction of  an analogue of the Robinson--Schensted--Knuth  correspondence for semi-skyline augmented
fillings, due to Mason, to  multisets  of
cells of a   staircase  possibly truncated by a smaller
 staircase  at the upper left end corner, or  at  the bottom right
 end corner.
 The restriction to be imposed on the pairs of semi-skyline augmented fillings is that the pair of shapes, rearrangements of each other,
satisfies an inequality in the Bruhat
order,  w.r.t. the symmetric group,
 where one shape is bounded by the reverse of the other. For semi-standard  Young  tableaux the inequality means that the pair of their right keys is such that one key is bounded by the Sch\"utzenberger evacuation of the other. This bijection is then used to obtain 
 an expansion formula  of the non-symmetric Cauchy kernel,
 over staircases or truncated staircases,
 in  the basis of Demazure characters
of type $A$, and the basis of  Demazure atoms. The expansion implies
Lascoux expansion formula,  when specialised to staircases or truncated staircases,  and make explicit, in the latter,   the
 Young tableaux  in the Demazure crystal   by interpreting Demazure operators  via elementary bubble sorting operators acting on weak compositions.
\end{abstract}
{\bf Keywords: }Young tableau, semi-skyline augmented filling, analogue of the Robinson-Schensted-Knuth correspondence, isobaric divided differences, Demazure character, Demazure atom,  Demazure crystal graph, nonsymmetric Cauchy kernels.

MSC[2010] Primary 05E05. Secondary 05E10, 17B37
\section{Introduction and statement of results}
\label{sec:in}
The purpose of this paper is to give a bijective proof via the
standard Robinson--Schensted--Knuth (RSK)-type bijection \cite{knuth} for the
truncated staircase shape version of the Cauchy identity, due to Lascoux, where Schur polynomials are replaced by Demazure characters and Demazure atoms \cite{lascouxcrystal,fulascoux}.  To this aim   we build on the interesting analogue of
the RSK bijection, recently given by  Mason \cite{masonrsk},  where semi-standard tableaux (SSYTs) are replaced by
 semi-skyline augmented fillings (SSAFs). The later combinatorial objects are coming from the
Haglund--Haiman--Loehr  formula for non-symmetric Macdonald polynomials \cite{hagnon}. The  RSK correspondence is an important combinatorial bijection between two line arrays of positive integers (or non-negative integer matrices)
 and pairs of SSYTs of the same shape with applications to the representation theory of the Lie algebra $\mathfrak{gl}_n$, and to the theory of symmetric functions among others.  Mason's
 bijection has the advantage of giving information about the
filtration of irreducible representations of $\mathfrak{gl}_n$ by Demazure modules,
which is detected by the  key of a SSYT, after Lascoux and Sch\"utzenberger \cite{lasschutz,lascouxkeys}, and manifested in  the shape of  a SSAF \cite{masondemazure}. Although a   general Ferrers shape version of the Cauchy identity had been
given by Lascoux in \cite{lascouxcrystal}, aside from  the staircase shape, the characterization of the pairs of SSYTs  is expressed in a less explicit way.  Regarding to the shapes to be considered here, our expansions are explicit.  Lascoux's proof  in \cite{lascouxcrystal} and Fu--Lascoux's proof in \cite{fulascoux} (in type~$A$ case)
are different from the standard bijective proof based on an RSK-type
correspondence, which  is precisely what is done here.

\subsection{Crystals, Demazure crystals and keys}
Given the general Lie algebra  $\mathfrak{gl}_n(\mathbb{C})$, and
its quantum group $U_q(\mathfrak{gl}_n)$ - the $q$-analogue of the universal enveloping algebra $U(\mathfrak{gl}_n)$ -
finite-dimensional representations of $U_q(\mathfrak{gl}_n)$ are
also classified by the highest weight. Let $\lambda$ be a dominant
integral weight (i.e. a partition), and $V(\lambda)$ the
integrable representation with highest weight $\lambda$, and
$u_\lambda$ the highest weight vector. For a given permutation $w$
in the symmetric group $\mathfrak{S}_n$, the shortest
 in its class modulo the stabiliser of $\lambda$, the Demazure
module is defined to be $V_w(\lambda):=U_q(\mathfrak{g})^{>0}.u_{w
\lambda},$ and the Demazure character is the character of
$V_w(\lambda)$. (We refer the reader to \cite{hongkang,kwon} for details.)
In the early $90$'s Kashiwara \cite{kashi,kashiwaraq} has associated with $\lambda$ a crystal
 $\mathfrak{B}^\lambda$, which can be realised in type $A$ as  a coloured
directed graph having vertices  all 
SSYTs of shape $\lambda$ with entries $\le n$,
and  arrows $P \overset{i}\rightarrow P'$ if and only $f_iP = P'$, for each crystal (coplactic) operator $f_i$, $1\le i<n$.
 The coloured directed graph $\mathfrak{B}^\lambda$ reflects the combinatorial structure of the  given integrable representation $V(\lambda)$ and the relationship between  $\mathfrak{B}^\lambda$ and the module  $V(\lambda)$  can be made precise using the notion of crystal basis for $V(\lambda)$ \cite{kashi,kashiwaraq}.  Littelmann  conjectured \cite{littel} and
Kashiwara  proved \cite{kashiwara} that the intersection of a crystal basis of
$V_\lambda$ with $V_w(\lambda)$ is a crystal basis for
$V_w(\lambda)$. The resulting subset $\mathfrak{B}_{w
\lambda}\subseteq \mathfrak{B}^\lambda$ is called Demazure crystal, and
the Demazure character corresponding to $\lambda$ and $w$, is  the polynomial combinatorially expressed by the SSYTs in  the
 Demazure crystal $\mathfrak{B}_{w \lambda}$. These polynomials are the  key polynomials corresponding to $\lambda$ and $w$  in Reiner-Shimozono's work \cite{reiner}.

  The irreducible representations of $\mathfrak{gl}_n$ have then a filtration by Demazure modules, compatible with the Bruhat order of $\mathfrak{S}_n$ and the crystal structure.
 That is,  $\mathfrak{B}_{w'\lambda}\subset \mathfrak{B}_{w\lambda}$ whenever  $w'< w$ in the Bruhat order on the classes modulo the stabiliser of $\lambda$ and $\mathfrak{B}^{\lambda}=\bigcup_{ w\in\mathfrak{S}_n}{\mathfrak{B}}_{w\lambda}$. In particular, if $\omega$ is the longest permutation of $\mathfrak{S}_n$, $\mathfrak{B}^{\lambda}=\mathfrak{B}_{\omega\lambda}$.
 Given that the Schur polynomial $s_\lambda$ is expressed combinatorially by
all  SSYTs of shape $\lambda$ and entries $\le n$, in the late $80$'s, Lascoux and Sch\"utzenberger \cite{lasschutz,lascouxkeys} identified the SSYTs
 contributing to the key polynomial  corresponding to $\lambda$ and $w$ via a  condition in the Bruhat order
involving their right keys. That is, the key polynomial is decomposed into a sum of Demazure atoms \cite{masondemazure} (standard bases \cite{lascouxkeys})
 which is equivalent to a decomposition of the Demazure crystal ${\mathfrak{B}}_{w\lambda}$.
\subsection{Demazure characters and Demazure operators}
 The  Demazure character (or key polynomial) $\kappa_\alpha$ and the Demazure atom $\widehat\kappa_\alpha$, with $\alpha\in \mathbb{N}^n$ (a rearrangement of $\lambda$), are also generated recursively by the application of  Demazure operators  (or isobaric divided differences \cite{lascouxdraft}) $\pi_i$ and $\hat\pi_i:=\pi_i-1$, for  $1\le i<n$, respectively, to the monomial $x^\lambda$. Such operators are defined for each simple reflection  of $\mathfrak{S}_n$ and satisfy the braid relations of a Coxeter group. See Section~\ref{sec:isob} for  precise definitions, recursive rules and combinatorial descriptions.
   They were introduced by Demazure \cite{demazure}
   for all Weyl groups and  were  studied combinatorially, in the case of $\mathfrak{S}_n$,  by  Lascoux and Sch\"utzenberger \cite{lasschutz,lascouxkeys} who produce a
   crystal structure by providing a combinatorial version for Demazure operators  in terms of  crystal (or coplactic) operators \cite{thibon}.
\subsection{Combinatorics of nonsymmetric  Macdonald polynomials and RSK analogue}
 Non symmetric Macdonald  polynomials $E_\alpha(X,q,t)$, with $\alpha\in\mathbb{N}^n$ (we are assuming zero in $\mathbb{N}$),  form a basis of $\mathbb{C}(q,t)[x_1,\dots, x_n]$, and were introduced and studied 
 by Opdam \cite{opdam}, Cherednik \cite{cherednick}, and Macdonald \cite{mac}. Their representational-theoretical nature in connection with Demazure characters has been   investigated by Sanderson \cite{sanderson} and Ion \cite{ion}. In 2004, Haglund, Haiman and Loher gave a combinatorial formula for non symmetric Macdonald polynomials \cite{hagnon}. Specialising the Haglund--Haiman--Loehr
formula for the nonsymmetric Macdonald polynomial $E_\alpha(x;q^{-1};t^{-1})$, \cite[Corollary~3.6.4]{hagnon}, by letting  $q,t\rightarrow 0$, implies that
$E_\alpha(x;\infty;\infty)$ is 
combinatorially expressed by all SSAFs of shape $\alpha$. See Section~\ref{sec:ssaf} for details on SSAFs.
 These polynomials are also a decomposition of the Schur
polynomial $s_\lambda$, with $\lambda$ the decreasing rearrangement of $\alpha$.
 Semi-skyline augmented fillings are in bijection with SSYTs so that the content is the same and the right key  is given by the shape of the SSAF \cite{masonrsk}. The Demazure atom $~\widehat\kappa_{\alpha}$ and $E_{\alpha}(x;\infty;\infty)$
 are then equal \cite{hagnon,masondemazure}.
An interesting analogue of the RSK bijection was given by Mason  \cite{masonrsk},
where SSYTs are replaced by SSAFs which manifest the keys.
 \subsection{Our results}
We consider the following Ferrers diagram,  in the French convention,
$\lambda=(m^{n-m+1},$ $m-1,$$\dots,n-k+1)$, $~1\leq m\leq n$, $~1\leq k\leq n$, $~n+1\le
m+k$, shown in green colour  in  Figure~\ref{fig:trunc}.
\begin{figure}[here]
 \begin{center}
\begin{tikzpicture} [scale=0.3] \filldraw[color=green!25] (-1,0)
rectangle (3,5); \filldraw[color=green!25] (3,0)
rectangle (5,2);  \filldraw[color=green!25] (-1,0)
rectangle (3,5); \filldraw[color=green!25] (3,1) rectangle
(5,3);\filldraw[color=green!25] (3,3) rectangle (4,4);
\draw[line width=0.5pt](-1,0)-- (-1,5);
\draw[line width=0.5pt, red](-1,6)-- (-1,8);
\draw[line width=0.5pt, red](-1,5)-- (-1,6);
\draw[line width=0.5pt](0,0)-- (0,5);
\draw[line width=0.5pt, red](0,8)-- (0,7);
\draw[line width=0.5pt, red](1,7)-- (1,6);
\draw[line width=0.5pt](1,0)-- (1,5);
\draw[line width=0.5pt](2,0)-- (2,5);
\draw[line width=0.5pt](3,0)-- (3,5);
\draw[line width=0.5pt](4,0)-- (4,4);
\draw[line width=0.5pt](-1,0)-- (5,0);
\draw[line width=0.5pt, red](4,0)-- (7,0);
\draw[line width=0.5pt](-1,1)-- (5,1);
\draw[line width=0.5pt](4,0) rectangle  (5,3);
\draw[line width=0.5pt, red](7,1)-- (6,1);
\draw[line width=0.5pt](-1,2)-- (5,2);
\draw[line width=0.5pt, red](6,2)-- (5,2);
\draw[line width=0.5pt](-1,3)-- (4,3);
\draw[line width=0.5pt](-1,4)-- (4,4);
\draw[line width=0.5pt](-1,5)-- (3,5);
\draw[line width=0.5pt,red](1,6)-- (2,6);
\draw[line width=0.5pt,red](2,5)-- (2,6);
\draw[line width=0.5pt, red](7,0)-- (7,1);
\draw[line width=0.5pt, red](6,2)-- (6,1);
\draw[line width=0.5pt, red](5,3)-- (5,2);
\draw[line width=0.5pt, red](5,3)-- (4,3);
\draw[line width=0.5pt](4,2)rectangle (5,3);
\draw[line width=0.5pt, red](-1,8)-- (0,8);
\draw[line width=0.5pt, red](1,7)-- (0,7);
\draw[line width=0.5pt] (-4,0) -- (-4,8);
\draw[line width=0.5pt] (-4,0) -- (-3.8,0);
\draw[line width=0.5pt] (-4,8) -- (-3.8,8);
 \draw[line width=0.5pt] (-2,0) -- (-2,5);
  \draw[line width=0.5pt] (-2,0) -- (-1.8,0);
 \draw[line width=0.5pt] (-2,5) -- (-1.8,5);
\node at (-2.7,2.5){$k$}; \node at (-4.7,4){$n$};
 \draw[line width=0.5pt] (-1,-1) -- (5,-1);
 \draw[line width=0.5pt] (-1,-1) -- (-1,-0.8);
\draw[line width=0.5pt] (5,-1) -- (5,-0.8);\node at (2.5,-1.7){$m$};
\end{tikzpicture}
\end{center}
\caption{The truncated Ferrers shape $\lambda$, in green, fitting the $k$ by $ m$ rectangle so that the  staircase of size $n$ is the smallest containing $\lambda$. If $k\le m$,  $(k,k-1,\ldots,1)$ is the biggest staircase inside of $\lambda$.}
\label{fig:trunc}
\end{figure}
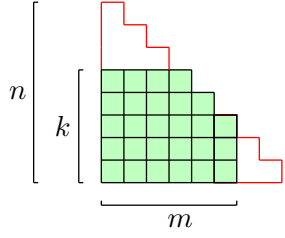
Theorem~\ref{maint}, in Section~\ref{sec:main}, exhibits an RSK-type bijection between multisets of cells of $\lambda$ and pairs of SSAFs where the image is described by a Bruhat  inequality  between the keys of the recording  and the insertion fillings.
  When $m+k=n+1$ then $\lambda$ is a rectangle and it reduces to the ordinary RSK correspondence in the sense that the inequality on the right keys is relaxed. This bijection is used in Section~\ref{sec:kernels}, Theorem~\ref{minmin}, to give an expansion of the non-symmetric Cauchy kernel
 $\prod_{\begin{smallmatrix}(i,j)\in \lambda\end{smallmatrix}}(1-x_iy_j)^{-1}$ in the basis of Demazure characters, and the basis of Demazure atoms.
 The kernel expands
   \begin{eqnarray}\label{0k<m}
\prod_{\begin{smallmatrix}(i,j)\in
\lambda\\
k\le m\end{smallmatrix}}(1-x_iy_j)^{-1}
=
\sum_{\begin{smallmatrix}\mu \in
\mathbb{N}^k
\end{smallmatrix}}\widehat{\kappa}_{\mu}(x)\kappa_{(0^{m-k},\alpha)}(y),
\end{eqnarray}
with  $\alpha=(\alpha_1,\ldots,\alpha_k)\in \mathbb{N}^k$  such that, for each $ i=k,\ldots, 1$, the entry $\alpha_i$ is the maximum element among the last  $\min\{i,n-m+1\}$ entries of $\mu$ in reverse order,  after deleting $\alpha_j$, for $i<j\le k$.
If $m<k$, the formula is symmetrical,   swapping in \eqref{0k<m} $x$ with $y$, and $k$ with $m$.
If $\lambda $ is a rectangle, $\alpha$ is a partition and the classical Cauchy identity  in the basis of Schur polynomials is recovered; and if $\lambda$ is the staircase of length $n$,  $\alpha$ is the reverse of $\mu$, and Lascoux's expansion in Theorem~6 of \cite{lascouxcrystal}, and in \cite{fulascoux}, 
is also recovered.
 For truncated staircases, the expansion \eqref{0k<m} implies Lascoux's formula in Theorem~7 of \cite{lascouxcrystal}, and  makes explicit the SSYTs of the Demazure crystal.



Our paper is organised in six sections. In Section~\ref{sec:weak},  we first recall the definitions of compositions, Young tableaux, and key tableaux,  then  the Bruhat orders of the symmetric group $\mathfrak{S}_n$, their several characterizations,
  and their conversions  to a 
  $\mathfrak{S}_n$-orbit.
 In Section~\ref{sec:ssaf}, we  review the
necessary terminology and theory of SSAFs, in particular, 
the
RSK analogue for SSAFs along with    useful  properties 
for the next section. 
 Then, in Section~\ref{sec:main}, we give our
main result, Theorem~\ref{maint}, and  an illustration of the bijection described  in this theorem. 
 Section~\ref{sec:iso} is devoted to the necessary theory of crystal graphs in type A in connection with the combinatorial descriptions of Demazure operators and  the two families of key polynomials  to be used  in the last section, in particular, in Lemma~\ref{keyintersection}. Finally, in the last section, we apply the bijection provided in Theorem~\ref{maint}
to obtain expansions of Cauchy kernels over  truncated stair cases as described in  Theorem~\ref{minmin}.
\section{Weak compositions, key tableaux and  Bruhat orders on $\mathfrak{S}_n$ and orbits} \label{sec:weak}
\subsection{Young tableaux and key tableaux} Let $\mathbb{N}$ denote the set of non-negative integers. Fix a
 positive integer $n$, and define  $[n]$ as  the set $\{1,\dots
,n\}.$ A  weak composition $\gamma=(\gamma_1,\dots,\gamma_n)$ is a
vector in $\mathbb{N}^n.$ If $\gamma_i=\cdots=\gamma_{i+k-1}$, for some $k\ge 1$, then
we also write $\gamma=(\gamma_1,\dots,\gamma_{i-1},$
$\gamma_i^k,\gamma_{i+k},$ $\ldots,\gamma_n)$. We often concatenate weak compositions $\alpha\in \mathbb{N}^r$ and $\beta\in\mathbb{N}^s$, with $r+s=n$, to form the weak composition $(\alpha,\beta)=$ $(\alpha_1,\ldots,\alpha_r,$ $\beta_1,\ldots,\beta_s)\in\mathbb{N}^n$.
A weak composition $\gamma$ whose entries are in weakly
decreasing order, that is, $\gamma_1 \geq \cdots \geq \gamma_n,$ is said to be a
 partition.
Every weak composition $\gamma$ determines a unique partition $\gamma^+$
obtained by arranging the entries of $\gamma$ in weakly decreasing
order. A partition $\lambda=(\lambda_1,\dots, \lambda_n)$ is
identified  with its Young diagram (or Ferrers shape) $dg(\lambda)$ in French
convention, an array of left-justified cells with $\lambda_i$ cells
in row $i$ from the bottom, for $1 \leq i \leq n.$  The cells are
located in the diagram  $dg(\lambda)$ by their row and column
indices
 $(i,j)$, where $1 \leq i \leq n$ and $1 \leq j \leq \lambda_i.$
The number $\ell(\lambda)$ of rows in the Young diagram $dg(\lambda)$ with a positive number of cells is said to be the  length of the partition $\lambda$.  For instance, for $n=4$, if $\lambda=(4,2,2,0)$, 
 $\ell(\lambda)=3$, and the Young diagram of  $\lambda=(4,2,2,0)$, is
 
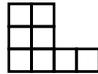
\begin{center}
 \begin{tikzpicture}[scale=0.6]
 \draw [line width=1pt] (0.5,1)rectangle(1,1.5);
\draw [line width=1pt] (0,0)rectangle(0.5,0.5);\draw [line width=1pt]
(0.5,0)rectangle(1,0.5);
\draw [line width=1pt] (1,0)rectangle(1.5,0.5);\draw [line width=1pt]
(1.5,0)rectangle(2,0.5);
\draw [line width=1pt] (0,0.5)rectangle(0.5,1);
\draw [line width=1pt] (0.5,0.5)rectangle(1,1);
\draw [line width=1pt] (0,1)rectangle(0.5,1.5);
\end{tikzpicture}
\end{center}
 A filling of shape $ \lambda$ (or a filling of $dg(\lambda)$), in the alphabet $[n]$, is a map $T : dg(\lambda)\rightarrow
 [n].$ A  semi-standard Young tableau (SSYT) $T$ of shape $sh(T)=\lambda$, in the alphabet $[n]$,  is a
filling of $dg(\lambda)$ which is weakly increasing in each row from
left to right and strictly increasing up in each column. Let SSYT$_n$ denote the set of all semi-standard Young tableaux with entries $\le n$. The   column word of $ T\in$SSYT$_n$  is the word, over the alphabet $[n]$, which consists
of the entries of each column, read top to bottom and  left
to right. The { content}
or { weight} of  $ T\in$SSYT$_n$ is the content or weight of its column word in the alphabet $[n]$, which is the weak composition
$c(T)=(\alpha_1,\dots,\alpha_n)$
 such that $\alpha_i$ is the multiplicity of $i$ in the column word of $T$. For instance, a SSYT
  of shape  $\lambda=(4,2,2,0)$,  in the alphabet $[4]$, with $col(T)=32143123$ and content $c(T)=(2,2,3,1)$   is
 \begin{center}
 \begin{tikzpicture}[scale=0.6]
 \draw [line width=1pt] (0.5,1)rectangle(1,1.5);\node at (0.75,1.25){\tiny4};
\draw [line width=1pt] (0,0)rectangle(0.5,0.5);\draw [line width=1pt]
(0.5,0)rectangle(1,0.5);
\draw [line width=1pt] (1,0)rectangle(1.5,0.5);\draw [line width=1pt]
(1.5,0)rectangle(2,0.5);
\draw [line width=1pt] (0,0.5)rectangle(0.5,1);
\draw [line width=1pt] (0.5,0.5)rectangle(1,1);
\draw [line width=1pt] (0,1)rectangle(0.5,1.5);
\node at
(0.25,0.25){\tiny1};\node at (0.75,0.25){\tiny1};\node at (1.25,0.25){\tiny2};\node at
(1.75,0.25){\tiny3};\node at (0.25,0.75){\tiny2};\node at (0.75,0.75){\tiny3};\node at
(0.25,1.25){\tiny3};\node at (-0.7,0.75){$T=$};
\end{tikzpicture}
\end{center}
A {key tableau} is a semi-standard Young tableau such that the set of entries
in the $(j + 1)^{th}$ column is a subset of the set of entries in
the $j^{th}$ column, for all $ j$. There is a bijection  \cite{reiner}
 between weak compositions in $\mathbb{N}^n$ and keys in the alphabet $[n]$ given by $\gamma\rightarrow key(\gamma),$ where
  $key(\gamma)$ is the key such
that for all $j$, the first $\gamma_j$ columns contain the letter
$j$. Any key tableau is of the form $key(\gamma)$ with $\gamma$ its
content and $\gamma^+$ the shape. (See Example~\ref{example1}.) When $\gamma=\gamma^+$ one obtains the key of shape and content $\gamma$, called  Yamanouchi tableau of shape $\gamma$. Note that $evac(key(\gamma_1,\ldots,\gamma_n))=key(\gamma_n,\ldots,\gamma_1)$ where  $"evac"$ denotes the Sch\"utzenberger's evacuation on SSYTs \cite{schQ,fulton,stanley}.

\subsection{Bruhat orders on $\mathfrak{S}_n$ and orbits}
\label{sec:Bruhat}
The symmetric group  $\mathfrak{S}_n$ is generated by the simple transpositions $s_i$ which  exchanges $i$ with $i+1$, $1\le i<n$, and they satisfy the Coxeter relations \begin{equation}\label{cox} s_i^2=1, \quad s_is_j=s_js_i, \;\;\text{for}\;\; |i-j|>1, \quad s_is_{i+1}s_i=s_{i+1}s_is_{i+1}.\end{equation}
 Given $\sigma\in \mathfrak{S}_n$, if $\sigma=s_{i_N}\cdots s_{i_1}$ is a decomposition of $\sigma$ into simple transpositions, where the number $N$  is minimised, we say that we have a  reduced decomposition of $\sigma$, and $N$ is called  its length $\ell(\sigma)$. In this case, we say that the sequence of indices $(i_N,\ldots , i_1)$ is a  reduced word for $\sigma$. The unique element
of maximal length in $\mathfrak{S}_n$  is denoted by $\omega$. It is a well known fact that any two reduced decompositions for $\sigma$ are connected by a sequence of the  last two Coxeter relations \eqref{cox}, called commutation and braid relations, respectively.  Recall that a pair $(i,j)$, with $i<j$, is said to be an inversion of $\sigma$ if $\sigma(i)>\sigma(j)$. The number of inversions of $\sigma$ is the same as $\ell(\sigma)$ \cite{manivel}.

 The (strong) Bruhat order in $\mathfrak{S}_n$ is a partial order in $\mathfrak{S}_n$ compatible with the length of a permutation. For any  $\theta$
 in $\mathfrak{S}_n$ and $t$ a transposition, we write \begin{equation}\label{closure}\theta<t\theta\;\mbox{ if and only if  $\ell(\theta)<\ell(t\theta)$}.\end{equation}
   The transitive closure of these relations is said to be the (strong) Bruhat order in $\mathfrak{S}_n$.
 Regarding $\theta$ as the linear array $\theta_1\theta_2\dots\theta_n$ with $\theta(i)=\theta_i$, the Bruhat order says that $\theta<t\theta$ with $t$ a transposition if and only if $t$ exchanges $\theta_i$ and $\theta_j$ with $\theta_i<\theta_j$ for some $i<j$. It can be shown that if $t\theta =\sigma$, there is also a transposition $t^{\prime}$ such that $\theta t^{\prime}=\sigma$ \cite{bourbaki}.
 We recall
the subword property of the  (strong)  Bruhat order in a Coxeter group.
\begin{prop}{\em \cite{ bourbaki}} Let $\theta$, $\sigma$ in $\mathfrak{S}_n$ and $(i_N,\ldots,  i_1)$ a reduced word for $\sigma$, then $\theta\le \sigma$ if and only if there exists a subsequence  of $(i_N,\ldots,  i_1)$ which is a reduced word for $\theta$.
\end{prop}
The maximal length element $\omega$ is the maximal element of the Bruhat order, $\sigma\le \omega$, for any $\sigma\in \mathfrak{S}_n$, and it satisfies $\omega^2=1$. Besides, its left and right translations $\sigma\rightarrow\omega\sigma$ and $\sigma\rightarrow\sigma\omega$ are anti automorphisms for the Bruhat order.

 Let $\lambda=(\lambda_1,\dots,\lambda_n)$ be a partition, and $\mathfrak{S}_n\lambda$ the  $\mathfrak{S}_n$-orbit  of $\lambda$. The  $\mathfrak{S}_n$-stabiliser of $\lambda$, $stab_{\lambda}:=\{\sigma \in \mathfrak{S}_n :
\sigma\lambda=\lambda\}$,  is the  parabolic subgroup
generated by $\{s_i,$$ 1\le i<n: s_i\lambda=\lambda\}$.
 For each coset $w stab_\lambda$ in $\mathfrak{S}_n/stab_{\lambda}$ we may always choose $w$ to be the shortest permutation in $w stab_\lambda$ \cite{bourbaki}.
This allows $\mathfrak{S}_n/stab_{\lambda}$,
 with cardinality  $n!/|stab_\lambda|$,  to be identified  with $\mathfrak{S}_n\lambda$ \cite{Brenti, bourbaki, reflection}.
 The  Bruhat order 
 on $\mathfrak{S}_n/stab_{\lambda}$ is the restriction to the minimal coset representatives, \cite{bourbaki,Stembridge0,Stembridge}, and can  be converted to an ordering of $\mathfrak{S}_n\lambda$ by taking the transitive closure
 of the relations
 \begin{equation}\label{inducedbruhat}\gamma< t\gamma,\;\mbox{if $\gamma_i>\gamma_{j}$, $i<j$, and $t$ the transposition $(i\,j)$ ($\gamma=(\gamma_1,\ldots,\gamma_n)\in
 \mathfrak{S}_n\lambda$).}
 \end{equation}
We call this ordering the Bruhat order on $\mathfrak{S}_n\lambda$.
 If we replace, in \eqref{closure}, $t$ with the simple transposition $s_i$,
 the transitive closure of a such relations defines  the left weak Bruhat order on $\mathfrak{S}_n$. Its restriction   to
   $\mathfrak{S}_n/stab_\lambda$
 is then converted to an ordering in  $\mathfrak{S}_n\lambda$ by replacing, in \eqref{inducedbruhat}, $t$ with $s_i$. This conversion of the left weak Bruhat order  on $\mathfrak{S}_n$
  to $\mathfrak{S}_n\lambda$ is interestingly described by elementary  bubble sorting operators \cite{hivert}.
The  elementary  bubble sorting operation $\pi_i$, $1\le i<n$,
  on words $\gamma_1\gamma_2\cdots\gamma_n$ of length n (or weak compositions in $\mathbb{N}^n$),  sorts the letters in positions
$i$ and $i + 1$ in weakly increasing order, that is,
 it swaps $\gamma_i$ and $\gamma_{i+1}$ if $\gamma_i > \gamma_{i+1}$, or fixes $\gamma_1\gamma_2\cdots\gamma_n$  otherwise.
    Define now the partial order on $\mathfrak{S}_n\lambda$ by taking the transitive closure of the relations $\gamma<\pi_i \gamma$ when $\gamma_i > \gamma_{i+1}$, ($\gamma\in  \mathfrak{S}_n\lambda$ and $1\le i<n$).
It can also be proved that the elementary  bubble sorting operations $\pi_i$, $1\le i<n$, satisfy the relations \begin{equation}\label{bubblerelation}\pi_i^2=\pi_i, \; \pi_i\pi_{i+1}\pi_i=\pi_{i+1}\pi_i\pi_{i+1},\;\mbox{ and }\;\pi_i\pi_j=\pi_j\pi_i, \;\mbox{for}\;|i-j|>1.\end{equation}

The  set of minimal length coset
representatives of  $\mathfrak{S}_n/stab_{\lambda}$ may be described as $\{\sigma$ $\in$ $\mathfrak{S}_n:$ $\ell(\sigma s_i)>$ $\ell(\sigma),$ $s_i\in stab_\lambda\}$ \cite{Brenti, bourbaki, reflection}.
We now recall  a construction of the minimal length coset
representatives for  $\mathfrak{S}_n/stab_{\lambda}$, due to Lascoux, in \cite{lascouxcrystal}, where the notion of key tableau is used. This allows  to convert the tableau criterion for the Bruhat order in $\mathfrak{S}_n$ to a tableau criterion for the
Bruhat order \eqref{inducedbruhat}  in $\mathfrak{S}_n\lambda$.
Recall that the bijection between staircase keys of shape
$(n,n-1,\dots,1)$ and permutations in  $\mathfrak{S}_n$  gives the well-known
tableau criterion for the Bruhat order in
$\mathfrak{S}_n$ \cite{Ehresmann,manivel}.
\begin{prop}\label{bruh} {\em \cite{manivel}}
Let $\sigma,$ $\beta$ $\in \mathfrak{S}_n,$ we have $\sigma \le
\beta$ if and only if
 $key(\sigma(n,\dots,$ $1)) \leq
key(\beta(n,\dots,1))$ for the entrywise comparison. \end{prop}
In \cite{lascouxcrystal}, Lascoux  constructs  the shortest permutation $w$ in the coset $w stab_\lambda$ such that $w\lambda=\gamma\in \mathbb{N}^n$ using the key tableau of $\gamma$    as follows: firstly,
add the complete column $[n\dots 1]$ as the left most column of
$key(\gamma),$ if $\gamma$ has an entry equal to zero;  secondly, write the elements of the right most column of
$key(\gamma)$ in increasing order then  the new elements that
appear in the column next to the last in increasing order and so on until the first column. The resulting word is the desired permutation $w$ in $\mathfrak{S}_n$.
\begin{ex}\label{example1}
Let $\gamma=(1,3,0,0,1)$ and its $\mathfrak{S}_5$-stabiliser $stab_{(3,1,1,0,0)}=<s_2,s_4>$, the parabolic subgroup generated by   the simple transpositions of  $\mathfrak{S}_5$ that leave  $\gamma$ invariant. Let

\begin{tikzpicture}[scale=0.7]
\draw [line width=1pt] (0,0)rectangle(0.5,0.5);\draw [line width=1pt]
(0.5,0)rectangle(1,0.5);
\draw [line width=1pt] (1,0)rectangle(1.5,0.5);\draw [line width=1pt]
(0,0.5)rectangle(0.5,1);
\draw [line width=1pt] (0,1)rectangle(0.5,1.5);

\node at
(0.25,0.25){\small1};\node at (0.75,0.25){\small2};\node at (1.25,0.25){\small2};\node at (0.25,0.75){\small2};\node at
(0.25,1.25){\small5};\node at (-1.4,0.0){\small$key(\gamma)=$};
\end{tikzpicture}.
First
add the complete column $[5,4,3,2,1],$ to get
\begin{tikzpicture}[scale=0.7]
\draw [line width=1pt] (0,0)rectangle(0.5,0.5);\draw [line width=1pt]
(0.5,0)rectangle(1,0.5);
\draw [line width=1pt] (1,0)rectangle(1.5,0.5);\draw [line width=1pt]
(1.5,0)rectangle(2,0.5);
\draw [line width=1pt] (0,0.5)rectangle(0.5,1);
\draw [line width=1pt] (0.5,0.5)rectangle(1,1);
\draw [line width=1pt] (0,1)rectangle(0.5,1.5);
\draw [line width=1pt] (0.5,1)rectangle(1,1.5);
\draw [line width=1pt] (0,1.5)rectangle(0.5,2);
\draw [line width=1pt] (0,2)rectangle(0.5,2.5);
\node at
(0.25,0.25){\small1};\node at (0.75,0.25){\small1};\node at (1.25,0.25){\small2};\node at
(1.75,0.25){\small2};\node at (0.25,0.75){\small2};\node at (0.75,0.75){\small2};
\node at (0.25,1.25){\small3};\node at (0.75,1.25){\small5};
\node at (0.25,1.75){\small4};\node at (0.25,2.25){\small5};
\end{tikzpicture}
Hence, $w=21534=s_1s_4s_3$ is the shortest permutation in the coset $w\,stab_{(3,1,1,0,0)}$.
\end{ex}
\begin{thm}\label{vBru}
Let $\alpha_1$ and $\alpha_2$ be in the $\mathfrak{S}_n\lambda$.
 Then

 $(a)$ $\alpha_1\leq \alpha_2$
 if and only if $key(\alpha_1)\leq key(\alpha_2).$

 $(b)$ $\alpha_1\leq \alpha_2$ if and only if $evac( key(\alpha_2))\leq evac( key(\alpha_1))$.
\end{thm}
\begin{proof} $(a)$
Let $\sigma_1$ and $\sigma_2$ be the shortest length representatives of $\mathfrak{S}_n/stab_\lambda$ such that
$\sigma_1\lambda=\alpha_1$, $\sigma_2\lambda=\alpha_2$.
 Then,  $\alpha_1\leq
\alpha_2$ if and only if $\sigma_1\leq \sigma_2$ in Bruhat order, and, by Proposition~\ref{bruh}, this means
 $key(\sigma_1(n,\dots,1))\leq key(\sigma_2(n,\dots,1)).$  Using the constructions of $\sigma_1$ and $\sigma_2$ explained above this is equivalent to say that  $key(\alpha_1)\leq
key(\alpha_2).$

$(b)$ Recall that  $evac( key(\alpha))= key(\omega\alpha)$.\end{proof}
\section {Semi-skyline augmented fillings}
\label{sec:ssaf}
\subsection{Definitions and properties}
\label{sec:ssaf1}
We follow closely the conventions and terminology in
\cite{hagmac,hagnon} and \cite{masondemazure,masonrsk}. A weak composition $\gamma = (\gamma _1,
\dots , \gamma_n)$ is visualised as a diagram consisting of $n$
columns, with $\gamma_j$ boxes in column $j$, for $1\le j \le n$. Formally, the $column
~diagram$ of $\gamma$ is the set $ dg^{\prime}(\gamma) = \{(i, j)\in
\mathbb{N}^2 : 1 \leq j \leq n, 1 \leq i \leq \gamma_j\}$ where the
coordinates are in French convention, $i$ indicates the vertical coordinate, indexing the
rows, and $j$ the horizontal coordinate, indexing the columns. (The prime reminds
that the components of $\gamma$ are the columns.)  The number of
cells in a column is called the height of that column and a cell $a$
in a column diagram is denoted $a=(i,j),$ where $i$ is the row index
and $j$ is the column index. The $augmented~ diagram$ of $\gamma$,
$\widehat{dg}(\gamma)=dg^{\prime}(\gamma)\cup \{(0,j): 1\leq j \leq
n\}$, is the column diagram with $n$ extra cells adjoined in row
$0$. This adjoined row is called the $basement$ and it always
contains the numbers 1 through $n$ in strictly increasing order. The
shape of $\widehat{dg}(\gamma)$ is defined to be $\gamma.$ For
example, column diagram and the augmented diagram for
$\gamma=(1,0,3,0,1,2,0)$ are
$$
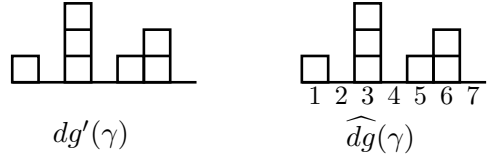
\begin{tikzpicture}[scale=0.35] \draw[line width=1pt] (0,0)
rectangle (1,1); \draw[line width=1pt] (2,0) rectangle
(3,3);\draw[line width=1pt] (4,0) rectangle (5,1);\draw[line
width=1pt] (5,0) rectangle (6,2); \draw[line width= 1pt] (0,0) --
(7,0); \draw[line width= 1pt] (2,1) -- (3,1);\draw[line width= 1pt]
(2,2) -- (3,2);\draw[line width= 1pt] (5,1) -- (6,1);
\draw[line width=1pt] (11,0) rectangle (12,1); \draw[line width=1pt]
(13,0) rectangle (14,3);\draw[line width=1pt] (15,0) rectangle
(16,1);\draw[line width=1pt] (16,0) rectangle (17,2); \draw[line
width= 1pt] (11,0) -- (18,0); \draw[line width= 1pt] (13,1) --
(14,1);\draw[line width= 1pt] (13,2) -- (14,2);\draw[line width=
1pt] (16,1) -- (17,1);
\node at (3,-2) {$dg^{\prime}(\gamma)$}; \node at (11.5, -0.5) {\small1};
\node at (12.5, -0.5) {\small2};
 \node at (13.5, -0.5) {\small3}; \node at (14.5, -0.5) {\small4}; \node at (15.5, -0.5) {\small5}; \node at (16.5, -0.5) {\small6}; \node at (17.5, -0.5) {\small7};
\node at (14, -2) {$\widehat{dg}(\gamma)$};
\end{tikzpicture}$$
An augmented filling  $F$ of an augmented diagram
$\widehat{dg}(\gamma)$ is a map $F: \widehat{dg}(\gamma)\rightarrow
[n],$ which can be pictured as an assignment of positive integer
entries to the non-basement cells of $\widehat{dg}(\gamma).$ Let
$F(i)$ denote the entry in the $i^{th}$ cell of the augmented
diagram encountered when $F$ is read across rows from left to right,
beginning at the highest row and working down to the bottom row.
This ordering of the cells is called the reading order. A cell $a
= (i, j)$ precedes a cell $b = (i^{\prime}, j^{\prime})$ in the
reading order if either $i^{\prime} < i$ or $i^{\prime} = i$ and
$j^{\prime} >j.$ The reading word  of $F$ is obtained by recording
the non-basement entries in reading order. The  content of an
augmented filling  $F$ is the  weak composition
$c(F)=(\alpha_1,\dots,\alpha_n)$ where   $\alpha_i$ is the number of
non-basement cells in $F$ with entry $i,$ and $n$ is the number of
basement elements.
The  standardization of  $F$ is the unique augmented filling
that one obtains  by sending the $i^{th}$ occurrence of $j$ in the
reading order to $i+\sum _{m=1}^{j-1}\alpha_m.$

 Let $a, b, c \in \widehat{dg}(\gamma)$ three cells
 situated as follows,
$\begin{tikzpicture}[scale=0.35] \draw[line width=1pt] (0,0)
rectangle (1,2); \draw[line width=1pt] (2.5,1) rectangle (3.5,2);
 \draw[line width= 1pt] (0,1) -- (1,1);
\node at (0.5,0.5) {\small b}; \node at (0.5,1.5) {\small a}; \node at (3,1.5)
{\small c}; \node at (1.75,1.5) {\small$\dots$};
\end{tikzpicture}$,
 where $a$ and $c$ are in the same row, possibly the first
row, possibly with cells between them, and the height of the column
containing $a$ and $b$ is greater than or equal to the height of the
column containing $c.$ Then the triple $a, b, c$ is an
inversion triple of type 1 if and only if after standardization
the ordering from smallest to largest of the entries in cells $a, b,
c$ induces a counterclockwise orientation. Similarly, consider three
cells $a, b, c\in \widehat{dg}(\gamma) $  situated as follows,
$\begin{tikzpicture}[scale=0.35] \draw[line width=1pt] (0,0)
rectangle (1,1); \draw[line width=1pt] (2.5,0) rectangle (3.5,2);
 \draw[line width= 1pt] (2.5,1) -- (3.5,1);
\node at (0.5,0.5) {\small a}; \node at (3,0.5) {\small c}; \node at (3,1.5) {\small b};
\node at (1.75,0.5) {\small$\dots$};
\end{tikzpicture}$
 where $a$ and $c$ are in the same row (possibly the
basement) and the column containing $b$ and $c$ has strictly greater
height than the column containing $a.$ The triple $a, b, c$ is an
 inversion triple of type 2  if and only if after
standardization the
ordering from smallest to largest of the entries in cells $a,~ b,~
c$ induces a clockwise orientation.

Define a  semi-skyline augmented filling (SSAF)  of an
augmented diagram $\widehat{dg}(\gamma)$ to be  an augmented filling
$F$ such that every triple is an inversion triple and columns are
weakly decreasing from bottom to top. The shape of the semi-skyline
augmented filling is $\gamma$ and denoted by $sh(F).$
The picture below is an example of a semi-skyline augmented filling
with shape $(1,0,3,2,0,1)$, reading word $1321346$ and content
$(2,1,2,1,0,1)$,
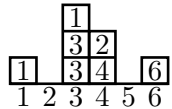
\begin{center}\begin{tikzpicture}[scale=0.35]
\draw[line width=1pt] (2,0) rectangle (3,3);\draw[line width=1pt]
(0,0) rectangle (1,1);
 \draw[line width=1pt]
(3,0) rectangle (4,2); \draw[line width=1pt] (5,0) rectangle (6,1);
 \draw[line width= 1pt] (0,0) -- (6,0); \draw[line width= 1pt]
(2,1) -- (4,1);\draw[line width= 1pt] (2,2) -- (3,2);
 \node at (0.5,-0.5) {1}; \node at (1.5, -0.5) {2}; \node at (2.5, -0.5) {3};
 \node at (3.5, -0.5) {4}; \node at (4.5, -0.5) {5}; \node at (5.5, -0.5) {6};\node at (0.5, 0.5) {1};
 \node at
(2.5,0.5){3}; \node at (2.5,1.5){3};\node at (2.5,2.5){1};\node at
(3.5,0.5){4};\node at (3.5,1.5){2};\node at (5.5,0.5){6};
\end{tikzpicture}
\end{center}
 The entry of a cell
in the first row of a SSAF is equal to the basement element where it
sits and, thus, in the first row the cell entries strictly increase from left
to the right. For any   weak composition $\gamma$ in $\mathbb{N}^n,$
there is a unique  SSAF,  with shape and content $\gamma,$ by putting
$\gamma_i$ cells with entries $i$ in the top of the basement element
$i.$  We call it key SSAF of shape $\gamma$.
The following is the key SSAF of shape  $(1,1,3,2,0,1)$,
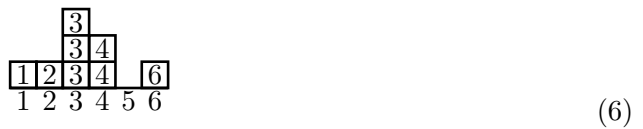
\begin{equation}\label{existence-SSAF}
\begin{tikzpicture}[scale=0.35]
\draw[line width=1pt] (1,0) rectangle (2,1);
\draw[line width=1pt] (2,0) rectangle (3,3);\draw[line width=1pt]
(0,0) rectangle (1,1);
 \draw[line width=1pt]
(3,0) rectangle (4,2); \draw[line width=1pt] (5,0) rectangle (6,1);
 \draw[line width= 1pt] (0,0) -- (6,0); \draw[line width= 1pt]
(2,1) -- (4,1);\draw[line width= 1pt] (2,2) -- (3,2);
 \node at (0.5,-0.5) {1}; \node at (1.5, -0.5) {2}; \node at (2.5, -0.5) {3};
 \node at (3.5, -0.5) {4}; \node at (4.5, -0.5) {5}; \node at (5.5, -0.5) {6};\node at (0.5, 0.5) {1};
 \node at
(2.5,0.5){3}; \node at (2.5,1.5){3};\node at (2.5,2.5){3};\node at
(3.5,0.5){4};\node at (3.5,1.5){4};\node at (5.5,0.5){6};\node at (1.5,0.5){2};
\end{tikzpicture}
\end{equation}
In \cite{masonrsk} a sequence of  lemmas provides several conditions on
triples of cells in a SSAF. In particular, we recall Lemma~2.6 in \cite{masonrsk} which characterises completely the relative values of the entries in the cells of a type 2 inversion triple in a SSAF. This property of type 2 inversion triples will be used in the proof of our
main theorem. Given a cell $a$  in SSAF $F$ define $F(a)$ to be the
entry in $a$.
\begin{lem}\label{triples} {\em \cite{masonrsk}} If $a,b,c$
is a type $2$ inversion triple in $F$, as defined above, then  $F(a)< F(b)\leq F(c)$.
\end{lem}

\subsection{An analogue of Schensted insertion and  RSK for SSAFs}
 The fundamental
operation of the Robinson--Schensted--Knuth  (RSK) algorithm is
Schensted insertion which is a procedure for inserting a positive
integer $k$ into a semi-standard Young tableau $T$. In \cite{masonrsk}, Mason
defines a similar procedure for inserting a positive integer $k$
into a SSAF $F$,
which is used  to describe an analogue of the RSK algorithm. If $F$
is a SSAF of shape
 $\gamma$, we set
 $F := (F(j)),$ where $F(j)$ is the entry in the $j^{th}$ cell in
reading order, with the cells in the basement included, and $j$ goes
from $1$ to $n + \sum_{i=1}^{n}\gamma_i.$  If $\hat{j}$ is the cell
immediately above $j$ and the cell is empty, set $F(\hat{j} ) := 0.$
The operation $k \rightarrow F,$ for $k \leq n$, is defined as
follows.

\noindent{\bf Procedure. The insertion $k \rightarrow F$}:

1. Set $i := 1,$ set $x_1 := k,$  and set $j :=
1.$

2. If $F(j) < x_i$ or $F(\hat{j})\geq x_i,$ then increase $j$ by 1
and repeat this step. Otherwise, set $x_{i+1} := F(\hat{j} )$ and
set $F(\hat{j}) := x_i.$

3. If $x_{i+1} \neq 0$ then increase $i$ by 1, increase $j$ by 1,
and repeat step 2. Otherwise, terminate the algorithm.


  The procedure terminates in
finitely many steps and the result is a SSAF.
\begin{ex} Insertion $3$  to the SSAF.
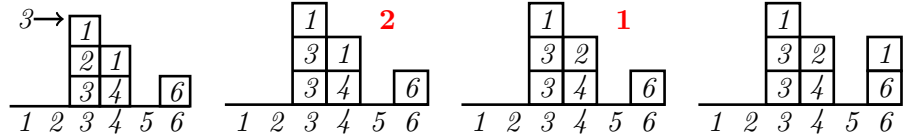
\begin{center}$\begin{tikzpicture}[scale=0.4]
\draw[line width=1pt] (2,0) rectangle (3,3);
 \draw[line width=1pt]
(3,0) rectangle (4,2); \draw[line width=1pt] (5,0) rectangle (6,1);
 \draw[line width= 1pt] (0,0) -- (6,0); \draw[line width= 1pt]
(2,1) -- (4,1);\draw[line width= 1pt] (2,2) -- (3,2);
\draw[arrows=->,line width=1 pt] (0.8,2.9)--(1.8,2.9); \node at
(0.5,-0.5) {1}; \node at (1.5, -0.5) {2}; \node at (2.5, -0.5) {3};
 \node at (3.5, -0.5) {4}; \node at (4.5, -0.5) {5}; \node at (5.5, -0.5)
{6};
 \node at
(2.5,0.5){3}; \node at (2.5,1.5){2};\node at (2.5,2.5){1};\node at
(3.5,0.5){4};\node at (3.5,1.5){1};\node at (5.5,0.5){6};
\node at (0.5,2.9){3};
\end{tikzpicture}$
$~~\begin{tikzpicture}[scale=0.45]
\draw[line width=1pt] (2,0) rectangle (3,3);
 \draw[line width=1pt]
(3,0) rectangle (4,2); \draw[line width=1pt] (5,0) rectangle (6,1);
 \draw[line width= 1pt] (0,0) -- (6,0); \draw[line width= 1pt]
(2,1) -- (4,1);\draw[line width= 1pt] (2,2) -- (3,2); \node at
(0.5,-0.5) {1}; \node at (1.5, -0.5) {2}; \node at (2.5, -0.5) {3};
 \node at (3.5, -0.5) {4}; \node at (4.5, -0.5) {5}; \node at (5.5, -0.5)
{6};
 \node at
(2.5,0.5){3}; \node at (2.5,1.5){3};\node at (2.5,2.5){1};\node at
(3.5,0.5){4};\node at (3.5,1.5){1};\node at (5.5,0.5){6};\node at
(4.8,2.5){\bf{\color{red}2}};
\end{tikzpicture}$
$~~\begin{tikzpicture}[scale=0.45]
\draw[line width=1pt] (2,0) rectangle (3,3);
 \draw[line width=1pt]
(3,0) rectangle (4,2); \draw[line width=1pt] (5,0) rectangle (6,1);
 \draw[line width= 1pt] (0,0) -- (6,0); \draw[line width= 1pt]
(2,1) -- (4,1);\draw[line width= 1pt] (2,2) -- (3,2); \node at
(0.5,-0.5) {1}; \node at (1.5, -0.5) {2}; \node at (2.5, -0.5) {3};
 \node at (3.5, -0.5) {4}; \node at (4.5, -0.5) {5}; \node at (5.5, -0.5)
{6};
 \node at(2.5,0.5){3}; \node at (2.5,1.5){3};\node at (2.5,2.5){1};\node at
(3.5,0.5){4};\node at (3.5,1.5){2};\node at (5.5,0.5){6};\node at
(4.8,2.5){{\bf\color{red}1}};
\end{tikzpicture}$
$~~\begin{tikzpicture}[scale=0.45]
\draw[line width=1pt] (2,0) rectangle (3,3);
 \draw[line width=1pt]
(3,0) rectangle (4,2); \draw[line width=1pt] (5,0) rectangle
(6,1);\draw[line width=1pt] (5,1) rectangle (6,2);
 \draw[line width= 1pt] (0,0) -- (6,0); \draw[line width= 1pt]
(2,1) -- (4,1);\draw[line width= 1pt] (2,2) -- (3,2); \node at
(0.5,-0.5) {1}; \node at (1.5, -0.5) {2}; \node at (2.5, -0.5) {3};
 \node at (3.5, -0.5) {4}; \node at (4.5, -0.5) {5}; \node at (5.5, -0.5)
{6};
 \node at
(2.5,0.5){3}; \node at (2.5,1.5){3};\node at (2.5,2.5){1};\node at
(3.5,0.5){4};\node at (3.5,1.5){2};\node at (5.5,0.5){6};\node at
(5.5,1.5){1};
\end{tikzpicture}$\end{center}
\end{ex}
Let {SSAF}$_n$ be the set of all semi-skyline augmented fillings with basement $[n]$. Based on this
Schensted insertion analogue, Mason gives a weight preserving and
shape rearranging bijection $\Psi$ between SSYT$_n$ and SSAF$_n$.
The bijection $\Psi$ is defined to be  the
insertion, from right to left, of the   column word  of a SSYT into the empty SSAF with basement $1,\ldots,n$.
 The shape of $\Psi(T)$ provides the  right key  of $T$,
$K_+(T)$,  a notion  due to Lascoux and Sch\"utzenberger \cite{lasschutz,lascouxkeys}. There are now several ways to describe the right key of
a tableau  \cite{lascouxkeys, fulton, lenart,masondemazure,willis}.
 For our purpose we consider the following Mason's theorem as the definition of right key  of $T$.
\begin{thm}[\sc Mason \cite{masondemazure}]
\label{SAFSSYT}
Given an arbitrary SSYT $T$, let $\gamma$ be the shape of $\Psi(T).$
Then $K_+(T)=key(\gamma).$
\end{thm}
 Given the partition $\lambda\in \mathbb{N}^n$, let $\mathfrak{B}^\lambda$ denote the set of all semi-standard Young tableaux in  {SSYT}$_n$ of shape $\lambda$.
 This theorem decompose   $\mathfrak{B}^\lambda$ into a disjoint union of semi-standard Young tableaux w.r.t. to their right keys:
$$\mathfrak{B}^\lambda =\biguplus_{\begin{smallmatrix}\gamma\in \mathfrak{S}_n\lambda
\end{smallmatrix}}\{T\in {SSYT}_n:K_+(T)=key(\gamma)\}.$$
\begin{ex}\label{ex9} One has $sh(\Psi(T))=(2,0,4,3,1)$,
\begin{center}$\begin{tikzpicture}[scale=0.4]
\draw[line width=1pt] (0,0) rectangle (1,4);
 \draw[line width=1pt]
(1,0) rectangle (2,3); \draw[line width=1pt] (2,0) rectangle
(3,2);\draw[line width=1pt] (3,0) rectangle (4,1);
 \draw[line width= 1pt] (0,1) -- (4,1); \draw[line width= 1pt]
(0,2) -- (3,2);\draw[line width= 1pt] (0,3) -- (1,3); \node at
(0.5,0.5) {1}; \node at (0.5, 1.5) {2}; \node at (0.5, 2.5) {3};
 \node at (0.5, 3.5) {5}; \node at (1.5, 0.5) {1}; \node at (1.5, 1.5) {3};
 \node at(1.5,2.5){4}; \node at (2.5,0.5){1};\node at (2.5,1.5){4};\node at
(3.5,0.5){3}; \node at (2,-1.5){$T$}; \draw[arrows=->,line width=1
pt] (5,0)--(8,0); \draw[line width=1pt] (9,0) rectangle (10,2);
 \draw[line width=1pt]
(11,0) rectangle (12,4); \draw[line width=1pt] (12,0) rectangle
(13,3);\draw[line width=1pt] (13,0) rectangle (14,1); \draw[line
width= 1pt] (9,0) -- (14,0);
 \draw[line width= 1pt] (9,1) -- (10,1); \draw[line width= 1pt]
(11,1) -- (13,1);\draw[line width= 1pt] (11,2) -- (13,2);\draw[line
width= 1pt] (11,3) -- (12,3);
 \node at (9.5,-0.5) {1}; \node at (10.5,
-0.5) {2}; \node at (11.5, -0.5) {3};
 \node at (12.5, -0.5) {4}; \node at (13.5, -0.5) {5};
 \node at (9.5, 0.5) {1};
 \node at(9.5,1.5){1}; \node at (11.5,0.5){3};\node at (11.5,1.5){3};\node at
(11.5,2.5){3};\node at (11.5,3.5){1};\node at (12.5,0.5){4};\node at
(12.5,1.5){4};\node at (12.5,2.5){2};\node at (13.5,0.5){5}; \node
at (11.5,-1.5){$\Psi(T)$};\node at (6.5,0.5){$\Psi$};\node at (15,0){$,$};
\end{tikzpicture}$$~~~~
\begin{tikzpicture}[scale=0.45]
 \draw[line
width=1pt] (11,-7) rectangle (12,-3); \draw[line width=1pt] (12,-7)
rectangle (13,-4); \draw[line width=1pt] (13,-7) rectangle
(14,-5);\draw[line width=1pt] (14,-7) rectangle (15,-6);
 \draw[line width= 1pt] (11,-6) -- (15,-6); \draw[line width= 1pt]
(11,-5) -- (14,-5);\draw[line width= 1pt] (11,-4) -- (12,-4); \node
at (11.5,-6.5) {1}; \node at (11.5, -5.5) {3}; \node at (11.5, -4.5)
{4};
 \node at (11.5, -3.5) {5}; \node at (12.5, -6.5) {1}; \node at (12.5, -5.5) {3};
\node at(12.5,-4.5){4}; \node at (13.5,-6.5){3};\node at
(13.5,-5.5){4};\node at (14.5,-6.5){3};
\node at (13,-8.5){$K_+(T)=key(2,0,4,3,1)$};
\end{tikzpicture}$
\end{center}
\end{ex}
Given the alphabet $[n]$, the RSK algorithm is a bijection between
biwords in lexicographic order and pairs of SSYT of the same shape
over $[n]$. Equipped with the Schensted insertion analogue,  Mason
 finds in \cite{masonrsk} an analogue $\Phi$ of the RSK yielding a pair
of SSAFs  with shapes a rearranging of each other. This bijection has an advantage over the classical RSK because
the pair of SSAFs  comes along with the extra pair of right keys.

The two line array
$w=\left(\begin{array}{cccc}i_1&i_2&\cdots&i_l\\j_1&j_2&\cdots&j_l
\end{array}\right),~~i_{r}< i_{r+1},~~$ or $~~i_r=i_{r+1} ~~\&~~ j_{r} \le j_{r+1},~$ $1\le i,j\le l-1$,
with $i_r, ~j_r\in [n]$, is called a biword in lexicographic order
over the alphabet $[n]$. The map $\Phi$ defines a bijection between
the set $\mathbb{A}_n$ of all  biwords $w$ in lexicographic order in
the alphabet $[n]$, and pairs  of SSAFs with shapes in the same $\mathfrak{S}_n$-orbit,
and the
contents are  those of the second and first rows of $w$, respectively.

\noindent{\bf Procedure. The map $\Phi:\mathbb{A}_n\longrightarrow
{SSAF}_n \times {SSAF}_n$}. Let $w~\in \mathbb{A}_n$.

 1. Set $r := l,$ where $l$ is the number of biletters in
$w$. Let $F =\emptyset= G,$ where $\emptyset$ is the empty $SSAF.$

2. Set $F := (j_r\rightarrow F).$ Let $h_r$ be the height of the
column in $(j_r\rightarrow F)$ at which the insertion procedure
$(j_r \rightarrow F)$ terminates.

3. Place $i_r$ on top of the leftmost column of height $h_r - 1$ in
$G$ such that doing so preserves the decreasing property of columns
from bottom to top. Set $G$ equal to the resulting figure.

4. If $r - 1 \neq 0,$ repeat step 2 for $r := r - 1.$ Else terminate
the algorithm.

\begin{re} \label{insertion}
$1.$ The entries in the top row of the biword are weakly increasing
when read from left to right. Henceforth, if $h_r > 1,$ placing
$i_r$ on top of the leftmost column of height $h_r -1$ in $G$
preserves the decreasing property of columns. If $h_r = 1,$ the
$i^{th}_ r$ column of $G$ does not contain an entry from a previous
step. It means that the number $i_r$ sits on the top of basement $i_r.$

\noindent $2.$ Let $h$ be the height of the column in $F$ at which
the insertion procedure
$(j \rightarrow F)$ terminates.
Lemma~\ref{triples} implies that there is no  column of height
$h+1$ in $ F$ to the right.

\noindent $3.$ Steps $2$ and $3$ guarantee that, at each stage of  the map $\Phi$ procedure, the shapes of the pair of SSAFs are a rearrangement of each other.
\end{re}
\begin{cor}[\sc Mason \cite{masondemazure,masonrsk}]
\label{colSAFSSYT}
The RSK algorithm commutes with the above analogue $\Phi.$ That is,
if $(P,Q)$ is the pair of SSYTs produced by RSK algorithm applied to
biword $w,$ then $(\Psi(P),\Psi(Q))=\Phi(w),$ and
$K_+(P)=key(sh(\Psi(P)))$, $K_+(Q)=key(sh(\Psi(Q)))$.
\end{cor}
This result is summarised in the following scheme from which, in
particular, it is clear   the RSK analogue $\Phi$ also shares the
symmetry of RSK,
$$
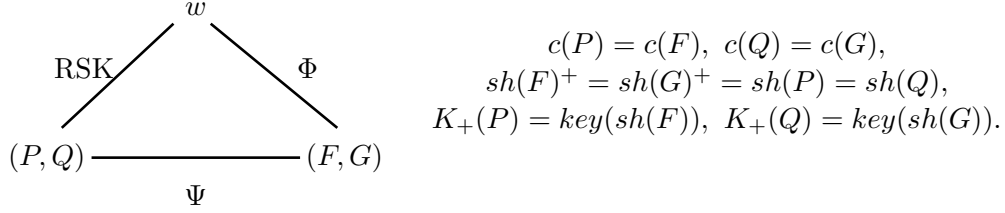
\begin{tikzpicture} [scale=0.99]\draw [line width=1pt] (0.2,0.4)--(1.7,1.8);
\draw [line width=1pt] (2.2,1.8)--(3.9,0.4);\draw [line width=1pt]
(0.6,0)--(3.4,0); \node at (0,0){$(P,Q)$};\node at (2,2){$w$};\node
at (4,0){$(F,G)$}; \node at (0.5,1.2){RSK};\node at
(3.5,1.2){$\Phi$};\node at (2,-0.5){$\Psi$}; \node at
(9,1){$sh(F)^+=sh(G)^+=sh(P)=sh(Q)$,}; \node at
(9,0.5){$K_+(P)=key(sh(F)),~K_+(Q)=key(sh(G)).$}; \node at
(9,1.5){$c(P)=c(F),~c(Q)=c(G),$};
\end{tikzpicture}$$

\section{Main Theorem}
\label{sec:main}
 We prove a restriction of the bijection $\Phi$ to  multisets
 of cells in a  staircase or truncated  staircase  of length $n$, such that the staircases of length $n-k$   on the upper left corner, or  of length $n-m$ on the bottom right
 corner,  with $1\le m\le n$, $1\le k\le n$ and $k+m\ge n+1$, are erased.
 The restriction to be imposed on the pairs of SSAFs is that the pair of shapes  in a same $\mathfrak{S}_n$-orbit,
satisfy an inequality in the Bruhat
order,
 where one shape is bounded by the reverse of the other. Equivalently, pairs of SSYTs  whose right keys are such that  one is bounded by the evacuation of the other.
The following lemma gives sufficient conditions to preserve the Bruhat order relation between two weak compositions when one box is added to a column of their diagrams.
\begin{lem}\label{propvector}
Let $\alpha=(\alpha_1,\alpha_2,\ldots,\alpha_n)$ and
$\beta=(\beta_1,\beta_2,\ldots,\beta_n)$  be in the same $
  \mathfrak{S}_{n}$-orbit,   with
$key(\beta)\leq key(\alpha).$ Given  $k\in \{1,\ldots,n\}$, let
$k^{\prime}\in \{1,\ldots,n\}$ be such that $\beta_{k^{\prime}}$ is
the left most  entry of $\beta$ satisfying
$\beta_{k^{\prime}}=\alpha_k.$ Then if
$\tilde{\alpha}=(\alpha_1,\alpha_2,\ldots,\alpha_k+1,\ldots,\alpha_n)$
and
$\tilde{\beta}=(\beta_1,\beta_2,\ldots,\beta_{k^{\prime}}+1,\ldots,\beta_n),$
it holds  $key(\tilde{\beta})\leq key(\tilde{\alpha}).$
\end{lem}
\begin{proof} Let $k,~ k^{\prime}\in \{1,\dots,n\}$ as in the lemma, and put $\alpha_k=\beta_{k^{\prime}}=m\ge 1.$
(The proof for $m=0$ is left to the reader. The case of interest for
our problem is $m>0$ which is related with the procedure of map
$\Phi$.) This means that $k$ appears exactly in the  first $m$
columns of $key(\alpha)$, and   $k^{\prime}$ is the smallest number
that does not appear in column $m+1$ of $key(\beta)$ but appears
exactly  in the
 first $m$ columns. Let $t$ be the row index of the cell with entry $k'$ in column $m$ of $key(\beta)$. Every entry less than $k^{\prime}$ in
  column $m$ of $key(\beta)$ appears in
column $m+1$ as well,   and since in a key tableau each column is
contained in the previous one, this implies that the first $t$ rows of
columns
 $m$ and $m+1$ of $key(\tilde{\beta})$ are equal.
 The only difference between $key(\tilde\beta)$ and $key(\beta)$ is in columns $m+1$, from row $t$ to the top.  Similarly if $z$ is the row
 index of the cell with entry $k$ in column $m+1$ of $key(\tilde\alpha)$, the only difference between $key(\tilde\alpha)$ and  $key(\alpha)$ is
 in columns $m+1$ from row $z$ to the top.
To obtain column $m+1$ of $key(\tilde\beta)$, shift in the column
$m+1$ of $key(\beta)$ all the cells with entries $>k'$ one row up,
and add to the  position left vacant (of row index $t$) a new cell
with entry $k'$. The column $m+1$ of $key(\tilde\alpha)$ is obtained
similarly, by shifting one row up in the  column $m+1$ of
$key(\alpha)$ all the cells with entries $>k$ and  adding a new cell
with entry $k$ in the vacant position.

 Put $p:=\min\{t,z\}$ and $q:=\max\{t,z\}.$ We divide the columns $m+1$
in each pair of tableaux  $key({\beta}),~key(\tilde{\beta})$ and
$key({\alpha}),~key(\tilde{\alpha})$ into three parts: the first,
from row one to row $p-1$; the second, from row $p$ to row $q$; and
the third, from  row $q+1$ to the top row. The first parts of column
$m+1$ of $key(\tilde{\beta})$ and $key(\beta)$ are the same,
equivalently, for  $key(\tilde\alpha)$ and  $key(\alpha).$ The third
part of column $m+1$ of $key(\tilde{\beta})$ consists of  row $q$
plus    the third part of $key(\beta)$, equivalently, for
$key(\tilde\alpha)$ and  $key(\alpha).$ As  the columns $m+1$ of
$key(\beta)$ and $key(\alpha)$ are entrywise comparable, the same
happens to the first and third parts of columns $m+1$ in $key(\tilde{\beta})$
and $key(\tilde{\alpha})$. It remains to analyse the second parts of
the pair $key(\tilde{\beta}),~key(\tilde{\alpha})$ which we split
into two cases according to the relative magnitude of $p$ and $q$.

{\em Case} $1$. $p=t<q=z$. Let $k'<b_t<\cdots< b_{z-1}$ and
$d_t<\cdots< d_{z-1}<k$ be  the entries of the
second parts of columns $m+1$ in   $key(\tilde{\beta})$ and
$key(\tilde{\alpha})$, respectively. By construction $k'<b_t\le d_t<d_{t+1}$,
$b_i<b_{i+1}\le d_{i+1}$, $t<i<z-2$, and $b_{z-1}\le d_{z-1}<k$,
and, therefore, the second parts are also comparable.

{\em Case} $2$. $p=z\le q=t$. In this case,
 the assumption on $k'$  implies that the first $q$ rows of columns $m$ and $m+1$ of $key(\tilde{\beta})$ are equal. On the other hand,
 since column $m$ of $key({\beta})$ is less or equal than column  $m$ of  $key({\alpha})$, which is equal to the column $m$ of $key(\tilde{\alpha})$
  and in turn is less or equal to column $m+1$ of $key(\tilde{\alpha})$, forces by transitivity that the second part of column $m+1$ of $key(\tilde{\beta})$
   is less or equal than the corresponding part of  $key(\tilde{\alpha})$.
\end{proof}
We illustrate the lemma with
\begin{ex} Let $\beta=(3,2^2,1,0^2,1)$,
$\alpha=(2,0,3,0,1,2,1)$, $\tilde\beta=(3,2^3,0^2,1)$, and
$\tilde\alpha=(2,0,3,0,2^2,1)$,
\begin{center}\begin{tikzpicture}[scale=0.45]
  \footnotesize
\draw[line width=1pt] (-4,0) rectangle (-3,5);
 \draw[line width= 1pt] (-3,0)rectangle (-2,3);
  \draw[line width= 1pt] (-2,0)rectangle (-1,1);

\draw[line width= 1pt] (3,0)rectangle (4,5);
 \draw[line width= 1pt] (4,0)rectangle(5,3);
\draw[line width= 1pt] (5,0)rectangle(6,1);

\draw[line width=1pt] (-4,1) -- (-1,1); \draw[line width=1pt] (-4,2)
-- (-2,2); \draw[line width=1pt] (-4,3) -- (-2,3); \draw[line
width=1pt] (-4,4) -- (-3,4);

\draw[line width=1pt] (3,1) -- (6,1); \draw[line width=1pt] (3,2) --
(5,2); \draw[line width=1pt] (3,3) -- (5,3); \draw[line width=1pt]
(3,4) -- (4,4);

  \footnotesize
\node at (-5.8, 3) {$key(\beta)=$}; \node at (-1,3) {$\leq$}; \node
at (1.2, 3) {$key(\alpha)=$};

\node at (-3.5, 0.5) {1};\node at (-3.5, 1.5) {2}; \node at (-3.5,
2.5) {3}; \node at (-3.5, 3.5) {4};\node at (-3.5, 4.5) {7};

\node at (-2.5, 0.5) {1};\node at (-2.5, 1.5) {2}; \node at (-2.5,
2.5) {3};

\node at (-1.5, 0.5) {1};

\node at (3.5, 0.5) {1};\node at (3.5, 1.5) {3}; \node at (3.5, 2.5)
{5}; \node at (3.5, 3.5) {6};\node at (3.5, 4.5) {7};

\node at (4.5, 0.5) {1};\node at (4.5, 1.5) {3}; \node at (4.5, 2.5)
{6};

\node at (5.5, 0.5) {3};

\draw[line width=1pt] (10,0) rectangle (11,5);
 \draw[line width= 1pt] (11,0)rectangle (12,4);
  \draw[line width= 1pt] (12,0)rectangle (13,1);

\draw[line width= 1pt] (17,0)rectangle (18,5);
 \draw[line width= 1pt] (18,0)rectangle(19,4);
\draw[line width= 1pt] (19,0)rectangle(20,1);

\draw[line width=1pt] (10,1) -- (13,1); \draw[line width=1pt] (10,2)
-- (12,2); \draw[line width=1pt] (10,3) -- (12,3); \draw[line
width=1pt] (10,4) -- (11,4);

\draw[line width=1pt] (17,1) -- (20,1); \draw[line width=1pt] (17,2)
-- (19,2); \draw[line width=1pt] (17,3) -- (19,3); \draw[line
width=1pt] (17,4) -- (18,4);

  \footnotesize
\node at (8.2, 3) {$key(\tilde\beta)=$}; \node at (13,3) {$\leq$};
\node at (15.2, 3) {$key(\tilde\alpha)=$};

\node at (10.5, 0.5) {1};\node at (10.5, 1.5) {2}; \node at (10.5,
2.5) {3}; \node at (10.5, 3.5) {4};\node at (10.5, 4.5) {7};

\node at (11.5, 0.5) {1};\node at (11.5, 1.5) {2}; \node at (11.5,
2.5) {3}; \node[color=red] at (11.5, 3.5) {4};

\node at (12.5, 0.5) {1};

\node at (17.5, 0.5) {1};\node at (17.5, 1.5) {3}; \node at (17.5,
2.5) {5}; \node at (17.5, 3.5) {6};\node at (17.5, 4.5) {7};

\node at (18.5, 0.5) {1};\node at (18.5, 1.5) {3}; \node[color=red]
at (18.5, 2.5) {5}; \node at (18.5, 3.5) {6};

\node at (19.5, 0.5) {3};

\end{tikzpicture}.\end{center}
\end{ex}
We are now ready to state and prove the main theorem.
\begin{thm}\label{maint}
Let $w$ be a biword in lexicographic order in the alphabet $[n]$,
and let $\Phi(w)=(F,G).$  For each biletter $\displaystyle{
\binom i j}$ in $w$ one has $i+j\leq n+1$ if and only
if $key(sh(G))\leq key( \omega sh(F)),$ where $\omega$ is the
longest permutation of $\mathfrak{S}_n$.
  Moreover, if the first
row of $w$ is a word in the alphabet $[k],$ with $1\leq k \leq n,$ and  the second row is a word in the alphabet $[m],$ with $1\leq m \leq n,$
the shape of $G$  has the last $n-k$
entries equal to zero, and the shape of $F$ the last $n-m$ entries equal to zero.
\end{thm}
\begin{proof}  We  describe
$\Phi(w)=(F,G)$ in terms of a Bruhat relation between the shapes of $F$ and $G$, when the billeters are cells in  a staircase of size $n$.

{\em Only if part.}
 We prove, by induction on the number of biletters of $w$, that if  $\Phi(w)=(F,G)$, where the billeters are cells in  a staircase of size $n$, then $sh(G)\le \omega sh(F)$. If $w$ is the empty word then $F$ and $G$ are the
empty semi-skyline augmented filling,  the shapes are null vectors, and there is nothing to prove. Let
$w^{\prime}=\left(\begin{array}{c cc c }
 i_{p+1}&i_p&\cdots &i_1\\j_{p+1} &j_p&\cdots &j_1\\
 \end{array}\right)$ be a biword in lexicographic order such that $p \geq 0$ and $i_t+j_t\leq n+1$ for all $1\leq t\leq p+1,$ and $w=\left(\begin{tabular}{c c c }
 $i_p\cdots i_1$\\ $j_p\cdots j_1$\\ \end{tabular}\right)$ such that
 $\Phi(w)=(F,G).$ Let $F^{\prime}:=(j_{p+1}\rightarrow F)$ and $h$ the
height of the column in $F^{\prime}$ at which the insertion
procedure terminates. There are two possibilities  for $h$ which the
third step of the algorithm procedure of $\Phi$ requires to
consider.

$\bullet~h=1.$ It means $j_{p+1}$ is sited on the top of the
basement element $j_{p+1}$ in $F$ and therefore $i_{p+1}$ goes to
the top of the basement element  $i_{p+1}$ in $G$. Let $G^{\prime}$
be the semi-skyline augmented filling obtained after placing $i_{p+1}$ in $G.$  See Figure~\ref{h=1}.
\begin{figure}[h!]
$$\begin{tikzpicture}[scale=0.6]
  \footnotesize
\draw[line width=1pt] (3,0) rectangle (4,1); \draw [line width=1pt]
(0,0)--(7,0);

 \draw[line width=1pt]
(10,0)
rectangle (11,1);
\draw[line width=1pt] (8,0) -- (15,0);
 \node at (0.5,-0.5){1};\node at (1.5,-0.5){\dots};\node at
(2.5,-0.5){\dots};\node at (3.5,-0.5){ $j_{p+1}$};\node at (4.5,-0.5){\dots};\node at
(5.5,-0.5){\dots};\node at (6.5,-0.5){n};
\node at (8.5,-0.5){1};\node at (9.5,-0.5){\dots};\node at
(10.5,-0.5){ $i_{p+1}$};\node at (11.5,-0.5){\dots};\node at (12.5,-0.5){\dots};\node
at (13.5,-0.5){\dots};\node at (14.5,-0.5){n};
\node at (3.5,0.5){\tiny $j_{p+1}$};\node at (5.5,0.5){ $\cdots$}; ;\node at (2,0.5){ $\cdots$};\node at (10.5,0.5){\tiny $i_{p+1}$};\node at (12,0.5){ $\cdots$}; \node at (13.5,0.5){ $\cdots$};
\node at (3.5,-1.5){$F'$};\node at (11.5,-1.5){$G'$};
\end{tikzpicture}$$
\caption{The pair $(F',G')$ of SSAFs exhibiting the  columns of height one, with  basements $j_{p+1}$ and $i_{p+1}$, where $i_{p+1}\le n-j_{p+1}+1$. $G'$ is empty to the left of $i_{p+1}$.}
\label{h=1}
\end{figure}
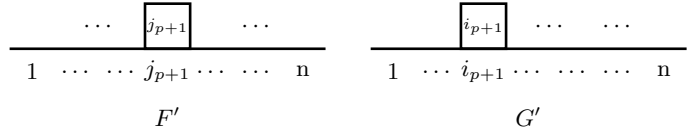

As
$i_{p+1}\leq i_t,$ for all $t,$   $i_{p+1}$ is the bottom entry of
the first column in $key( sh(G^{\prime}))$  whose   remaining entries
constitute the first column of $key(sh(G))$.
Suppose $n+1-j_{p+1}$ is
added to the row $z$ of the first column in $key(\omega sh(F))$ by
shifting all the entries above it
 one row up. Let $i_{p+1}<a_1<\cdots<a_z<a_{z+1}<\cdots <a_l$ be the entries in the first column of  $key
(sh(G^{\prime}))$ and
$b_1<b_2<\cdots< b_{z-1}<n+1-j_{p+1}<b_z<\cdots < b_l$ be  the entries in the first column
 of   $key(\omega sh(F^{\prime}))$,  where
$a_1<\cdots<a_z<\cdots <a_l$ and $b_1<\cdots<b_z<\cdots < b_l$ are
 the entries in the corresponding first columns of  $key(
sh(G))$  and $key(\omega sh(F)).$ If $z=1,$ as $i_{p+1}\leq
n+1-j_{p+1}$ and $a_i\leq b_i$ for all $1\leq i\leq l,$ then $key
(sh(G^{\prime}))\leq key(\omega sh(F^{\prime})).$ If $z>1,$ as
$i_{p+1}<a_1\leq b_1<b_2,$ we have $i_{p+1}\leq b_1$ and $a_1\leq
b_2.$ Similarly $a_i\leq b_i<b_{i+1},$ and $a_i\leq b_{i+1},$ for
all $2\leq i\leq z-2.$ Moreover $a_{z-1}\leq b_{z-1}< n+1-j_{p+1},$
therefore $a_{z-1}\leq n+1-j_{p+1}.$ Also $a_i\leq b_i$ for all
$z\leq i\leq l.$ Hence $key (sh(G^{\prime}))\leq key(\omega
sh(F^{\prime})).$

 $\bullet~h>1.$
 Place $i_{p+1}$ on the top of
the leftmost column of height $h-1$. This means, by
Lemma~\ref{propvector}, $key( sh(G^{\prime}))\leq key( \omega
sh(F^{\prime})).$ See Figure~\ref{fig:h=1}.
\begin{figure}[here]
$$\begin{tikzpicture}[scale=0.56]
  \footnotesize
\draw[line width=1pt] (4.5,0) rectangle (5.5,1); \draw [line width=1pt]
(0,0)--(7,0);
\draw[line width=1pt] (4.5,1) rectangle (5.5,2);
\draw[line width=1pt] (4.5,2) rectangle (5.5,3);
\draw[line width=1pt] (4.5,3) rectangle (5.5,4);

\draw[line width=1pt] (1.5,0) rectangle (2.5,1);
\draw[line width=1pt] (1.5,1) rectangle (2.5,2);

 \draw[line width=1pt](10,0)rectangle (11,1);
  \draw[line width=1pt](10,1)rectangle (11,2);
    \draw[line width=1pt](12,3)rectangle (13,4);

     \draw[line width=1pt](12,0)rectangle (13,1);
  \draw[line width=1pt](12,1)rectangle (13,2);
   \draw[line width=1pt](12,2)rectangle (13,3);
\node at (2,4){}; \node at (10.5,4){ };
\draw[line width=1pt] (8,0) -- (15,0);
 \node at (0.5,-0.5){1};\node at (1,-0.5){\dots};\node at
(2,0.5){ \tiny $j_{p+1}$};
 \node at (2,-0.5){ $j_{p+1}$};\node at (3.5,-0.5){ \dots};\node at (4.5,-0.5){\dots};\node at
(5.5,-0.5){\dots};\node at (6.5,-0.5){n};
\node at (1,1.5){\dots};\node at (3.5,1.5){\dots};\node at (6.5,1.5){\dots};
\node at (8.5,-0.5){1};\node at (9.5,-0.5){\dots};\node at
(12.5,-0.5){ \dots};\node at (11.5,-0.5){\dots};\node at (10.5,-0.5){$i_{p+1}$};\node at (10.5,0.5){\tiny$i_{p+1}$};
\node
at (13.5,-0.5){\dots};\node at (14.5,-0.5){n};
\node at (8.5,1.5){\dots};\node at (11.5,1.5){\dots};\node at (14.5,1.5){\dots};

\node at (5,3.5){\tiny $j_{p+1}$}; \node at (12.5,3.5){\tiny $i_{p+1}$};
\node at (3.5,-1.5){$F'$};\node at (11.5,-1.5){$G'$};
\draw [decorate,decoration={brace,amplitude=5pt}] (4.25,0.1) -- (4.25,3.9);
\node at (3.75,2){$h$};
\draw [decorate,decoration={brace,amplitude=5pt}] (11.9,0.1) -- (11.9,3.9);
\node at (11.4,2){$h$};
\end{tikzpicture}$$
\caption{The pair $(F',G')$ of SSAFs exhibiting the columns of height $h>1$.}
\label{fig:h=1}
\end{figure}
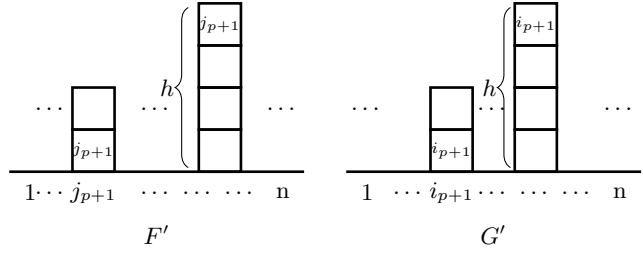


 {\em If part}.
 It is enough to show the following.
\textit{If there exists a biletter
$\binom i j$ 
 in $w$ such that $i+j> n+1,$ then at least one entry of $key(sh(G))$ is strictly bigger than the corresponding entry of $key(\omega sh(F)).$}

This means that if not all biletters satisfy $i+j\le n+1$ then either the shapes $sh(G)$ and $\omega  sh(F)$  are not comparable in the Bruhat order, or $sh(G)>\omega  sh(F)$. (Example~\ref{ex:main}.$2$. and \ref{ex:main}.$3$. show that both situations may happen.) Therefore,  $key(sh(G))\nleq \omega sh(F)$.

\textit{Before delving into the proof,  for the reader's convenience, we give an outline of it}. Let $w=\left(\begin{tabular}{c c c }
 $i_{p}\cdots i_{1}$\\ $ j_{p} \cdots j_{1}$\\ \end{tabular}\right)$, $p\ge 1$,
be a biword  in lexicographic order on the alphabet $[n]$, and
$
 \binom {i_{t}} { j_{t}}$, $t\ge 1$, 
  the  first biletter in $w$, from right to left,
  with $~i_t+j_t> n+1.$ Set $F_0=G_0:=\emptyset$, and for $d\ge 1$, let $(F_d,G_d):=\Phi\left(\begin{tabular}{c ccc}
 $i_{d} \cdots i_1$\\ $ j_{d}\cdots j_1$\\ \end{tabular}\right)$. After inserting  $ j_t$ in $F_{t-1}$ and placing $i_t$ in $G_{t-1}$ where $i_t>n+1-j_t$, it will be easily seen that the letters $i_t$ and $n+1-j_t$ appear as bottom entries in the first column of $key(sh(G_t))$ and in the first column of $key(\omega sh(F_t))$,  respectively. We call to this pair of letters  $(i_t,n+1-j_t)$ where  $i_t>n+1-j_t$ a \textit{problem} in the key-pair $(key(sh(G_{t})), key(\omega sh(F_{t})))$. Further insertions of letters $j_d$ in $F_{d-1}$ and placements of $i_d$ in $G_{d-1}$, for $d>t$,  either corresponding to cells below, or above the
 staircase of size $n$, will not solve the \textit{problem} of a pair of letters  $(i_t,n+1-j_t)$ such that  $i_t>n+1-j_t$,    in some row of homologous columns in the  key-pair $(key(sh(G_{d-1})), key(\omega sh(F_{d-1})))$. The \textit{problem} $i_t>n+1-j_t$ will always appear  in some row of a pair of homologous columns in the key-pair $(key(sh(G_{d})), key(\omega sh(F_{d})))$ for any $d\ge t$. To show this, we keep track of the \textit{problem} in a sequence of four claims with the aim  to \textit{locate} the \textit{problem} at any stage of the insertion. The \textit{locus} of  $i_t>n+1-j_t$ in a row of homologous columns in  the key-pair $(key(sh(G_{d})), key(\omega sh(F_{d})))$ with $d\ge t$, will be called the classification of the \textit{problem}.

We now embark on the details.
First  apply the map $\Phi$ to the biword  $\left(\begin{tabular}{c
c c }
 $i_{t-1}\cdots i_{1}$\\ $ j_{t-1} \cdots j_{1}$\\ \end{tabular}\right)$ to obtain the pair  $(F_{t-1},G_{t-1})$ of SSAFs
  whose right keys satisfy, by the ``only if part" of the theorem,  $key(sh(G_{t-1}))$ $\leq $ $key(\omega sh(F_{t-1})).$
 Now insert $j_t$ to $F_{t-1}.$
As $i_k+j_k\leq n+1,$ for $1\leq k\leq t-1,$ $~i_k+j_k\leq n+1 <
i_t+j_t$, and $i_t\leq i_k,~1\leq k\leq t-1,$ then $j_t> j_k,~1\leq
k\leq t-1$, and, since $w$ is in lexicographic order, this implies
$i_t<i_{t-1}$. Therefore, $j_t$ sits on the top of the basement
element  $j_t$ in $F_{t-1}$ and $i_t$ sits on the top of the
basement element $i_t$ in $G_{t-1}.$ Since the column  with basement $j_t$, in the insertion filling, and the column with basement $i_t$, in the recording filling,  play an important role in what follows they will be denoted by $J$ and $I$, respectively.  See Figure~\ref{fig:onlyif}.
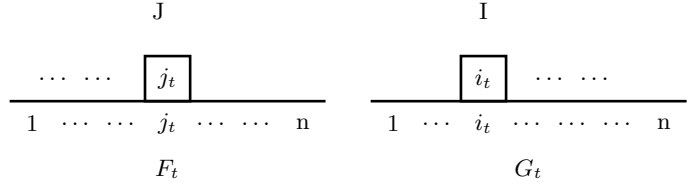
\begin{figure}[here]
\begin{center}\begin{tikzpicture}[scale=0.6]
  \footnotesize
\draw[line width=1pt] (3,0) rectangle (4,1); \draw [line width=1pt]
(0,0)--(7,0);
 \draw[line width=1pt]
(10,0)
rectangle (11,1);
\draw[line width=1pt] (8,0) -- (15,0);
 \node at (0.5,-0.5){1};\node at (1.5,-0.5){\dots};\node at
(2.5,-0.5){\dots};\node at (3.5,-0.5){ $j_{t}$};\node at (4.5,-0.5){\dots};\node at
(5.5,-0.5){\dots};\node at (6.5,-0.5){n};
\node at (8.5,-0.5){1};\node at (9.5,-0.5){\dots};\node at
(10.5,-0.5){ $i_{t}$};\node at (11.5,-0.5){\dots};\node at (12.5,-0.5){\dots};\node
at (13.5,-0.5){\dots};\node at (14.5,-0.5){n};
\node at (3.5,0.5){ $j_{t}$}; \node at (10.5,0.5){ $i_{t}$};
\node at (2,0.5){ $\dots$}; \node at (1,0.5){ $\dots$};
\node at (13,0.5){ $\dots$}; \node at (12,0.5){ $\dots$};
 \node at (3.5,-1.5){$F_t$};\node at (11.5,-1.5){$G_t$};
  \node at (10.5,2){ I};\node at (3.3,2){J};
\end{tikzpicture}\end{center}
\caption{The pair $(F_t,G_t)$ of SSAFs after inserting $j_t$ and placing $i_t$, where $i_t>n+1-j_t$, in $(F_{t-1},G_{t-1})$ . The columns $J$ and $I$, with basements $j_t$ and $i_t$, respectively, are both of height one. The SSAF $F_t$ is empty to the right of the column $J$, and $G_t$ is empty to the left of the column $I$.}
\label{fig:onlyif}
\end{figure}

It means that $n+1-j_t$ is
added to the first row and first column of $key (\omega
sh(F_{t-1}))$ and all entries in this column are shifted one row up.
Similarly, $i_t$ is added to the first row and first column of
$key(sh(G_{t-1}))$, and all the entries in this column are shifted
one row up.  As $i_{t} > n+1- j_{t}$ then  the first column of
$key(sh(G_{t}))$ and the first column of $key(\omega sh(F_{t}))$  are not
entrywise comparable, and we  have a ``{\em problem}" in
the key-pair $(key(sh(G_{t})), key(\omega sh(F_{t})))$.  See Figure~\ref{fig:beforeclaim1}.
\begin{figure}[here]
\begin{center}\begin{tikzpicture}
 \node at (0,0){${\color{red}i_t}$};\node at (2.5,0){${\color{red}n+1-j_t}$};\node at
(0,0.5){$a_1$};\node at (0,1.5){\vdots};
\node at
(2,0.5){$b_1$};\node at (2,1.5){\vdots};
\node at (1.25,0){$>$};\node at (1.25,0.5){$\leq$};\node at (1.25,1.5){$\leq$};
\end{tikzpicture}\end{center}
\caption{The bottom entry  in the first column of $key(sh(G_{t}))$ and the bottom entry  in the first column of $key(\omega sh(F_{t}))$ satisfy $i_t$ $>$ $n+1-j_t$. The key-pair $(key(sh(G_{t})), key(\omega sh(F_{t}))$ is not comparable when $t>1$.}
\label{fig:beforeclaim1}
\end{figure}

From now on, we shall see that the
``{\em problem}"
 $i_{t} > n+1- j_{t}$ remains in some row of a pair of homologous columns in the key-pair $(key(sh(G_{d})), key(\omega sh(F_{d}))),$ with $d\ge t$.
Given $d\ge t$, $J$ still will denote the column with basement $j_t$ in
$F_d$, and  $I$ the column with basement $i_t$ in $G_d$.  Their heights will be denoted by $|J|$
and $|I|$, respectively, and they are $\ge 1$.  For each $i\ge 1$, let
$r_i$  and $k_i$ denote the number of columns of height $\ge i$,  to the right of $J$ and to the left of $I$, respectively.  See Figure~\ref{fig:classproblem}. The fillings (not their basements) of the columns $I$ and $J$ as their heights depend indeed on $d$ but we shall avoid cumbersome notation as long as there is no danger of confusion. However, we put the superscript $d$ and $d+1$ on  $k_i$ and $r_i$ to distinguish between $(F_d,G_d)$ and $(F_{d+1},G_{d+1})$, whenever clarity of presentation makes this necessary.
\begin{figure}[h!]
\begin{center}\begin{tikzpicture}[scale=0.45]
  \footnotesize
\draw[line width=1pt] (3,0) rectangle (4,1);
\draw[line width=1pt] (3,1) rectangle (4,2);\draw[line width=1pt] (3,3) rectangle (4,4);
\node at (3.5,2.75){$\vdots$};
\draw [line width=1pt]
(0,0)--(7,0);
 \draw[line width=1pt](10,0)rectangle (11,1);
 \draw[line width=1pt](10,1)rectangle (11,2);\draw[line width=1pt](10,3)rectangle (11,4);
\node at (10.5,2.75){$\vdots$};
\draw[line width=1pt] (8,0) -- (15,0);
 \node at (0.5,-0.5){1};\node at (1.5,-0.5){\dots};\node at
(2.5,-0.5){\dots};\node at (3.5,-0.5){ $j_{t}$};\node at (3.5,0.5){ $j_{t}$};\node at (4.5,-0.5){\dots};\node at
(5.5,-0.5){\dots};\node at (6.5,-0.5){n};
\node at (8.5,-0.5){1};\node at (9.5,-0.5){\dots};\node at
(10.5,-0.5){ $i_{t}$};\node at
(10.5,0.5){ $i_{t}$};\node at (11.5,-0.5){\dots};\node at (12.5,-0.5){\dots};\node
at (13.5,-0.5){\dots};\node at (14.5,-0.5){n};
\node at (3.5,4.5){ $J$}; \node at (10.5,4.5){ $I$};
\node at (5,2.5){$r_i$}; \node at (5,2){ $\longleftrightarrow$};
\node at (9,2.5){$k_i$}; \node at (9,2){ $\longleftrightarrow$};
 \node at (3.5,-1.5){$F_d$};\node at (11.5,-1.5){$G_d$};
\end{tikzpicture}\end{center}
\caption{The pair of SSAFs $(F_d,G_d)$, $d\ge t$, with the nonempty columns $I$ and $J$. For each $i\ge 1$,  $r_i$ and $k_i$ count the number of columns of height $\ge i$,  to the right of $J$ and to the left of $I$, respectively. If $d=t$, as in Figure~\ref{fig:onlyif}, one has $r_i=k_i=0$ for all $i$.}
\label{fig:classproblem}
\end{figure}
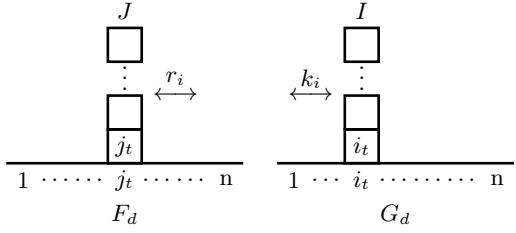

The
classification of the ``\textit{problem}" will
follow from a sequence of four claims as follows.
 The first  is used to prove the second. The second  is used to prove the third. Finally, the last claim complements  the third.

\noindent{\em Claim $1$}: \textit{Let  $(F_d,G_d)$, with $d\ge t$. Then
$k_i\ge r_i\ge 0$, for all $i\ge 1$.}
\begin{proof} By induction on $d\ge t$. For $d=t,$ one has, $k_i=r_i=0,$ for all $i\ge 1$. See Figure~\ref{fig:onlyif}. Let $d\ge t$, and suppose   $(F_d,G_d)$
 satisfies $k_i\ge
r_i\ge 0,$ for all $i\ge 1$. Let us prove for $(F_{d+1},G_{d+1})$.
If the insertion of $j_{d+1}$ terminates in a column of height $l$
to the left or on the top of $J$, then $r_i^{d+1}=r_i^d$, for all $ i$,
$k_i^{d+1}=k_i^d$, for all $i\neq l+1$, and $k_{l+1}^{d+1}=k_{l+1}^d+1$ or
$k_{l+1}^d$. Thus, $k_i^{d+1}\ge r_i^{d+1}$, for all $i\ge 1$. On the other hand,
if the insertion of $j_{d+1}$ terminates to the right of $J$, then
in $F_d$ one has $r_{l}^d>r_{l+1}^d$, and two cases have to be
considered for placing $i_{d+1}$ in $G_d$. First, $i_{d+1}$ sits on
the left of $I$ and, hence, $k_{l+1}^{d+1}=$ $k_{l+1}^d+1$ $\ge r_{l+1}^{d}+1= $
$r_{l+1}^{d+1} $, $k_i^{d+1}= $ $k_i^d\ge $ $r_i^{d}=r_i^{d+1}$, for $i\neq l+1$.
Second, either $i_{d+1}$ sits on the top of $I$ or to the right of
$I$,  in both cases,  $(F_d,G_d)$ satisfy
    $k_{l+1}^d=k_l^d \ge $ $r_l^d
> r_{l+1}^d,$  and, therefore, $k_{l+1}^d > r_{l+1}^d.$ This implies for $(F_{d+1},G_{d+1})$,
$r_{l+1}^{d+1}=r_{l+1}^d+1,$ and  $k_{l+1}^{d+1}= k_{l+1}^d$ $\ge r_{l+1}^{d+1}$,
$k_i^{d+1}=k_i^{d}\ge$ $ r_i^{d}=r_i^{d+1}$, for $i\neq l+1$.
 See Figure~\ref{fig:claim1fig}.
\begin{figure}[h!]
$$\begin{tikzpicture}
\node at (-0.2,0.69){$j_{d+1}$};
\draw [decorate,decoration={brace,amplitude=5pt}] (0.5,0.24) -- (0.5,1.1);
\node at (2.3,1){left or top of $J$ 
};
\node at (2,0.25){right of $J: i_{d+1}$};
\draw [decorate,decoration={brace,amplitude=5pt}] (3.6,-0.1) -- (3.6,0.6);
\node at (5.5,0){top or right of $I$ 
};
\node at (4.7,0.5){left of $I$ 
};
\end{tikzpicture}$$
\caption{All the possible places for terminating the insertion of  $j_{d+1}$ and placing $i_{d+1}$  with respect to the  columns $I$ and $J$.}
\label{fig:claim1fig}
\end{figure}
\end{proof}
 Using Claim $1$, it is shown  next  that the number of columns of height $|I|< i\le|J|$, to the right of $J$  is strictly bigger than the number of columns of height $|I|< i\le|J|$, to the left of $I$, whenever the height of $J$ is strictly bigger than the height of $I$.

\noindent{\em Claim $2$}.  \textit{Let  $(F_d,G_d)$, with $d\ge t$, and
$|J|>|I|$. Then $k_i>r_i\ge 0$, $~i=|I|+1,\ldots,|J|$.}
\begin{proof}
 Since, for $d=t$, it holds $|I|=|J|$, and there is a $d>t$ where for the first time one has  $|J|=|I|+1$.
 We  assume that, for some  $d> t$, one has  $(F_d,G_d)$ with  $ |J|-|I|\ge 1$. Then, either $(F_{d-1},G_{d-1})$ has $|I|=|J|$ or $|J|>|I|$.
In the first case, it means that the insertion of $j_d$  has
terminated on the top of $J$ and the cell $i_d$ sits on the
  left of $I$ on  a column of height $|J|=|I|$. (Otherwise, it would sit on the top of $I$ and again one would have  $(F_d,G_d)$ with new columns $J$ and $I$ satisfying $ |J|=|I|$. Absurd.) Then, by the previous claim,
   $k_{|J|+1}^{d}= k_{|J|+1}^{d-1} + 1 >$ $ r_{|J|+1}^{d-1}= r_{|J|+1}^{d}.$ See Figure~\ref{firstcase}.
\begin{figure}[h!]
$$\begin{tikzpicture}[scale=0.56]
  \footnotesize
\draw[line width=1pt] (4,0) rectangle (5,1); \draw [line width=1pt]
(0,0)--(7,0);
\draw[line width=1pt] (4,1) rectangle (5,2);
\draw[line width=1pt] (4,2) rectangle (5,3);

\draw[line width=1pt] (2,0) rectangle (3,1);
\draw[line width=1pt] (2,1) rectangle (3,2);
\draw[line width=1pt] (2,2) rectangle (3,3);

 \draw[line width=1pt](10,0)rectangle (11,1);
  \draw[line width=1pt](10,1)rectangle (11,2);
   \draw[line width=1pt](10,2)rectangle (11,3);

     \draw[line width=1pt](12,0)rectangle (13,1);
  \draw[line width=1pt](12,1)rectangle (13,2);
   \draw[line width=1pt](12,2)rectangle (13,3);
\node at (2.5,4){ $J$}; \node at (12.5,4){ $I$};
\draw[line width=1pt] (8,0) -- (15,0);
 \node at (0.5,-0.5){1};\node at (1.5,-0.5){\dots};\node at
(2.5,-0.5){ $j_{t}$};\node at (2.5,0.5){ $j_{t}$};
\node at (3.5,-0.5){ \dots};\node at (4.5,-0.5){\dots};\node at
(5.5,-0.5){\dots};\node at (6.5,-0.5){n};
\node at (1,1.5){\dots};\node at (3.5,1.5){\dots};\node at (6.5,1.5){\dots};
\node at (8.5,-0.5){1};\node at (9.5,-0.5){\dots};\node at
(10.5,-0.5){ \dots};\node at (11.5,-0.5){\dots};\node at (12.5,-0.5){$i_{t}$};\node at (12.5,0.5){$i_{t}$};
\node
at (13.5,-0.5){\dots};\node at (14.5,-0.5){n};
\node at (8.5,1.5){\dots};\node at (11.5,1.5){\dots};\node at (14.5,1.5){\dots};

\node at (3.5,-1.5){$F_{d-1}$};\node at (11.5,-1.5){$G_{d-1}$};
\end{tikzpicture}$$
$$\begin{tikzpicture}[scale=0.56]
  \footnotesize
\draw[line width=1pt] (4,0) rectangle (5,1); \draw [line width=1pt]
(0,0)--(7,0);
\draw[line width=1pt] (4,1) rectangle (5,2);
\draw[line width=1pt] (4,2) rectangle (5,3);
\draw[line width=1pt] (2,3) rectangle (3,4);

\draw[line width=1pt] (2,0) rectangle (3,1);
\draw[line width=1pt] (2,1) rectangle (3,2);
\draw[line width=1pt] (2,2) rectangle (3,3);

 \draw[line width=1pt](10,0)rectangle (11,1);
  \draw[line width=1pt](10,1)rectangle (11,2);
   \draw[line width=1pt](10,2)rectangle (11,3);
    \draw[line width=1pt](10,3)rectangle (11,4);

     \draw[line width=1pt](12,0)rectangle (13,1);
  \draw[line width=1pt](12,1)rectangle (13,2);
   \draw[line width=1pt](12,2)rectangle (13,3);
\node at (2.5,5){ $J$}; \node at (12.5,4){ $I$};
\draw[line width=1pt] (8,0) -- (15,0);
 \node at (0.5,-0.5){1};\node at (1.5,-0.5){\dots};\node at
(2.5,-0.5){ $j_{t}$};\node at
(2.5,0.5){ $j_{t}$};
\node at (3.5,-0.5){ \dots};\node at (4.5,-0.5){\dots};\node at
(5.5,-0.5){\dots};\node at (6.5,-0.5){n};
\node at (1,1.5){\dots};\node at (3.5,1.5){\dots};\node at (6.5,1.5){\dots};
\node at (8.5,-0.5){1};\node at (9.5,-0.5){\dots};\node at
(10.5,-0.5){ \dots};\node at (11.5,-0.5){\dots};\node at (12.5,-0.5){$i_{t}$};\node at (12.5,0.5){$i_{t}$};
\node
at (13.5,-0.5){\dots};\node at (14.5,-0.5){n};
\node at (8.5,1.5){\dots};\node at (11.5,1.5){\dots};\node at (14.5,1.5){\dots};

\node at (2.5,3.5){ $j_{d}$}; \node at (10.5,3.5){ $i_{d}$};
\node at (3.5,-1.5){$F_{d}$};\node at (11.5,-1.5){$G_{d}$};
\end{tikzpicture}$$
\caption{The pairs $(F_{d-1},G_{d-1})$,  $(F_{d},G_{d})$ of SSAFs  when $|I|=|J|$ in $(F_{d-1},G_{d-1})$. It holds $r^{d-1}_{|J|+1}=r^d_{|J|+1}$
and $k^{d}_{|J|+1}=k^{d-1}_{|J|+1}+1$.}
\label{firstcase}
\end{figure}
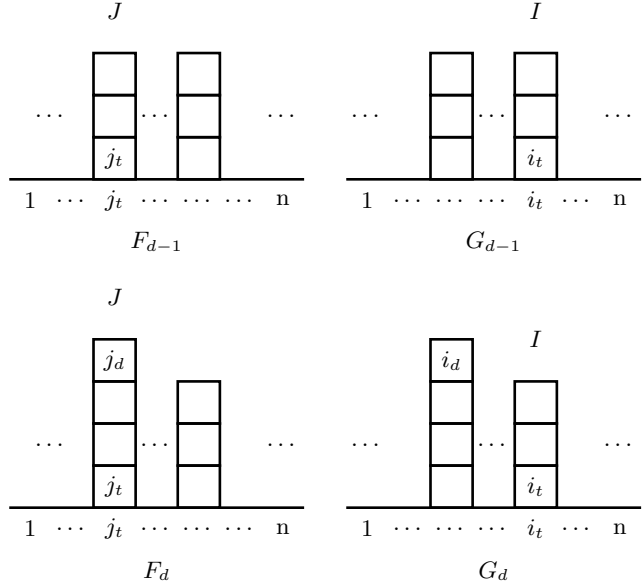

Having in mind the first case,  we suppose in the second that $(F_{d-1},G_{d-1})$  satisfies
$k_i^{d-1}> r_i^{d-1}\ge 0,$ for $i=|I|+1,\dots, |J|$.  Put $z:=|I|<h:=|J|$ for $I$ and $J$ in $(F_{d-1},G_{d-1})$. Let us prove that, if $(F_{d},G_{d})$ has  $|J|>|I|$, then still $k_i^{d}> r_i^{d}\ge 0,$ for $i=|I|+1,\dots, |J|$.
 If the
insertion of $j_d$ terminates in a column of height $l$ ($\neq h-1$)
to the left of $J$ then $r_i^{d}=r_i^{d-1}$, for all $i\ge 1$,
$k_{l+1}^{d}=k_{l+1}^{d-1}+1$, or $k_{l+1}^{d-1}$ and $k_i^{d}=k_i^{d-1}$, for $i\neq l+1$,
and $z\le |I|<|J|=h$, $|J|-|I|\ge 1$. Therefore, $k_i^{d}>r_i^{d}\ge 0$, for
$i=|I|+1,\ldots, |J|$.
If the insertion terminates on the top of
$J$, then $|J|=h+1$, $|I|=z$, $r_i^{d}=r_i^{d-1}$, for all $i\ge 1$,
$k_i^{d}=k_i^{d-1}$, for $i=z+1,\ldots, h$, and $k_{h+1}^{d}=k_{h+1}^{d-1}+1>r_{h+1}^{d}$
or $k_{h+1}^{d}=k_h^{d-1}>r_h^d\ge r_{h+1}^{d}$. Again $k_i^{d}>r_i^{d}$, for
$i=|I|+1,\ldots,|J|=h+1$.
 Finally, if the insertion terminates to
the right of $J$, $|J|=h$ and three cases for the height $l$ have to
be considered. When $l<z$, or $l\ge h$,  $r_i^{d}=r_i^{d-1}<k_i^{d-1}=k_i^{d}$, for
$i=z+1,\ldots, |J|$; when  $l=z$, then either $|I|=z$ and
$k_{z+1}^{d}=k_{z+1}^{d-1}+1>$ $r_{z+1}^{d-1}+1=r_{z+1}^{d}$, $k_i^{d}=k_i^{d-1}>$ $r_i^{d-1}=r_i^{d}$,
$z<i\le |J|$ or $z+1=|I|\le |J|$, and $k_i^{d}=k_i^{d-1}>$ $r_i^{d-1}=r_i^{d}$,
$i=z+2,\ldots, |J|$; and when $z<l< h$, then $|I|=z$ and $r_i^{d}=r_i^{d-1}$,
$i\neq l+1$, and either $k_{l+1}^{d}=k_{l+1}^{d-1}+1>$ $r_{l+1}^{d-1}+1=r_{l+1}^{d}$
or $k_{l+1}^{d}=k_{l+1}^{d-1}=k_l^{d-1}>r_l^{d-1}\ge $ $r_{l+1}^{d-1}+1=r_{l+1}^{d}$. Henceforth $k_i^{d}>r_i^{d}$,
for $i=|I|+1,\ldots,|J|$.
 See Figure~\ref{claim2fig}.
\begin{figure}[h!]
\begin{tikzpicture}
\node at (0,2){$j_{d}$};\node at (-1.4,2){2nd case:};
\draw [decorate,decoration={brace,amplitude=5pt}] (0.7,0.4) -- (0.7,3.1);
\node at (2,3){left of $J: i_{d}$};
\node at (2,2){top of $J: i_{d}$};
\node at (2.1,0.5){right of $J: i_{d}$};
\node at (2,0){height $l$};
\draw [decorate,decoration={brace,amplitude=5pt}] (3.5,2.74) -- (3.5,3.26);
\node at (5.5,3.25){top or right of $I$ 
};
\node at (4.7,2.75){left of $I$ 
};
\draw [decorate,decoration={brace,amplitude=5pt}] (3.5,1.74) -- (3.5,2.26);
\node at (4.7,1.75){left of $I$ 
 };
\node at (4.7,2.25){right of $I$ 
};
\draw [decorate,decoration={brace,amplitude=5pt}] (3.7,-0.6) -- (3.7,1.26);
\node at (5.5,1.25){$l<z$ or $l\ge h$ 
};
\node at (4.3,0.5){$l=z$};
\draw [decorate,decoration={brace,amplitude=5pt}] (5.2,0.24) -- (5.2,0.76);
\node at (6.3,0.75){left of $I$
};
\node at (6.3,0.25){top of $I$
};

\node at (4.6,-0.5){$z<l<h$};
\draw [decorate,decoration={brace,amplitude=5pt}] (5.7,-0.76) -- (5.7,-0.24);
\node at (6.85,-0.25){left of $I$ 
 };
\node at (7,-0.75){right of $I$ 
};
\end{tikzpicture}
\caption{All the  possible places for $i_d$ depending on the terminating place for the insertion of $j_d$ in $(F_{d-1},G_{d-1})$
 under the condition $h=|J|>z=|I|$, and the new columns $J$ and $I$ still satisfy $|J|>|I|$.}
\label{claim2fig}
\end{figure}
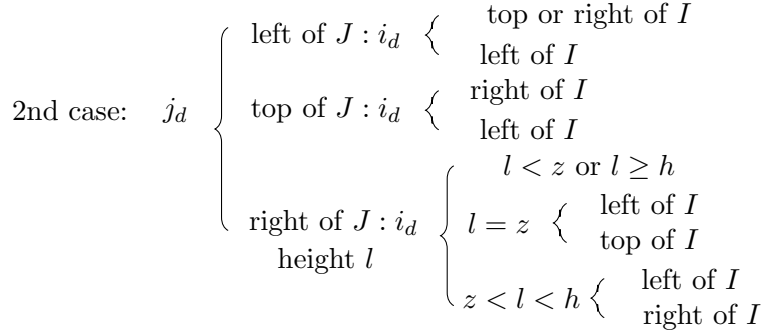
\end{proof}

\noindent{\em Claim $3$}: \textit{Let  $(F_d,G_d)$, with $d\ge t$, be such
that, for some $s\ge 1$,  one has    $|I|, |J|\ge s$ and
$k_s=r_s>0$. Then, for $(F_{d+1},G_{d+1})$ there exists also an
$s\ge 1$ with the same properties.}
\begin{proof}
 Observe  that, from the previous claim, $k_{s+1}^d=r_{s+1}^d$ and $|J|\ge s+1$ only if $|I|\ge s+1$. If the
 insertion of $j_{d+1}$ terminates on the top of a column of height $l\neq s-1$, then still $|I|, |J|\ge s$ and
 $k_s^{d+1}=r_s^{d+1}>0$. It remains to analyse when $l=s-1$ which means that the insertion of $j_{d+1}$ either terminates to the
 left or to the right of $J$. In the first case,  $(F_d,G_d)$ satisfies $|J|\ge s+1$ (using Remark~\ref{insertion}), $r_s^d=r_{s+1}^d$, and, therefore,
  $k_{s+1}^d\ge$ $ r_{s+1}^d=$ $r_s^d= k_s^d\ge$ $ k_{s+1}^d$. It implies for $(F_{d+1},G_{d+1})$ that  $ k_{s+1}^{d+1}= r_{s+1}^{d+1}>0$,
   $|J|,|I|\ge s+1$, and thus the claim is true for $s+1$. In the second case, $(F_d,G_d)$ satisfies
   $k_{s-1}^d\ge r_{s-1}^d>r_s^d=k_s^d$ and thus $k_{s-1}^d>k_s^d$. Thereby the cell $i_{d+1}$ sits to the left of $I$ and
   $r_s^{d+1}=r_s^d+1=k_s^d+1=k_s^{d+1}$, with $|I|,|J|\ge s$. The claim is true for $s$.
 See Figure~\ref{claim3fig}.
\begin{figure}[h!]
$$\begin{tikzpicture}
\node at (-1,2.45){$j_{d+1}$};
\draw [decorate,decoration={brace,amplitude=5pt}] (-0.4,1.74) -- (-0.4,3.1);
\node at (1.4,3){height:  $l= s-1$};
\node at (1.6,1.75){height:  $l\neq s-1$ 
};
\draw [decorate,decoration={brace,amplitude=5pt}] (3.3,2.74) -- (3.3,3.6);
\node at (4.6,3.5){left of $J$ 
};
\node at (4.6,2.75){right of $J$ 
};
\end{tikzpicture}$$
\caption{All the possibilities  for the insertion of $j_{d+1}$ in  $F_{d}$.}
\label{claim3fig}
\end{figure}
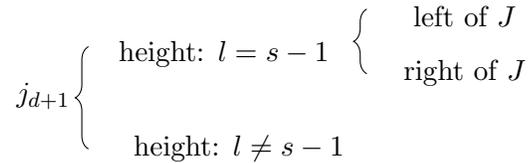
   \end{proof}
\noindent Next claim describes the pair  $(F_d,G_d)$ of SSAFs, for
$d\ge t$, when it does not fit the conditions of Claim $3$.

\noindent {\em Claim $4$}. \textit{Let $(F_d,G_d)$, with $d\ge t$, be  a
pair of SSAFs such that, for all  $i=1,$ $\ldots, $ $
\min\{|I|,|J|\}$, $k_i=r_i>0$ never holds. Then, $|J|\le |I|$ and,
there is $1\le f\le |J|$, such that $k_i>r_i$, for  $1\le i<f$, and
$k_i=r_i=0$, for $i\ge f$.}
\begin{proof} We show by induction on $d\ge t$ that   $(F_d,G_d)$ either satisfy  the conditions of the Claim $3$ or,
otherwise, $|J|\le |I|$ and, there is $1\le f\le |J|$, such that
$r_i<k_i$, for  $1\le i<f$, and $k_i=r_i=0$, for  $i\ge f$.
For
$d=t,$ we have $|I|=|J|=1,$ and $k_i^d=r_i^d=0,$ $i\ge 1$. Put $f:=1.$
Let $(F_d,G_d),$ with $d\ge t$. If  $(F_d,G_d) $ fits the conditions
of Claim $3$, then $(F_{d+1}, G_{d+1})$ does it as well.
Otherwise, assume  for $(F_d,G_d)$, $|J|\le |I|$, and, there exists
$1\le f\le |J|$, such that $r_i^d<k_i^d$, for  $1\le i<f$, and
$k_i^d=r_i^d=0$, for $i\ge f$. We show next that $(F_{d+1}, G_{d+1})$
either fits the conditions of the previous Claim $3$, or,
otherwise, it is as described in the present claim. If the insertion
of $j_{d+1}$ terminates to the left  of $J$, and $i_{d+1}$ sits on
the top or to the right of $I$, still $|I|\ge |J|$ and there is
nothing to prove. If $i_{d+1}$ sits on the top of a column of height
$l$, to the left of $I$, then, since $k_f^{d}=0$, one has $l<f$, and two
cases have to be considered. When $l=f-1$, it implies $|I|\ge|J|\ge
f+1$ (using Remark~\ref{insertion}), $r_f^{d+1}=0$ and $k_f^{d+1}=1$, and $(F_{d+1}, G_{d+1})$ satisfies the
claim for $f+1$; in the case of $l<f-1$,
$r_{l+1}^{d+1}=r_{l+1}^{d}<k_{l+1}^d+1=k_{l+1}^{d+1}$ and still, for the same $f$, $k_i^{d+1}>r_i^{d+1}$,
$1\le i<f$, $k_i^{d+1}=r_i^{d+1}=0$, $i\ge f$. If the insertion of $j_{d+1}$
terminates on the top of $J$, since $k_{|J|}^d=0$ and $|I|\ge |J|$,
then  $i_{d+1}$ either sits on the top of $I$ when $|I|=|J|$, and
still for the same $f$, $k_i^{d+1}>r_i^{d+1}$, $1\le i<f$, $k_i^{d+1}=r_i^{d+1}=0$, $i\ge
f$, or sits to the right of $I$, when $|I|>|J|$, and still $|I|\ge
|J|+1$, and, for the same $f$, $k_i^{d+1}>r_i^{d+1}$, $1\le i<f$, $k_i^{d+1}=r_i^{d+1}=0$,
$i\ge f$. If the insertion of $j_{d+1}$ terminates to the right of
$J$ on the top of a column of height $l<f$ (recall that $r_f^{d}=0$),
then, since $|I|> f$, $i_{d+1}$ either sits on the left of $I$ or to
the right of $I$. In the first case, if $l=f-1$, one has
$r_f^{d+1}=r_f^{d}+1=k_f^{d}+1=k_f^{d+1}=1,$  and, therefore, we are in the conditions
of Claim $3$, with   $s=f< |J|\le |I|$; if $l<f-1$, still
$r_{l+1}^{d+1}=r_{l+1}^{d}+1<k_{l+1}^{d}+1=k_{l+1}^{d+1},$ so $k_i^{d+1}> r_i^{d+1},$ for $1\le
i<f$ and $r_i^{d+1}=k_i^{d+1}=0$, for $i\ge f$. In the second case, it means
$k_{l+1}^{d+1}=k_{l+1}^{d}=k_l^{d}>r_l^{d}\ge r_{l+1}^{d}+1= r_{l+1}^{d+1}$ and hence
$k_{l+1}^{d+1}>r_{l+1}^{d+1}$, with $l+1<f$.  Note that $k_f=0$ and $k_{f-1}>r_{f-1}$, so $k_{f-1}\neq 0$ therefore $l\neq f-1$. Similarly, $k_i^{d+1}>r_i^{d+1}$,
for $1\le i<f$ and $k_i^{d+1}=r_i^{d+1}=0$, $i\ge f$.
 See Figure~\ref{claim4fig}.
\begin{figure}[h!]
\begin{tikzpicture}
\node at (-1,1.75){$j_{d+1}$};
\draw [decorate,decoration={brace,amplitude=5pt}] (-0.4,0.4) -- (-0.4,3.1);
\node at (1,3){left of $J: i_{d+1}$};
\node at (1,1.75){top of $J: i_{d+1}$};
\node at (1,0.5){right of $J: i_{d+1}$};
\draw [decorate,decoration={brace,amplitude=5pt}] (2.5,2.74) -- (2.5,3.6);
\node at (4.6,3.5){top or right of $I$ 
};
\node at (4.6,2.75){left of $I$ with height $l$:};
\draw [decorate,decoration={brace,amplitude=5pt}] (6.9,2.4) -- (6.9,3.1);
\node at (8.1,3){$l=f-1$ 
 };
\node at (8.1,2.5){$l<f-1$ 
};
\draw [decorate,decoration={brace,amplitude=5pt}] (2.5,1.4) -- (2.5,2.1);
\node at (3.8,2){top of $I$ 
 };
\node at (3.8,1.5){right of $I$ 
};
\draw [decorate,decoration={brace,amplitude=5pt}] (2.6,-0.1) -- (2.6,0.76);
\node at (4,0){right of $I$ 
 };
\node at (4.6,0.75){left of $I$ with height $l$:};
\draw [decorate,decoration={brace,amplitude=5pt}] (6.9,0.4) -- (6.9,1.1);
\node at (8.1,1){$l=f-1$ 
};
\node at (8.1,0.5){$l<f-1$ 
};
\end{tikzpicture}
\caption{All the possibilities for  inserting $j_{d+1}$ in $F_d$ and placing $i_{d+1}$  in $G_d$.}
\label{claim4fig}
\end{figure}
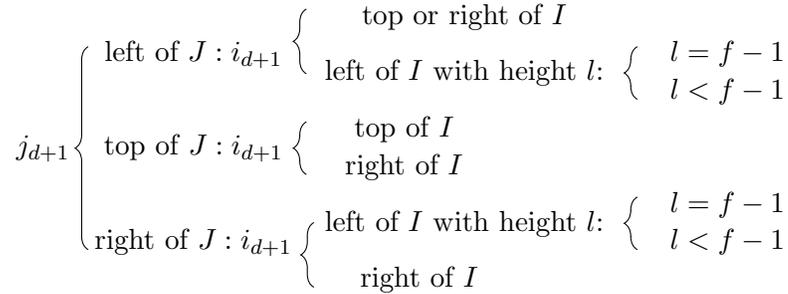
\end{proof}
Recall that for any $d\ge t$, $i_t$ appears in the $|I|$ first columns of the $key(sh(G_{d}))$ and $j_t$ appears in the $|J|$ first columns of the $key(\omega sh(F_{d}))$.

\noindent{\em
Classification of the ``problem"}:
For each $d\ge t$, either  there
exists $s\ge 1$ such that $|J|,|I|\ge s$, $r_s=k_s>0$; or
$1\le|J|\le |I|$, and there exists $1\le f\le |J|$, such that
$k_i>r_i$, for  $1\le i< f$, and $k_i=r_i=0$, for $i\ge f$. In the
first case,  one has a ``\textit{problem}" in the $(r_s+1)^{th}$ rows of the
$s^{th}$ columns in the key-pair $(key(sh(G_{d})), key(\omega
sh(F_{d})))$. In the second case, one has a ``\textit{problem}'' in the bottom of
the $|J|^{th}$ columns. See Example~\ref{ex:main}.2. and Figure~\ref{classi2}.
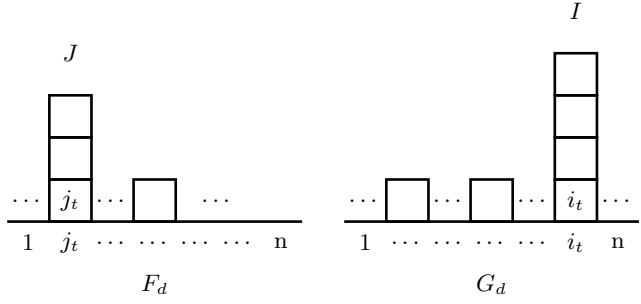
\begin{figure}[h!]
$$\begin{tikzpicture}[scale=0.56]
  \footnotesize
\draw[line width=1pt] (3,0) rectangle (4,1); \draw [line width=1pt]
(0,0)--(7,0);
\draw[line width=1pt] (1,0) rectangle (2,1);
\draw[line width=1pt] (1,1) rectangle (2,2);
\draw[line width=1pt] (1,2) rectangle (2,3);
 \draw[line width=1pt](9,0)rectangle (10,1);
 \draw[line width=1pt](11,0)rectangle (12,1);
    \draw[line width=1pt](13,3)rectangle (14,4);
     \draw[line width=1pt](13,0)rectangle (14,1);
  \draw[line width=1pt](13,1)rectangle (14,2);
   \draw[line width=1pt](13,2)rectangle (14,3);
\node at (1.5,4){ $J$}; \node at (13.5,5){ $I$};
\draw[line width=1pt] (8,0) -- (15,0);
 \node at (0.5,-0.5){1};\node at (1.5,-0.5){$j_{t}$};\node at (1.5,0.5){$j_{t}$};
 \node at
(2.5,-0.5){\dots };\node at (3.5,-0.5){ \dots};\node at (4.5,-0.5){\dots};\node at
(5.5,-0.5){\dots};\node at (6.5,-0.5){n};
\node at (0.5,0.5){\dots};\node at (2.5,0.5){\dots};\node at (5,0.5){\dots};
\node at (8.5,-0.5){1};\node at (9.5,-0.5){\dots};\node at
(10.5,-0.5){ \dots};\node at (11.5,-0.5){\dots};\node at (13.5,-0.5){$i_{t}$};\node at (13.5,0.5){$i_{t}$};
\node
at (12.5,-0.5){\dots};\node at (14.5,-0.5){n};
\node at (8.5,0.5){\dots};\node at (10.5,0.5){\dots};\node at (12.5,0.5){\dots};\node at (14.5,0.5){\dots};
\node at (3.5,-1.5){$F_{d}$};\node at (11.5,-1.5){$G_{d}$};
\end{tikzpicture}$$
\caption{$1\le|J|=3\le |I|=4$, and there exists $1\le f=2\le |J|=3$, such that
$k_i>r_i$, for  $1\le i< f$, and $k_i=r_i=0$, for $i\ge f$. The entry at the bottom of the 3th column of $key(\omega sh(F_d))$ is $ n-j_t+1$ while the corresponding entry in the 3th column of $key(sh(G_d))$ is equal to $i_{t}$ and $n-j_t+1<i_t$.}
\label{classi2}
\end{figure}

Finally, if the second row of $w$ is over the alphabet
$[m],$ there
 is no cell on the top of the basement of  $F$ greater than $m.$
 Therefore, the shape of $F$ has the last $n-m$ entries equal to
 zero and thus its decreasing rearrangement is a partition of length $\le m$. Using the symmetry of $\Phi$, the other case is similar.
\end{proof}
\begin{re}\label{remark2}
1. Given $\nu\in
\mathbb{N}^n$ and $\beta\in\mathbb{N}^n$ such that $\beta\le \omega \nu$, recalling the definition of key SSAF, \eqref{existence-SSAF}, there exists always  a pair
$(F,G)$ of SSAFs with shapes $\nu$ and $\beta$ respectively. For instance, the corresponding key SSAF pair.

2. If   the rows in $ w$ are swapped and  rearranged in lexicographic order, one obtains  the
biword $\tilde w$ such that $\Phi(\tilde w)=(G,F)$ with
$key(sh(F))\le key(\omega sh(G))$.

3. If $k+m\le n+1$, the biletters of $w$ are cells of a rectangle $(m^k)$ inside the staircase of length $n$, and  $sh(G)\le \omega sh(F)$ is  trivially satisfied. Notice that  $key(sh(G))$ has entries $\le k$ while $key(\omega sh(F))$ has entries $\ge n-m+1\ge k$.

4. Example~\ref{ex:main}.$2$., below, shows that if $w$ consists both of biletters  above and inside the staircase of size $n$, then we may have either $sh(G)$ and $\omega sh(F)$ not comparable or  $sh(G)>\omega sh(F)$.
\end{re}
Using the  bijection $\Psi$ between $SSYT_n$ and $SSAF_n$,
one has,
\begin{cor}
Let $w$ be a biword in lexicographic order in the alphabet  $[n]$,
and let $w$ $\underrightarrow{RSK}$ $ (P,Q).$
 For each biletter $
 \binom i j
 $ in $w$ we have $1\le j\le m\le n$, $1\le i\le k\le n$, and $i+j\leq n+1$ if and only
if $Q$ has entries $\le k$, $ P$ has entries $\le m$, and $K_+(Q)\leq evac(K_+(P))$.
\end{cor}
Two examples are given to illustrate Theorem~\ref{maint}.
\begin{ex}\label{ex:main}
{\em
\begin{enumerate}
\item Given $w=\left(\begin{array}{ccccccc} 4&6&6&7\\
4&1&2&1\end{array}\right),$   $\Phi(w)$ and its key-pair satisfying
$key(sh(G))\le key(\omega sh(F))$ are calculated.
 $$\begin{tikzpicture}[scale=0.43]
  \footnotesize
\draw[line width=1pt] (0,0) rectangle (1,1); \draw [line width=1pt]
(0,0)--(7,0);

 \draw[line width=1pt]
(14,0)
rectangle (15,1);
\draw[line width=1pt] (8,0) -- (15,0);

 \node at (0.5,-0.5){1};\node at (1.5,-0.5){2};\node at
(2.5,-0.5){3};\node at (3.5,-0.5){4};\node at (4.5,-0.5){5};\node at
(5.5,-0.5){6};\node at (6.5,-0.5){7};

\node at (8.5,-0.5){1};\node at (9.5,-0.5){2};\node at
(10.5,-0.5){3};\node at (11.5,-0.5){4};\node at (12.5,-0.5){5};\node
at (13.5,-0.5){6};\node at (14.5,-0.5){7};

\node at (0.5,0.5){1}; \node at (14.5,0.5){7};

 \node at (3.5,-1.5){$sh(F_1)=(1,0^{6})$};\node at (11.5,-1.5){$sh(G_1)=(0^{6},1)$};\node at
 (8,-2.5){$key(sh(G_1))=7=key(\omega sh(F_1))$};

  \footnotesize
\draw[line width=1pt] (16,0) rectangle (17,1); \draw[line width=1pt]
(17,0) rectangle (18,1);\draw[line width=1pt] (29,0) rectangle
(30,1); \draw [line width=1pt] (16,0)--(23,0);

 \draw[line width=1pt]
(30,0)
rectangle (31,1);
\draw[line width=1pt] (24,0) -- (31,0);

 \node at (16.5,-0.5){1};\node at (17.5,-0.5){2};\node at
(18.5,-0.5){3};\node at (19.5,-0.5){4};\node at (20.5,-0.5){5};\node
at (21.5,-0.5){6};\node at (22.5,-0.5){7};

\node at (24.5,-0.5){1};\node at (25.5,-0.5){2};\node at
(26.5,-0.5){3};\node at (27.5,-0.5){4};\node at (28.5,-0.5){5};\node
at (29.5,-0.5){6};\node at (30.5,-0.5){7};

\node at (16.5,0.5){1};\node at (17.5,0.5){2}; \node at
(30.5,0.5){7};\node at (29.5,0.5){6};

 \node at (19.5,-1.5){$sh(F_2)=(1^2,0^{5})$};\node at (27.5,-1.5){$sh(G_2)=(0^{5},1^2)$};\node at
 (24,-2.5){$key(sh(G_2))=67=key(\omega sh(F_2))$};

 \node at (7.5,0){,}; \node at (15.5,0){;}; \node at (23.5,0){,};
\end{tikzpicture}$$

\begin{center}\begin{tikzpicture}[scale=0.43]
  \footnotesize
   \node at (7.5,0){,}; \node at (15.5,0){;}; \node at (23.5,0){,};
\draw[line width=1pt] (0,0) rectangle (1,1); \draw[line width=1pt]
(1,0) rectangle (2,1); \draw[line width=1pt] (0,1) rectangle
(1,2);\draw [line width=1pt] (0,0)--(7,0);\draw[line width=1pt]
(13,1) rectangle (14,2);\draw[line width=1pt] (13,0) rectangle
(14,1); \draw[line width=1pt] (14,0) rectangle (15,1); \draw[line
width=1pt] (8,0) -- (15,0);

 \node at (0.5,-0.5){1};\node at (1.5,-0.5){2};\node at
(2.5,-0.5){3};\node at (3.5,-0.5){4};\node at (4.5,-0.5){5};\node at
(5.5,-0.5){6};\node at (6.5,-0.5){7};

\node at (8.5,-0.5){1};\node at (9.5,-0.5){2};\node at
(10.5,-0.5){3};\node at (11.5,-0.5){4};\node at (12.5,-0.5){5};\node
at (13.5,-0.5){6};\node at (14.5,-0.5){7};

\node at (0.5,0.5){1};\node at (1.5,0.5){2}; \node at
(14.5,0.5){7};\node at (13.5,0.5){6};\node at (0.5,1.5){1};\node at
(13.5,1.5){6};

 \node at (3.5,-1.5){$sh(F_3)=(2,1,0^{5})$};\node at (10.95,-1.5){$sh(G_3)=(0^{5},2,1)$};\node at
 (9,-3){$key(sh(G_3))=\begin{array}{cc}7\\6&6\end{array}\leq \begin{array}{cc}7\\6&7\end{array}=key(\omega sh(F_3))$};

  \footnotesize
\draw[line width=1pt] (16,0) rectangle (17,1); \draw[line width=1pt]
(17,0) rectangle (18,1); \draw[line width=1pt] (16,1) rectangle
(17,2);\draw [line width=1pt] (16,0)--(23,0);\draw[line width=1pt]
(29,1) rectangle (30,2);\draw[line width=1pt] (29,0) rectangle
(30,1); \draw[line width=1pt] (30,0) rectangle (31,1); \draw[line
width=1pt] (24,0) -- (31,0);\draw[line width=1pt] (19,0) rectangle
(20,1);\draw[line width=1pt] (27,0) rectangle (28,1);

 \node at (16.5,-0.5){1};\node at (17.5,-0.5){2};\node at
(18.5,-0.5){3};\node at (19.5,-0.5){4};\node at (20.5,-0.5){5};\node
at (21.5,-0.5){6};\node at (22.5,-0.5){7};

\node at (24.5,-0.5){1};\node at (25.5,-0.5){2};\node at
(26.5,-0.5){3};\node at (27.5,-0.5){4};\node at (28.5,-0.5){5};\node
at (29.5,-0.5){6};\node at (30.5,-0.5){7};

\node at (16.5,0.5){1};\node at (17.5,0.5){2}; \node at
(30.5,0.5){7};\node at (29.5,0.5){6};\node at (16.5,1.5){1};\node at
(29.5,1.5){6};\node at (19.5,0.5){4};\node at (27.5,0.5){4};

 \node at (18.9,-1.5){$sh(F_4)=(2,1,0,1,0^{3})$};\node at (27.5,-1.5){$sh(G_4)=(0^{3},1,0,2,1)$};\node at
 (22.5,-4){$key(sh(G_4))=\begin{array}{cc}7\\6\\4&6\end{array}\leq \begin{array}{cc}7\\6\\4&7\end{array}=key(\omega sh(F_4))$};

\end{tikzpicture}\end{center}

\item
Let $w=\left(\begin{array}{ccccccc} 1&2&3&3&{\color{red}5}&6\\
6&3&2&4&{\color{red}3}&1\end{array}\right),$  $n=6$, $i_2=5 >6+1-3=6+1-j_2=4$. We
calculate $\Phi(w)$ whose  key-pair $(key(sh(G))\nleq key(\omega sh(F)))$.
\begin{center}\begin{tikzpicture}[scale=0.46]
 \footnotesize
\draw[line width=1pt] (0,0) rectangle (1,1); \draw [line width=1pt]
(0,0)--(6,0);

 \draw[line width=1pt]
(12,0) rectangle (13,1); \draw[line width=1pt] (7,0) -- (13,0);

 \node at (0.5,-0.5){1};\node at (1.5,-0.5){2};\node at
(2.5,-0.5){3};\node at (3.5,-0.5){4};\node at (4.5,-0.5){5};\node at
(5.5,-0.5){6};

\node at (7.5,-0.5){1};\node at (8.5,-0.5){2};\node at
(9.5,-0.5){3};\node at (10.5,-0.5){4};\node at (11.5,-0.5){5};\node
at (12.5,-0.5){6};

\node at (0.5,0.5){1};\node at (12.5,0.5){6};

 \node at (3,-1.5){$sh(F_1)=(1,0^{5})$};\node at (10,-1.5){$sh(G_1)=(0^{5},1)$};\node at
 (5.5,-3){$key(sh(G_1))=6=key(\omega sh(F_1))$};

  \footnotesize
\draw[line width=1pt] (15,0) rectangle (16,1); \draw[line width=1pt]
(17,0) rectangle (18,1);\draw[line width=1pt] (26,0) rectangle
(27,1); \draw [line width=1pt] (15,0)--(21,0);

 \draw[line width=1pt]
(27,0)
rectangle (28,1); 
\draw[line width=1pt] (22,0) -- (28,0);

 \node at (15.5,-0.5){1};\node at (16.5,-0.5){2};\node at
(17.5,-0.5){3};\node at (18.5,-0.5){4};\node at (19.5,-0.5){5};\node
at (20.5,-0.5){6};

\node at (22.5,-0.5){1};\node at (23.5,-0.5){2};\node at
(24.5,-0.5){3};\node at (25.5,-0.5){4};\node at (26.5,-0.5){5};\node
at (27.5,-0.5){6};

\node at (15.5,0.5){1};\node at (17.5,0.5){3}; \node at
(27.5,0.5){6};\node at (26.5,0.5){5};

 \node at (17.5,-1.5){$sh(F_2)=(1,0,1,0^{3})$};\node at (24.5,-1.5){$sh(G_2)=(0^{4},1^2)$};\node at
 (20.5,-3){$key(sh(G_2))=\begin{array}{c}6\\{\color{red}5}\end{array}> \begin{array}{c}6\\{\color{red}4}\end{array}=key(\omega sh(F_2))$};
 \node at(6.5,0){,}; \node at (14,0){;}; \node at (21.5,0){,};
\end{tikzpicture}\end{center}
\begin{center}\begin{tikzpicture}[scale=0.46]
  \footnotesize
\draw[line width=1pt] (0,0) rectangle (1,1); \draw[line width=1pt]
(2,0) rectangle (3,1); \draw[line width=1pt] (3,0) rectangle
(4,1);\draw [line width=1pt] (0,0)--(6,0);\draw[line width=1pt]
(9,0) rectangle (10,1);\draw[line width=1pt] (12,0) rectangle
(13,1);\draw[line width=1pt] (11,0) rectangle (12,1); \draw[line
width=1pt] (7,0) -- (13,0);

 \node at (0.5,-0.5){1};\node at (1.5,-0.5){2};\node at
(2.5,-0.5){3};\node at (3.5,-0.5){4};\node at (4.5,-0.5){5};\node at
(5.5,-0.5){6};

\node at (7.5,-0.5){1};\node at (8.5,-0.5){2};\node at
(9.5,-0.5){3};\node at (10.5,-0.5){4};\node at (11.5,-0.5){5};\node
at (12.5,-0.5){6};

\node at (0.5,0.5){1};\node at (2.5,0.5){3}; \node at
(12.5,0.5){6};\node at (9.5,0.5){3};\node at (3.5,0.5){4};\node at
(11.5,0.5){5};

 \node at (2.7,-1.5){\scriptsize$sh(F_3)=(1,0,1^2,0^{2})$};\node at (10.1,-1.5){\scriptsize$sh(G_3)=(0^{2},1,0,1^2)$};\node at
 (7.5,-3.7){$key(sh(G_3))=\begin{array}{c}6\\{\color{red}5}\\3\end{array}>
  \begin{array}{cc}6\\{\color{red}4}\\3\end{array}=key(\omega sh(F_3))$};

  \footnotesize

\draw[line width=1pt] (15,0) -- (21,0);

\draw[line width=1pt] (15,0) rectangle (16,1);\draw[line width=1pt]
(17,0) rectangle (18,1);\draw[line width=1pt] (18,0) rectangle
(19,1);\draw[line width=1pt] (17,1) rectangle (18,2);

\node at (15.5,-0.5){1};\node at (16.5,-0.5){2}; \node at
(17.5,-0.5){3}; \node at (18.5,-0.5){4}; \node at (19.5,-0.5){5};
\node at (20.5,-0.5){6};

\draw[line width=1pt] (22,0) -- (28,0);

\draw[line width=1pt] (24,0) rectangle (25,1);\draw[line width=1pt]
 (26,0) rectangle (27,1); \draw[line width=1pt](27,0) rectangle
(28,1); \draw[line width=1pt]  (24,1) rectangle (25,2);

\node at (22.5,-0.5){1};\node at (23.5,-0.5){2}; \node at
(24.5,-0.5){3}; \node at (25.5,-0.5){4}; \node at (26.5,-0.5){5};
\node at (27.5,-0.5){6};

\node at (15.5,0.5){1};\node at (17.5,0.5){3}; \node at
(17.5,1.5){2}; \node at (18.5,0.5){4};

\node at (24.5,0.5){3}; \node at (24.5,1.5){3}; \node at
(26.5,0.5){5}; \node at (27.5,0.5){6};

 \node at (17.8,-1.5){\scriptsize$sh(F_4)=(1,0,2,1,0^{2})$};\node at (25.4,-1.5){\scriptsize$sh(G_4)=(0^{2},2,0,1^2)$};\node at
 (20,-4.5){$key(sh(G_4))=\begin{array}{cc}6\\{\color{red}5}\\3&3\end{array}\ngtr
 \begin{array}{cc}6\\{\color{red}4}\\3&4\end{array}=key(\omega sh(F_4))$};
  \node at(6.5,0){,}; \node at (14,0){;}; \node at (21.5,0){,};
  \node at (19.7,-5.5){$\nleq$};
\end{tikzpicture}\end{center}
\begin{center}\begin{tikzpicture}[scale=0.46]
  \footnotesize
\draw[line width=1pt] (0,0) rectangle (1,1); \draw[line width=1pt]
(2,0) rectangle (3,1); \draw[line width=1pt] (3,0) rectangle
(4,1);\draw[line width=1pt] (2,1) rectangle (3,2); \draw[line
width=1pt] (3,1) rectangle (4,2);\draw [line width=1pt]
(0,0)--(6,0);\draw[line width=1pt] (9,0) rectangle (10,1);\draw[line
width=1pt] (12,0) rectangle (13,1); \draw[line width=1pt] (9,1)
rectangle (10,2);\draw[line width=1pt] (11,1) rectangle (12,2);
\draw[line width=1pt] (11,0) rectangle (12,1); \draw[line width=1pt]
(7,0) -- (13,0);

\node at (2.5,1.5){3};\node at (3.5,1.5){2};\node at
(9.5,1.5){3};\node at (11.5,1.5){2};

 \node at (0.5,-0.5){1};\node at (1.5,-0.5){2};\node at
(2.5,-0.5){3};\node at (3.5,-0.5){4};\node at (4.5,-0.5){5};\node at
(5.5,-0.5){6};

\node at (7.5,-0.5){1};\node at (8.5,-0.5){2};\node at
(9.5,-0.5){3};\node at (10.5,-0.5){4};\node at (11.5,-0.5){5};\node
at (12.5,-0.5){6};

\node at (0.5,0.5){1};\node at (2.5,0.5){3}; \node at
(12.5,0.5){6};\node at (9.5,0.5){3};\node at (3.5,0.5){4};\node at
(11.5,0.5){5};

 \node at (2.7,-1.5){\scriptsize$sh(F_5)=(1,0,2^2,0^{2})$};\node at (10.1,-1.5){\scriptsize$sh(G_5)=(0^{2},2,0,2,1)$};\node at
 (7.5,-3.7){$key(sh(G_5))=\begin{array}{cc}6\\{\color{red}5}&{\color{red}5}\\3&3\end{array}>
  \begin{array}{cc}6\\{\color{red}4}&{\color{red}4}\\3&3\end{array}=key(\omega sh(F_5))$};

  \footnotesize

\draw[line width=1pt] (15,0) -- (21,0);

\draw[line width=1pt] (15,0) rectangle (16,1);\draw[line width=1pt]
(17,0) rectangle (18,1);\draw[line width=1pt] (18,0) rectangle
(19,1);\draw[line width=1pt] (17,1) rectangle (18,2);\draw[line
width=1pt] (18,1) rectangle (19,2);\draw[line width=1pt] (20,0)
rectangle (21,1);

\node at (15.5,-0.5){1};\node at (16.5,-0.5){2}; \node at
(17.5,-0.5){3}; \node at (18.5,-0.5){4}; \node at (19.5,-0.5){5};
\node at (20.5,-0.5){6};

\draw[line width=1pt] (22,0) -- (28,0);

\draw[line width=1pt] (22,0) rectangle (23,1); \draw[line width=1pt]
(24,0) rectangle (25,1);\draw[line width=1pt]
 (26,0) rectangle (27,1); \draw[line width=1pt](27,0) rectangle
(28,1); \draw[line width=1pt]  (24,1) rectangle (25,2); \draw[line
width=1pt]  (26,1) rectangle (27,2);

\node at (22.5,-0.5){1};\node at (23.5,-0.5){2}; \node at
(24.5,-0.5){3}; \node at (25.5,-0.5){4}; \node at (26.5,-0.5){5};
\node at (27.5,-0.5){6};

\node at (15.5,0.5){1};\node at (17.5,0.5){3}; \node at
(17.5,1.5){3}; \node at (18.5,1.5){2};\node at (18.5,0.5){4};\node
at (20.5,0.5){6};

\node at (24.5,0.5){3}; \node at (24.5,1.5){3}; \node at
(26.5,0.5){5}; \node at (27.5,0.5){6};\node at (26.5,1.5){2}; \node
at (22.5,0.5){1};

 \node at (17.6,-1.5){\scriptsize$sh(F_6)=(1,0,2^{2},0,1)$};\node at (25.5,-1.5){\scriptsize$sh(G_6)=(1,0,2,0,2,1)$};\node at
 (20,-4.5){$key(sh(G_6))=\begin{array}{cc}6\\{\color{red}5}\\3&{\color{red}5}\\1&3\end{array}>
 \begin{array}{cc}6\\{\color{red}4}\\3&{\color{red}4}\\1&3\end{array}=key(\omega sh(F_6))$};
 \node at(6.5,0){,}; \node at (14,0){;}; \node at (21.5,0){,};
\end{tikzpicture}\end{center}
\item If  the biword $w$, in $2$., is restricted to the last 4 biletters,  $sh(G)$ and $\omega sh(F)$ are not comparable. However, when the two first are added,  $sh(G)>\omega sh(F)$ holds, as one sees above.
\end{enumerate}
}
\end{ex}

\section{ Isobaric divided differences and crystal graphs}
\label{sec:iso}
We review the main results on Demazure operators with an eye on their combinatorial interpretations either as bubble sorting operators acting on weak compositions or their combinatorial version in terms of crystal (coplactic) operators to be used in the last section.
\label{sec:isob}
 \subsection{Isobaric divided differences,  and the generators of the $0$-Hecke algebra}
The action of the simple transpositions $s_i$  $\in\mathfrak{S}_n$ on   weak compositions  
in $\mathbb{N}^n$, 
  induces an action of $\mathfrak{S}_n$ on  the polynomial ring $\mathbb{Z}[x_1,  \ldots,x_n]$  by considering weak compositions $\alpha\in\mathbb{N}^n$ as exponents of monomials $x^\alpha:=x_1^{\alpha_1}x_2^{\alpha_2}\cdots x_n^{\alpha_n}$ \cite{lascouxdraft}, and  defining $s_ix^\alpha:=x^{s_i\alpha}$ as the transposition
of $ x_i$ and $x_{i+1}$ in the monomial $x^\alpha$. If $f\in\mathbb{Z}[x_1,\ldots,x_n]$,  $s_if$ indicates the result of the action of $s_i$ in each monomial of $f$.
 For $i=1,\dots, n-1$, one defines the linear operators $\pi_i$,
 $\widehat\pi_i$ on $\mathbb{Z}[x_1, \ldots,x_n]$  by
\begin{equation}\label{demazureact}
 \pi_if=\frac{x_if-s_i(x_{i}f)}{x_i-x_{i+1}},~~~~~~
\widehat{\pi}_i f=(\pi_i-1)f=\pi_i f-f,
\end{equation}
where $1$ is the identity operator on $\mathbb{Z}[x_1,\ldots,x_n]$.
These operators are called  isobaric divided differences \cite{lascouxdraft}, and the first is the Demazure operator \cite{demazure} for the general linear Lie algebra $\mathfrak{gl}_n(\mathbb {C})$.
Isobaric divided difference  operators $\pi_i$ and $\widehat\pi_i$, $1\le i<n$,  \eqref{demazureact}, have an equivalent definition
\begin{equation}\label{demazureact1}
\pi_i(x_i^ax_{i+1}^bm)=\begin{cases} x_i^ax_{i+1}^bm+(\sum_{j=1}^{a-b}x_i^{a-j}x_{i+1}^{b+j})m,&\mbox{if $a>b$},\\
x_i^ax_{i+1}^bm,&\mbox{if $a=b$},\\
x_i^ax_{i+1}^bm-(\sum_{j=0}^{b-a-1}x_i^{a+j}x_{i+1}^{b-j})m,&\mbox{if $a<b$},
\end{cases}
\end{equation}
where $m$ is a monomial  not containing $x_i$ nor $x_{i+1}$. It follows from the definition that $\pi_i(f)=f$ and $\widehat\pi_i(f)=0$ if and only if $s_if=f$. They both satisfy the commutation and
the braid relations   \eqref{cox} of $\mathfrak{S}_n$,
and this
guarantees that, for any permutation $\sigma\in \mathfrak{S}_n$, there exists a well defined isobaric divided difference
$\pi_{\sigma}:=\pi_{i_N}\cdots\pi_{i_2}\pi_{i_1}$ and $\hat\pi_{\sigma}:=\hat\pi_{i_N}\cdots\hat\pi_{i_2}\hat\pi_{i_1}$, where $(i_N,\ldots ,{i_2},{i_1})$ is any reduced word of $\mathfrak{S}_n$. In addition, they satisfy the quadratic relations
$\pi_i^2=\pi_i$ and $\hat\pi_i^2=-\hat\pi_i.$

 The $0$-Hecke algebra $H_n(0)$ of $\mathfrak{S}_n$, a deformation of the group algebra of $\mathfrak{S}_n$,
is an associative
$\mathbb{C}$-algebra generated by $T_1,\ldots,T_{n-1}$ satisfying the commutation and the braid relations of the symmetric group $\mathfrak{S}_n$,
and the quadratic relation $T_i^2=T_i$ for $1\le i<n$. Setting $\widehat T_i:=T_i-1$, for $1\le i<n$,  one obtains another set of generators of the
$0$-Hecke algebra $H_n(0)$. The sets
  $\{T_\sigma: \sigma\in \mathfrak{S}_n\}$ and $\{\hat T_\sigma: \sigma\in \mathfrak{S}_n\}$ are both linear bases for $H_n(0)$, where $T_\sigma=T_{i_N}\cdots T_{i_2} T_{i_1}$ and $\hat T_{\sigma}:=\hat T_{i_N}\cdots\hat T_{i_2}\hat T_{i_1}$, for any reduced expression  $s_{i_N}\cdots s_{i_2}s_{i_1}$ in $\mathfrak{S}_n$ \cite{bourbaki}.
Since Demazure operators $\pi_i$ \eqref{demazureact} or bubble sort operators \eqref{bubblerelation} satisfy the same relations as $T_i$, and similarly for isobaric divided difference operators $\widehat\pi_i$ \eqref{demazureact} and $\widehat T_i$, the $0$-Hecke algebra $H_n(0)$ of $\mathfrak{S}_n$ may be viewed as an algebra of operators realised either by any of the two isobaric divided differences \eqref{demazureact}, or by bubble sort operators \eqref{bubblerelation}, swapping entries $i$ and $i+1$ in a weak composition $\alpha$, if $\alpha_i>\alpha_{i+1}$, and doing nothing, otherwise. Therefore,  the two  families
  $\{\pi_\sigma: \sigma\in \mathfrak{S}_n\}$ and $\{\hat\pi_\sigma:\sigma\in \mathfrak{S}_n\}$ are both linear bases for $H_n(0)$, and
 from the relation  $\widehat \pi_i=\pi_i-1$,  the change of basis from the first to the second is given by a sum over the Bruhat order in $\mathfrak{S}_n$, precisely \cite{lascouxgrothendick,ponsgrothendick},
\begin{eqnarray}\label{bruhatdec1}\pi_\sigma=\sum_{\theta\le \sigma}\hat\pi_\theta.
\end{eqnarray}


\subsection{Demazure characters, Demazure atoms and sorting operators}  Let $\lambda\in\mathbb{N}^n$ be a partition and $\alpha$ a weak composition in the $\mathfrak{S}_n$-orbit of $\lambda$.
  Write $\alpha=\sigma\lambda$, where  $\sigma$
  is a minimal length coset representative of $\mathfrak{S}_n/stab_\lambda$.
  The   key polynomial \cite{lascouxkeys,reiner} or  Demazure character \cite{demazure,joseph} in type $A$, corresponding to the dominant weight $\lambda$ and permutation $\sigma$,  is the polynomial  in $\mathbb{Z}[x_1, \ldots,x_n]$ indexed by the weak composition $\alpha\in\mathbb{N}^n$, defined by
\begin{equation}\label{key}
\kappa_{\alpha}:=\pi_{\sigma}x^\lambda,
\end{equation}
and the  standard basis \cite{lasschutz,lascouxkeys} or  Demazure atom \cite{masondemazure} is defined similarly,
\begin{equation}\label{keyhat}\widehat\kappa_{\alpha}:=\widehat\pi_{\sigma}x^\lambda.\end{equation}
Due to  \eqref{bruhatdec1}, the   identity \eqref{keyhat} consists of all monomials in $\kappa_{\alpha}$ which
do not appear in  $\kappa_{\beta}$ for any $\beta<\alpha$.
  Thereby,  key polynomials \eqref{key} are decomposed into Demazure atoms \cite{lascouxkeys,lascouxdraft},
\begin{equation}\label{bruhatdec2}\kappa_{\alpha}=\sum_{\beta\le
\alpha}\widehat\kappa_{\beta}.\end{equation}

Key polynomials $\{\kappa_{\alpha}:$ $\alpha\in$ $\mathbb{N}^n\}$ and Demazure atoms $\{\hat\kappa_{\alpha}:$ $\alpha$ $\in\mathbb{N}^n\}$ form  linear $\mathbb{Z}$-basis for $\mathbb{Z}[x_1, \ldots,x_n]$ \cite{reiner}. The change of basis from the first to the second is expressed in \eqref{bruhatdec2}.
 The operators $\pi_i$ act on key polynomials $\kappa_\alpha$ via elementary bubble sorting operators on the entries of the weak composition $\alpha$
 \cite{reiner},
\begin{equation}\label{dempro}\pi_i\kappa_\alpha=\begin{cases} \kappa_{s_i\alpha} & \mbox{if } \alpha_i>\alpha_{i+1} \\ \kappa_\alpha & \mbox{if } \alpha_i\le \alpha_{i+1} \end{cases} .\end{equation}
 This suggests the following recursive definition of key polynomials \cite{lascouxdraft}.
 For $\alpha\in \mathbb{N}^n$,
the key polynomial $\kappa_\alpha$ (resp. $\hat\kappa_\alpha$) is
$\kappa_\alpha=\hat\kappa_\alpha=x^\alpha, \;\mbox{if $\alpha$ is a partition.}$
Otherwise,
$\kappa_\alpha = \pi_i\kappa_{s_i\alpha}\,\mbox{(resp. $\hat\kappa_\alpha= $ $\hat\pi_i\hat\kappa_{s_i\alpha}$)},\; \mbox{if $\alpha_{i+1} > \alpha_i$}$.
The key polynomial $\kappa_\alpha$ is symmetric in $x_i$ and $x_{i+1}$ if and only if $\alpha_{i+1} \ge \alpha_i$. Thus it lifts the Schur polynomial $s_{\alpha^{+}}(x)$, $\kappa_\alpha=s_{\alpha^{+}}(x)$, when $\alpha_1\le\cdots\le\alpha_n$.
 Henceforth, \eqref{bruhatdec2} contains, as a special case, the decomposition of a Schur polynomial, into Demazure atoms, \begin{equation}\label{bruhatdecschur}s_{\alpha^+}(x)=\sum_{\alpha\in
\mathfrak{S}_n\alpha^+}\widehat\kappa_{\alpha}.\end{equation}

 \subsection{Crystals and combinatorial descriptions of Demazure operators 
 } In \cite{lascouxkeys} Lascoux and Sch\"utzenberger have given a combinatorial version for Demazure operators $\pi_i$ and $\hat\pi_i$ in terms of  crystal (or coplactic) operators $f_i$, $e_i$ to produce a
   crystal graph on  $\mathfrak{B}^\lambda$, the set of SSYTs with entries $\le n$ and shape $\lambda$ \cite{kashi,kashiwaraq,thibon}.
A SSYT can be uniquely recovered from its column word. The action of the crystal
operators $f_i$ and $e_i$, $1\le i<n$, on $T\in\mathfrak{B}^\lambda$ is described by the usual parentheses matching on the column word of $T$, and we refer the reader for details to \cite{kwon,thibon}.
 For convenience, we extend $f_i$ and $e_i$ to $\mathfrak{B}^\lambda\cup \{0\}$ by setting them to map $0$ to $0$.

 Kashiwara and Nakashima \cite{kashiwaraq,kashinakashima}
 have given to $\mathfrak{B}^\lambda$ a $U_q(\mathfrak{gl}_n)$- crystal structure. We view crystals as special graphs. The crystal
graph  on $\mathfrak{B}^\lambda$
is   a coloured
directed graph whose vertices  are
the elements of $\mathfrak{B}^\lambda$, and the edges are
coloured with a colour $i$,  for each pair of crystal operators
$f_i,~ e_i$, such that there exists a coloured $i$-arrow from the
vertex $T$ to $T'$ if and only if $f_i(T)=T'$, equivalently,
$e_i(T')=T$. We refer to  \cite{kashiwaraoncrystal,hongkang,lecouvey} for details.
Start with  the Yamanouchi tableau  $Y:=key(\lambda)$ and apply  all  the crystal operators
$f_{i}$'s    until each unmatched $i$ has been
converted to $i+1$, for $1\le i<n$ \cite{kashiwaraq,kashiwaraoncrystal}. See Example~\ref{crystal}.
  The resulting set is $\mathfrak{B}^\lambda$  whose elements index  basis vectors for the representation of the quantum group  $U_q(\mathfrak{gl}_n)$ with highest weight $\lambda$. From the definition of this graph,   in each vertex there is at most one incident arrow of colour $i$, and at most one outgoing arrow of colour $i$.  Hence, for any $i$, $1\le i<n$, the crystal graph on $\mathfrak{B}^\lambda$ decompose $\mathfrak{B}^\lambda$ into disjoint
     connected components of colour $i$, $P_1\overset{i}\rightarrow\cdots\overset{i}\rightarrow P_k$, called $i$-strings, having lengths $k-1\ge 0$. A SSYT $P_1$, 
     satisfying $e_i(P_1)=0$,  is said to be the head of the $i$-string, and, in the case of $f_i(P_k)=0$, $P_k$ is called the end of the $i$-string.
Given $\alpha$  in the  $\mathfrak{S}_n\lambda$,
the Demazure crystal  $\mathfrak{B}_\alpha$
 is viewed as a certain subgraph of the crystal  $\mathfrak{B}^\lambda$  which can be defined inductively  \cite{kashiwara,littel} as $ \mathfrak{B}_\alpha=\{Y\}$ if $\alpha=\lambda$, otherwise
 \begin{equation}\label{demazuregraph}\mathfrak{B}_\alpha=\{f_i^k(T): T\in\mathfrak{B}_{s_i\alpha}, k\ge 0, e_i(T)=0\}\setminus\{0\},
 \quad\mbox{ if $\alpha_{i+1}>\alpha_i$}.\end{equation}
 (When $\alpha$ is the reverse of $\lambda$, one has $\mathfrak{B}_{\omega\lambda}=\mathfrak{B}^\lambda$.)
The vertices of this subgraph index  basis vectors of the Demazure module  $V_\sigma(\lambda)$ where $\sigma$ is a minimal length coset representative modulo the stabiliser of $\lambda$, such that $\sigma\lambda=\alpha$.
In fact $\mathfrak{B}_{\alpha}$ \eqref{demazuregraph} is well defined, it does not depend on the reduced expression for $\sigma$. More generally, write $\alpha=s_{i_N}\dots s_{i_2} s_{i_1}
\lambda$, with $(i_N,\dots, i_2, i_1)$ a reduced word, then
 apply the crystal operator
$f_{i_1}$ to  $Y$ until each unmatched $i_1$ has been
converted to $i_{1}+1,$ then apply similarly
$f_{i_{2}}$ to each of the previous Young tableaux until each unmatched $i_2$ has been
converted to $i_{2}+1$, and continue this procedure with $f_{i_3},\ldots,$  $f_{i_N}$.  Therefore,
   $\mathfrak{B}_\alpha=$$\{f_{i_N}^{m_N} $ $\ldots $ $f_{i_1}^{m_1}(Y):m_k\ge 0\}\setminus\{0\}$.

 Let $T\in\mathfrak{B}^\lambda$, and  $f_{s_i}(T):=\{f^{m}_i(T): m\ge 0\}\setminus\{0\}$. (If $f_i(T)=0$, $f_{s_i}(T)=\{T\}$.)
     If $P$ is the head of an $i$-string  $S\subseteq\mathfrak{B}^\lambda$, $S= f_{s_i}(P)$. We abuse notation and say
      the Demazure operator $\pi_i$ \eqref{demazureact1} sends  the head of an $i$-string to the sum of all
elements of the string \cite{lascouxkeys,kashiwara},
    \begin{equation}\label{demazureoperator}\pi_i(x^P)=\sum_{T\in S}x^T
     ,\;\;\mbox{and}\;\;
     \pi_i\bigg(\sum_{T\in S}x^T\bigg)=\pi_i(x^P).\end{equation}
   If $\beta\le\alpha$, then $\mathfrak{B}_{\beta}\subseteq \mathfrak{B}_{\alpha}$. Let $s_i\alpha<\alpha$, equivalently, $\alpha_i<\alpha_{i+1}$. For any $i$-string $S\subseteq\mathfrak{B}^\lambda$, either   $\mathfrak{B}_{s_i\alpha}\cap S=\mathfrak{B}_{\alpha}\cap S$ is empty, or  $\mathfrak{B}_{s_i\alpha}\cap S=$$\mathfrak{B}_{\alpha}\cap S=S$, or $\mathfrak{B}_{s_i\alpha}\cap S$ is only the head of $S$ in which case  $S\subseteq\mathfrak{B}_\alpha$.  Since $\mathfrak{B}^\lambda$ is the disjoint union of $i$-strings, from these string properties, and \eqref{demazureoperator}, one has for any $i$-string $S$
  \begin{equation}\label{demazureoperator2}\sum_{T\in{\mathfrak{B}}_{\alpha}\cap S}x^T=\pi_i\bigg(\sum_{T\in{\mathfrak{B}}_{s_i\alpha}\cap S}x^T\bigg); \;\mbox{and}\;\sum_{T\in{\mathfrak{B}}_{\alpha}}x^T=\pi_i\bigg(\sum_{T\in{\mathfrak{B}}_{s_i\alpha}}x^T\bigg).\end{equation}
  Henceforth,
    $\kappa_{\alpha}=\pi_i\kappa_{s_i\alpha}$, if $\alpha_i<\alpha_{i+1}$, and $\kappa_{\alpha}=\pi_{i_N}\cdots\pi_{i_1}x^\lambda$ for any reduced word $({i_N},\ldots,{i_1})$  such that $s_{i_N}\cdots s_{i_1}\lambda=\alpha$.
  \begin{ex}\label{crystal}
 The Demazure crystal  $\mathfrak{B}_{s_2s_1\lambda}$ with
 $\lambda=(3,1,0)$ is shown with thick edges while the rest of the crystal  graph $\mathfrak{B}^{\lambda}$ is shown with thinner lines.
 The $1$ and $2$-strings are represented in black and red colours respectively. The key polynomial $\kappa_{(1,0,3)}$ is calculated using the thick coloured strings in the Demazure crystal graph  ${\mathfrak{B}}_{s_2s_1(3,1,0)}$, $\kappa_{(1,0,3)}=\pi_2\pi_1x^{(3,1,0)}$
$=\pi_2 (x^{(3,1,0)+x^{(2,2,0)}+x^{(1,3,0)}})$
$=x^{(3,1,0)}+x^{(2,2,0)}+x^{(1,3,0)}+x^{(3,0,1)}+x^{(2,1,1)}$
$+x^{(2,0,2)}+x^{(1,2,1)}+x^{(1,1,2)}+x^{(1,0,3)}.$

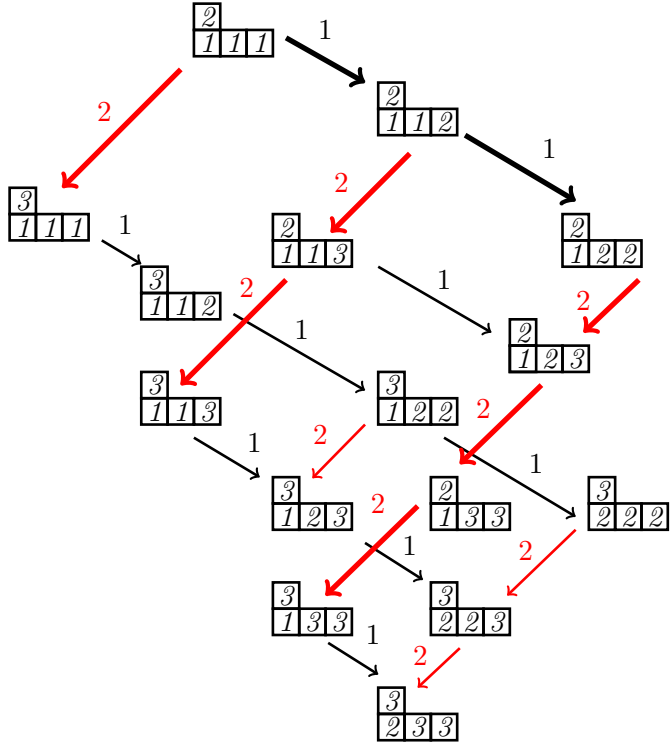
\begin{tikzpicture}[scale=0.7] \draw[line width=1pt] (0,9.5)
rectangle (0.5,10);\draw[line width=1pt] (0.5,9.5) rectangle
(1,10);\draw[line width=1pt] (1,9.5) rectangle (1.5,10);\draw[line
width=1pt] (0,10) rectangle (0.5,10.5);
\draw[line width=1pt] (2.5,6) rectangle (3,6.5);\draw[line
width=1pt] (3,6) rectangle (3.5,6.5);\draw[line width=1pt] (3.5,6)
rectangle (4,6.5);\draw[line width=1pt] (2.5,6.5) rectangle (3,7);
\draw[line width=1pt] (2.5,8) rectangle (3,8.5);\draw[line
width=1pt] (3,8) rectangle (3.5,8.5);\draw[line width=1pt] (3.5,8)
rectangle (4,8.5);\draw[line width=1pt] (2.5,8.5) rectangle (3,9);
\draw[line width=1pt] (3.5,13) rectangle (4,13.5);\draw[line
width=1pt] (4,13) rectangle (4.5,13.5);\draw[line width=1pt]
(4.5,13) rectangle (5,13.5);\draw[line width=1pt] (3.5,13.5)
rectangle (4,14);
\draw[line width=1pt] (5,2) rectangle (5.5,2.5);\draw[line
width=1pt] (5.5,2) rectangle (6,2.5);\draw[line width=1pt] (6,2)
rectangle (6.5,2.5);\draw[line width=1pt] (5,2.5) rectangle (5.5,3);
\draw[line width=1pt] (5,4) rectangle (5.5,4.5);\draw[line
width=1pt] (5.5,4) rectangle (6,4.5);\draw[line width=1pt] (6,4)
rectangle (6.5,4.5);\draw[line width=1pt] (5,4.5) rectangle (5.5,5);
\draw[line width=1pt] (5,9) rectangle (5.5,9.5);\draw[line
width=1pt] (5.5,9) rectangle (6,9.5);\draw[line width=1pt] (6,9)
rectangle (6.5,9.5);\draw[line width=1pt] (5,9.5) rectangle
(5.5,10);
\draw[line width=1pt] (7,0) rectangle (7.5,0.5);\draw[line
width=1pt] (7.5,0) rectangle (8,0.5);\draw[line width=1pt] (8,0)
rectangle (8.5,0.5);\draw[line width=1pt] (7,0.5) rectangle (7.5,1);
\draw[line width=1pt] (7,6) rectangle (7.5,6.5);\draw[line
width=1pt] (7.5,6) rectangle (8,6.5);\draw[line width=1pt] (8,6)
rectangle (8.5,6.5);\draw[line width=1pt] (7,6.5) rectangle (7.5,7);
\draw[line width=1pt] (7,11.5) rectangle (7.5,12);\draw[line
width=1pt] (7.5,11.5) rectangle (8,12);\draw[line width=1pt]
(8,11.5) rectangle (8.5,12);\draw[line width=1pt] (7,12) rectangle
(7.5,12.5);
\draw[line width=1pt] (8,2) rectangle (8.5,2.5);\draw[line
width=1pt] (8.5,2) rectangle (9,2.5);\draw[line width=1pt] (9,2)
rectangle (9.5,2.5);\draw[line width=1pt] (8,2.5) rectangle (8.5,3);
\draw[line width=1pt] (8,4) rectangle (8.5,4.5);\draw[line
width=1pt] (8.5,4) rectangle (9,4.5);\draw[line width=1pt] (9,4)
rectangle (9.5,4.5);\draw[line width=1pt] (8,4.5) rectangle (8.5,5);
\draw[line width=1pt] (9.5,7) rectangle (10,7.5);\draw[line
width=1pt] (9.5,7) rectangle (10.5,7.5);\draw[line width=1pt]
(10.5,7) rectangle (11,7.5);\draw[line width=1pt] (9.5,7.5)
rectangle (10,8);
\draw[line width=1pt] (11,4) rectangle (11.5,4.5);\draw[line
width=1pt] (11.5,4) rectangle (12,4.5);\draw[line width=1pt] (12,4)
rectangle (12.5,4.5);\draw[line width=1pt] (11,4.5) rectangle
(11.5,5);
\draw[line width=1pt] (10.5,9) rectangle
(11,9.5);\draw[line width=1pt] (11,9) rectangle
(11.5,9.5);\draw[line width=1pt] (11.5,9) rectangle
(12,9.5);\draw[line width=1pt] (10.5,9.5) rectangle (11,10);
\node at (0.75,9.75){1};\node at (1.25,9.75){1}; \node at
(0.25,9.75){1}; \node at (0.25,10.25){3};
\node at (2.75,6.25){1};\node at (3.25,6.25){1}; \node at
(3.75,6.25){3}; \node at (2.75,6.75){3}; \node at
(2.75,8.25){1};\node at (3.25,8.25){1}; \node at (3.75,8.25){2};
\node at (2.75,8.75){3}; \node at (3.75,13.25){1};\node at
(4.25,13.25){1}; \node at (4.75,13.25){1}; \node at (3.75,13.75){2};
\node at (5.25,2.25){1};\node at (5.75,2.25){3}; \node at
(6.25,2.25){3}; \node at (5.25,2.75){3};
 \node at(5.25,4.25){1};\node at (5.75,4.25){2}; \node at (6.25,4.25){3};
\node at (5.25,4.75){3}; \node at (5.25,9.25){1};\node at
(5.75,9.25){1}; \node at (6.25,9.25){3}; \node at (5.25,9.75){2};
\node at (7.25,0.25){2};\node at (7.75,0.25){3}; \node at
(8.25,0.25){3}; \node at (7.25,0.75){3}; \node at
(7.25,6.25){1};\node at (7.75,6.25){2}; \node at (8.25,6.25){2};
\node at (7.25,6.75){3}; \node at (7.25,11.75){1};\node at
(7.75,11.75){1}; \node at (8.25,11.75){2}; \node at (7.25,12.25){2};
\node at (8.25,2.25){2};\node at (8.75,2.25){2}; \node at
(9.25,2.25){3}; \node at (8.25,2.75){3}; \node at
(8.25,4.25){1};\node at (8.75,4.25){3}; \node at (9.25,4.25){3};
\node at (8.25,4.75){2};
\node at (9.75,7.25){1};\node at (10.25,7.25){2}; \node at
(10.75,7.25){3}; \node at (9.75,7.75){2};
 \node at(11.25,4.25){2};\node at (11.75,4.25){2}; \node at (12.25,4.25){2};
\node at (11.25,4.75){3}; \node at(10.75,9.25){1};\node at
(11.25,9.25){2}; \node at (11.75,9.25){2}; \node at (10.75,9.75){2};
\draw[arrows=->,line width=2pt] (5.25,13.35) --
(6.75,12.5);\draw[arrows=->,line
width=2pt] (8.65,11.5) -- (10.75,10.25); \node at (6,13.5){$1$};
\node at (10.25,11.25){$1$};
\draw[arrows=->,line width=1pt] (7,9) -- (9.15,7.75); \node at
(8.25,8.75){$1$};
\draw[arrows=->,line width=1pt] (1.75,9.50) --
(2.5,9.05);\draw[arrows=->,line
width=1pt] (4.25,8.10) -- (6.75,6.65);\draw[arrows=->,line width=1pt]
(8.25,5.75) -- (10.75,4.25);\node at (2.2,9.8){$1$};\node at
(5.55,7.8){$1$};\node at (10,5.25){$1$};
\draw[arrows=->,line width=1pt] (3.5,5.75) --
(4.75,5);\draw[arrows=->,line width=1pt]
(6.75,3.75) -- (7.85,3.05);\node at (4.65,5.65){$1$};\node at
(7.60,3.7){$1$};
\draw[arrows=->,line width=1pt] (6.05,1.85) -- (7,1.25);\node at
(6.9,2){$1$};
\draw[arrows=->,line width=1pt, color=red] (10.75,4) --
(9.45,2.75);\draw[arrows=->,line
width=1pt, color=red] (8.55,1.75) -- (7.75,1);\node[color=red] at
(9.8,3.6){$2$};\node[color=red] at (7.8,1.6){$2$};
\draw[arrows=->,line width=1pt, color=red] (6.75,6) -- (5.75,4.98);
\node[color=red] at (5.9,5.8){$2$};
\draw[arrows=->,line width=2pt, color=red] (11.95,8.75) --
(10.90,7.75);\draw[arrows=->,line width=2pt, color=red] (10.10,6.75) --
(8.55,5.25);\draw[arrows=->,line width=2pt, color=red] (7.75,4.45) --
(6,2.75); \node[color=red] at (10.9,8.35){$2$}; \node[color=red]
at (9,6.3){$2$}; \node[color=red] at (7,4.5){$2$};
\draw[arrows=->,line width=2pt, color=red] (7.6,11.15) --
(6.1,9.65);\draw[arrows=->,line width=2pt, color=red] (5.25,8.75) --
(3.25,6.75); \node[color=red] at (6.3,10.6){$2$}; \node[color=red]
at (4.5,8.6){$2$};
\draw[arrows=->,line width=2pt, color=red] (3.25,12.75) -- (1,10.5);
\node[color=red] at (1.8,11.95){$2$};
\end{tikzpicture}\\
\end{ex}
 Set $\widehat{\mathfrak{B}}_{\alpha}:=$
  $\mathfrak{B}_{\alpha}\setminus \bigcup_{\beta<\alpha} \mathfrak{B}_{\beta}$. Then ${\mathfrak{B}_\alpha=\biguplus_{\beta\le\alpha}\widehat{\mathfrak{B}}_\beta}$.
 In Example~\ref{crystal}, with $\alpha=(1,0,3)=$ $s_2s_1(3,1,0)$, the  component  $\widehat{\mathfrak{B}}_{s_2s_1(3,1,0)}=$ $\mathfrak{B}_{s_2s_1(3,1,0)}$ $\setminus $ $(\mathfrak{B}_{s_1(3,1,0)}\cup$ $ \mathfrak{B}_{s_2(3,1,0)})$ consists of
   the two lowest thick red strings, starting in the thick black string, minus their heads.
Lascoux and Sch\"utzenberger have  characterised  the SSYTs in
$\widehat{\mathfrak{B}}_{\alpha}$ \cite{lascouxkeys} as those whose right key is  $key(\alpha)$, precisely the unique key tableau in $\widehat{\mathfrak{B}}_{\alpha}$.
The Demazure crystal $\mathfrak{B}_{\alpha}$ consists of all Young tableaux in $\mathfrak{B}^{\lambda}$ with right key bounded by $key(\alpha)$.
\begin{thm}[\sc Lascoux, Sch\"utzenberger \cite{lasschutz,lascouxkeys}]
The Demazure atom $\widehat\kappa_{\sigma\lambda}=\hat\pi_\sigma x^\lambda$ is the sum of the weight monomials of all SSYTs with entries $\le n$ whose right key is equal to $key(\sigma\lambda)$, with $\sigma$ a minimal length coset representative modulo the stabiliser of $\lambda$.
    \end{thm}
 We may put together the three combinatorial interpretations of Demazure characters and Demazure atoms
   $$\displaystyle{\hat\kappa_\alpha=\sum_{T\in\widehat{\mathfrak{B}}_{\alpha}}x^T=
    \sum_{\begin{smallmatrix}T\in SSYT_n\\K_+(T)=key(\alpha)\end{smallmatrix} }x^T=\sum_{\begin{smallmatrix}F\in SSAF_n\\sh(F)=\alpha\end{smallmatrix}}x^F, }$$ \begin{equation}\kappa_\alpha=\sum_{T\in{\mathfrak{B}}_{\alpha}}x^T=\sum_{\begin{smallmatrix}T\in SSYT_n\\K_+(T)\le key(\alpha)\end{smallmatrix}}x^T=\sum_{\begin{smallmatrix}F\in SSAF_n\\sh(F)\le \alpha\end{smallmatrix}}x^F.\nonumber\end{equation}
 In particular, the sum of the weight  monomials  over all
crystal graph  $\mathfrak{B}^\lambda$ gives the Schur polynomial
$s_\lambda$,
 and thus Demazure atoms
decompose Schur polynomials in $\mathbb{Z}[x_1,\ldots,x_n]$.

\section{Expansions of Cauchy kernels over truncated staircases}
\label{sec:kernels}
\subsection{Cauchy identity and Lascoux's non-symmetric Cauchy kernel expansions}\label{sec:kernels1}
Given $n\in \mathbb{N}$ positive, let $m$ and $k$ be fixed positive integers where $1\le m\le n$ and $1\le k\le n$. 
 Let $x=(x_1,\ldots, x_n)$ and $y=(y_1,\ldots,y_n)$ be two sequences of indeterminates.
The well-known Cauchy identity  expresses the
Cauchy kernel $\prod_{i=1}^{k}\prod_{j=1}^m(1-x_iy_j)^{-1}$, symmetric in $x_i$ and $y_j$ separately, as a sum of
products of Schur polynomials $s_{\mu^+}$ in $(x_1,x_2,\dots, x_k)$ and
$(y_1,y_2,\dots,y_m),$
\begin{equation}\label{aaa1}
\prod_{\begin{smallmatrix}(i,j)\in
(m^k)\end{smallmatrix}}(1-x_iy_j)^{-1}
=\prod_{i=1}^{k}\prod_{j=1}^{m}(1-x_iy_j)^{-1}=\sum_{\mu^+ 
}s_{\mu^+}(x_1,\ldots,x_k)s_{\mu^+}(y_1,\ldots,y_m),
\end{equation}
\noindent over all partitions $\mu^+$   of length
$\le\min\{k,m\}.$ Using either the RSK  correspondence  \cite{knuth}
or the $\Phi$ correspondence,
  the Cauchy formula \eqref{aaa1}
   can be interpreted as a bijection between  monomials, on the left hand side, and pairs of SSYTs or SSAFs on the right. As the basis of key polynomials lifts the Schur polynomials w.r.t. the same list of indeterminates, and key polynomials decompose into Demazure atoms \eqref{bruhatdecschur}, the expansion \eqref{aaa1} can  also be expressed in the two bases of key polynomials.  Assuming $k\le m$, we may write \eqref{aaa1} as
   \begin{eqnarray}\nonumber
&&\sum_{\mu^+\in \mathbb{N}^k}s_{\mu^+}(x_1,\ldots,x_k)s_{(\mu^+,0^{m-k})}(y_1,\ldots,y_m)=\sum_{\mu^+\in \mathbb{N}^k}\sum_{\begin{smallmatrix}\mu\in
\mathfrak{S}_k\mu^+
 \end{smallmatrix}}\widehat{\kappa}_{\mu}(x)\kappa_{(0^{m-k},\omega\mu^+)}(y),\nonumber\\
 &=&\sum_{\mu\in \mathbb{N}^k}\widehat{\kappa}_{\mu}(x)\kappa_{(0^{m-k},\omega\mu^+)}(y).\label{last}
\end{eqnarray}
(Since we are dealing with two sequences of indeterminates $x$ and $y$, it is convenient to write $\kappa_\alpha(x)$ and $\kappa_\alpha(y)$ instead of $\kappa_\alpha$. Similarly for Demazure atoms.)

    We now replace in the Cauchy kernel the rectangle $(m^k)$  by the truncated staircase
$\lambda=(m^{n-m+1},m-1,\dots,n-k+1),$ with  $1\leq m\leq n$, $~1\leq k\leq n,$ and
$n+1\leq m+k$, as shown in Figure~\ref{fig:trunc}.
 If $n+1=m+k$, we recover the rectangle shape $(m^k)$. When $m=n=k,$ one has the
staircase partition $\lambda=(n,n-1,\dots,2,1)$, that is, the cells $(i,j)$ in the NW-SE diagonal of the square diagram $(n^n)$ and
below it, and thus $(i,j)\in\lambda$
if and only if $i+j\leq n+1$.
Lascoux has given in \cite{lascouxcrystal}, and with Fu, in \cite{fulascoux}, the following  expansion for the non-symmetric Cauchy kernel over staircases,
\begin{equation}\label{cau21}\prod_{\begin{smallmatrix}i+j\leq{n+1}\\
1\le i,\,j\le n\end{smallmatrix} }(1-x_iy_j)^{-1}=\sum_{\nu\in
\mathbb{N}^{n}}\widehat{\kappa}_{\nu}(x)\kappa_{\omega
\nu}(y),\end{equation} where   $\hat\kappa$  and $\kappa$  indicate the two
families of key polynomials in $x$ and $y$ respectively, and $\omega$ is the longest permutation
of $\mathfrak{S}_n$.

In \cite{lascouxcrystal}, Lascoux extends \eqref{cau21} to an expansion of $\prod_{(i,j)\in \lambda}
(1-x_iy_j)^{-1}$, over any Ferrers shape $\lambda$, as follows.  Given a Ferrers shape $\lambda$, consider  $\rho  := (t, t-1,\dots, 1)$,  the biggest staircase  contained in
$\lambda$, and a pair of permutations $\sigma(\lambda, NW)$ and $\sigma(\lambda, SE)$ encoding the  cells in  a NW and SE parts of the skew-diagram $\lambda/\rho$, the diagram consisting of the cells in $\lambda$ not in $\rho$. To define such a pair of permutations, one takes an arbitrary cell in the staircase $(t + 1,t,\dots , 1)$ which does not belong to $\lambda$. The SW-NE diagonal passing through this cell
cuts the skew-diagram of $\lambda/\rho$, into two pieces that are called
 the North-West (NW) part and the South-East (SE)  part of
$\lambda/\rho$. Fill each cell of row $r\ge 2$ of the NW part with the number
$r-1$. Similarly, fill each cell of column $c \ge 2$ of the SE part with the
number $c - 1$. Reading the columns of the NW part, from right to
left, top to bottom, and interpreting $r$ as the simple transposition $s_r$, gives a
reduced decomposition of the permutation $\sigma(\lambda,NW)$; similarly, reading rows  of the SE part,
 from top to bottom, and from right to left,  gives the
permutation $\sigma(\lambda, SE)$.
\begin{thm}[{\sc Lascoux, \cite[Theorem~7]{lascouxcrystal}}]
Let $\lambda$ be a partition in $\mathbb{N}^n$, $\rho(\lambda)= (t, t-1,$ $\dots, 1)$ the maximal staircase
contained in the diagram of $\lambda$,  and
 $\sigma(\lambda, NW),$ $\sigma(\lambda, SE)$  the two permutations obtained by
cutting the diagram of $\lambda/\rho$ as explained above. Then
\begin{equation}\label{ext}
\prod_{(i,j)\in \lambda}
(1-x_iy_j)^{-1}=\sum_{\mu\in \mathbb{N}^{t}}(\pi
_{\sigma(\lambda,NW)}\widehat{\kappa}_{\mu}(x))(\pi
_{\sigma(\lambda,SE)}\kappa_{\omega \mu}(y)).\end{equation}
\end{thm}
 For our truncated staircase shape $\lambda$, Figure~\ref{fig:trunc}, if $1\le k\le m\le n$,  $\lambda=(m^{n-m+1},
m-1,\dots,n-k+1),$ where  $n-k\le m-1,$  $\rho=(k,k-1,\ldots,1)$, and the cell $(k+1,1)$, on the top of the first column of $\lambda$, does not belong to $\lambda$. In this case, the NW piece of $\lambda/\rho$ is empty, thus  $\sigma(\lambda,NE)=id$, and the SE piece consists of all cells in $\lambda/\rho$. In  Figure~\ref{fig2:word}, the row reading word,  top to bottom   and right to left, defines the reduced word \begin{equation}\label{SE}{\sigma(\lambda,SE)}= \prod_{i=1}^{k-(n-m)-1}(s_{i+n-k-1}\dots s_i) \prod_{i=0}^{n-m}(s_{m-1}\dots s_{k-(n-m)+i}).\end{equation}

\begin{figure}[here]
\begin{center}
\begin{tikzpicture}[scale=0.68]
\filldraw[color=green!15] (-1,0) rectangle (8,1);
\filldraw[color=green!15] (-1,1) rectangle (8,2);
\filldraw[color=green!15] (-1,2) rectangle (8,3);
\filldraw[color=green!15] (-1,3) rectangle (7,4);
\filldraw[color=green!15] (-1,4) rectangle (6,5);
\draw[line width=1pt] (-1,0) rectangle (8,5); \draw[line width=1pt] (0,0)
-- (0,5);
 \draw[line width= 1pt]
(1,0) -- (1,5);
 \draw[line width= 1pt] (2,0) -- (2,5);
\draw[line width= 1pt] (3,0) -- (3,5); \draw[line width= 1pt] (4,0)
-- (4,5);\draw[line width= 1pt] (5,0) -- (5,5);\draw[line width=
1pt] (6,0) -- (6,5);\draw[line width=
1pt] (7,0) -- (7,5);
 \draw[line width= 1pt] (-1,1)-- (8,1);
\draw[line width= 1pt] (-1,2)-- (8,2);
 \draw[line width= 1pt] (-1,3)--(8,3);
 \draw[line width= 1pt] (-1,4)--(8,4);
 \draw[line width= 2 pt, red ] (3.97,0)--(3.97,1);
 \draw[line width= 2 pt, red ] (4,1)--(3,1);
  \draw[line width= 2 pt, red ] (3,1)--(3,2);
   \draw[line width= 2 pt, red ] (3,2)--(2,2);
    \draw[line width= 2 pt, red ] (2,2)--(2,3);
     \draw[line width= 2 pt, red ] (2,3)--(1,3);
      \draw[line width= 2 pt, red ] (1,3)--(1,4);
       \draw[line width= 2 pt, red ] (1,3.95)--(0,3.95);
       \draw[line width= 2 pt, red ] (0,3.95)--(0,4.95);
       \draw[line width= 2 pt, red ] (-1,4.95)--(0,4.95);
  \draw[line width= 2 pt, green ] (8,0)--(8,3);
\draw[line width= 2 pt, green ] (8,3)--(7,3);
 \draw[line width= 2
pt, green ] (7,3)--(7,4); \draw[line width= 2 pt, green ]
(6,4.03)--(7,4.03);\draw[line width= 2 pt, green ]
(-1,5.03)--(6,5.03);\draw[line width= 2 pt, green ]
(6,4.03)--(6,5.03); \footnotesize
\node at (0.5, 4.5) {$1$}; \node at (1.5, 4.5) {$2$};
\node at (1.5, 3.5) {$\ddots$};  \node at (6.5,
3.5) {$\ddots$};\node at (5.5,
4.5) { n-k};  \node at (3.5, 1.5) {$\ddots$};
\node at (3.5,
2.5) { $\ddots$};  \node at (2.5, 4.5) { $\ddots$};
 \node at (2.5, 2.7) { k-};  \node at (2.5, 2.3) {\scriptsize n+m};  \node
at (4.5, 0.5) { k};
\node at (5.5, 0.5) {$\ldots$};
\node at (7.5, 0.5) { m-1};\node at (7.5, 1.5) {$\vdots$};\node
at (7.5,
2.5) { m-1};
 \draw[arrows=<->,line width=1 pt] (-0.8,-0.5)--(8.0,-0.5); \node at
(4,-1){$m$};
  \draw[arrows=<->,line width=1 pt] (-1.5,0.2)--(-1.5,5.0); \node at
(-2,2.5){$k$};
\end{tikzpicture}
\end{center}
\caption{ The labels in $\lambda/\rho$ indicate the column index of $\lambda$ minus one. The reading word, from  right to left and from the top to bottom, defines the reduced word  ${\sigma(\lambda,SE)}$ \eqref{SE}.}
\label{fig2:word}
\end{figure}
Similarly, in Figure~\ref{fig:trunc}, if $n\ge k\ge m\ge 1$, then  $\rho=(m,m-1,\ldots,1)$, and the cell $(1,m+1)$ immediately after to the end of the first row of $\lambda$, does not belong to $\lambda$. Thus the SE piece of $\lambda/\rho$ is empty,  $\sigma(\lambda,SE)=id$, and the NW piece consists of all cells in $\lambda/\rho$.  Recall that $\overline\lambda$, the conjugate partition of $\lambda$, is the transpose of the Ferrers diagram $\lambda$, and notice that  $\sigma(\overline\lambda,SE)=\sigma(\lambda,NW)$.
 Therefore, the  formula \eqref{ext} is translated  to
\begin{equation}\label{ext1}
\prod_{\begin{smallmatrix}(i,j)\in \lambda\\
k\le m
\end{smallmatrix}}
(1-x_iy_j)^{-1}=\sum_{\mu\in \mathbb{N}^{k}}\widehat{\kappa}_{\mu}(x)(\pi
_{\sigma(\lambda,SE)}\kappa_{\omega \mu}(y));\end{equation}
\begin{equation}\label{ext2}\prod_{\begin{smallmatrix}(i,j)\in \lambda\\
m\le k\end{smallmatrix}}
(1-x_iy_j)^{-1}=\sum_{\mu\in \mathbb{N}^{m}}(\pi
_{\sigma(\lambda,NW)}\widehat{\kappa}_{\mu}(x))\kappa_{\omega \mu}(y).\end{equation}
Indeed \eqref{ext2} is just  \eqref{ext1}, with $x$ and $y$ swapped, followed by the change of basis  \eqref{bruhatdec2} from  Demazure characters to Demazure atoms, where we also use the linearity of Demazure operators. Then we have
{\allowdisplaybreaks
\begin{align} \nonumber  \prod_{\begin{smallmatrix}(i,j)\in
\lambda\\
m\le k\end{smallmatrix}}(1-x_iy_j)^{-1}&=\prod_{\begin{smallmatrix}(j,i)\in
\overline\lambda\\
\nonumber
m\le k\end{smallmatrix}}(1-x_iy_j)^{-1}=\sum_{\mu\in \mathbb{N}^m}
\widehat{\kappa}_{\mu}(y)\pi_{\sigma(\overline\lambda,SE)}
\kappa_{\omega\mu}(x)\\
\nonumber
&=\sum_{\mu\in \mathbb{N}^m}
\widehat{\kappa}_{\mu}(y)\pi_{\sigma(\lambda,NW)}
\kappa_{\omega\mu}(x)=\sum_{\mu\in \mathbb{N}^m}
\widehat{\kappa}_{\mu}(y)\pi_{\sigma(\lambda,NW)}
\sum_{\beta\le \omega\mu}\widehat\kappa_{\beta}(x)\\
\nonumber
&=\sum_{\mu\in \mathbb{N}^m}\sum_{\begin{smallmatrix}\beta\in \mathbb{N}^m\\ \beta\le \omega\mu\end{smallmatrix}}
\widehat{\kappa}_{\mu}(y)\pi_{\sigma(\lambda,NW)}\widehat\kappa_{\beta}(x)
=
\sum_{\beta\in \mathbb{N}^m}\sum_{\begin{smallmatrix}\mu\in \mathbb{N}^m\\ \mu\le \omega\beta\end{smallmatrix}}
\widehat{\kappa}_{\mu}(y)\pi_{\sigma(\lambda,NW)}\widehat\kappa_{\beta}(x)\nonumber\\
&=\sum_{\beta\in \mathbb{N}^m}\pi_{\sigma(\lambda,NW)}\widehat\kappa_{\beta}(x)\sum_{\begin{smallmatrix}\mu\in \mathbb{N}^m\\ \mu\le \omega\beta\end{smallmatrix}}
\widehat{\kappa}_{\mu}(y)=
\sum_{\beta\in \mathbb{N}^m}\pi_{\sigma(\lambda,NW)}\widehat\kappa_{\beta}(x)\kappa_{\omega\beta}(y).
\label{conjugate}
\end{align}}

Next we give a bijective proof of \eqref{ext1}, which amounts to computing the indexing weak composition of the Demazure character $\pi
_{\sigma(\lambda,SE)}~\kappa_{\omega \mu}(y)$, by making explicit the Young tableaux in the Demazure crystal.
\subsection{Our expansions}
We  now use the bijection in Theorem~\ref{maint}
 to give an expansion of the non-symmetric Cauchy kernel for  the shape $\lambda=(m^{n-m+1}, m-1,\dots,n-k+1),$ where $1\leq
m\leq n$, $1\leq k\leq n,$ and $n+1\leq m+k$,  which includes, in particular, the rectangle  \eqref{aaa1},  the staircase  \eqref{cau21}, and implies
the  truncated staircases \eqref{ext1}.

The generating function for the multisets of ordered pairs of positive
integers $\{(a_1, b_1), $ $(a_2, b_2),$ $\ldots ,$ $ (a_r, b_r)\}$, $r\ge 0$,
 where $(a_i,b_i)$ $\in \lambda$, that is, $~ a_i+b_i$ $\leq n+1$, $~1\leq a_i\leq k$, $~ 1\leq b_i\leq m$, $~
1\leq i\leq r$,  weighted by the  contents
$((\alpha,0^{n-k});\; (\delta,0^{n-m}))\in $$\mathbb{N}^k\times\mathbb{N}^m$,  with $\alpha_j$  the number of $ i$'s such that $a_i = j$, and $\delta_j$  the number
of $i$'s such that $b_i = j$, is
 $$\displaystyle{\prod_{(i,j)\in \lambda }(1-x_iy_j)^{-1}=\sum_{\begin{smallmatrix}\{(a_i,b_i)\}_{i=1}^r\\
 r\ge 0\end{smallmatrix}}
x_{a_1} y_{b_1} \cdots  x_{a_r} y_{b_r}}=\sum_{\begin{smallmatrix}\{(a_i,b_i)\}_{ i=1}^ r\\
r\ge 0\end{smallmatrix}}x^\alpha y^{\delta}.$$
Each multiset
$\{(a_1, b_1),$ $ (a_2, b_2),$ $\ldots , $ $(a_r, b_r)\}$, $r\ge 0$, and, hence, each monomial $ x_{a_1}$ $ y_{b_1} \cdots$ $  x_{a_r}
y_{b_r}$, $r\ge 0$, is in one-to-one correspondence   with the lexicographically ordered biword $ \left(
\begin{smallmatrix}a_{r}&\cdots&a_{1}\\ b_{r}&\cdots&b_{1} \end{smallmatrix} \right)$ in the product of alphabets $[k]\times[m]$. In turn, each biword is bijectively mapped  by $\Phi$  into  the pair  $(F,G)$ of SSAFs such that $G$ has entries in $\{a_1,\ldots , a_r\}$, $F$ has entries in $\{b_1,\ldots , b_r\}$, and their shapes $sh(G)=\mu\in\mathbb{N}^k$, and $sh(F)=\beta\in\mathbb{N}^m$, in a same $\mathfrak{S}_n$-orbit, satisfy $(\beta,0^{n-m})\leq  (0^{n-k},\omega\mu)$ with $\omega$ the longest permutation in $\mathfrak{S}_k$. (For $r=0$, put $F=G=\emptyset$.)
Thereby, $  x_{a_1} y_{b_1} \cdots  x_{a_r} y_{b_r} =y^{F}x^{G}$, for all $r\ge 0$. Assume $k\le m$. Since $(\mu,0^{n-k}),$ $(\beta,0^{n-m})$ are in a same $\mathfrak{S}_n$-orbit, $(\mu^+,0^{m-k})=$ $\beta^+\in \mathbb{N}^m$. We then may write
\begin{eqnarray} \label{last1113} \nonumber\prod_{(i,j)\in
\lambda}(1-x_iy_j)^{-1}&=&\sum_{\begin{smallmatrix}\mu \in
\mathbb{N}^k\end{smallmatrix}}\sum_{\begin{smallmatrix}F,G \,\in
SSAF_n\\
sh(F)=\beta\in\mathbb{N}^m,~sh(G)=\mu\\
(\beta,0^{n-m})\le (0^{n-k},\omega\mu) \end{smallmatrix}}y^Fx^G \end{eqnarray}
\begin{eqnarray} &=&
\sum_{\begin{smallmatrix}\mu \in
\mathbb{N}^k \end{smallmatrix}}\left(\sum_{\begin{smallmatrix}G\in
{SSAF}_n\\sh(G)=\mu\end{smallmatrix}}x^{G}\right)\left(\sum_{\begin{smallmatrix}
\beta \in
\mathbb{N}^m\\ (\beta,0^{n-m})\leq (0^{n-k},\omega\mu)
\end{smallmatrix}}\sum_{\begin{smallmatrix}F\in
{SSAF}_n\\sh(F)=\beta\end{smallmatrix}}y^{F}\right) \nonumber
\end{eqnarray}
\begin{eqnarray}\nonumber&=&\sum_{\begin{smallmatrix}\mu \in
\mathbb{N}^k \end{smallmatrix}}\left(\sum_{\begin{smallmatrix}Q\in
SSYT_n\\sh(Q)=\mu^{+}\\K_+(Q)=key(\mu)\end{smallmatrix}}x^{Q}\right)\left(\sum_{\begin{smallmatrix}
\beta \in
\mathbb{N}^m\\ (\beta, 0^{n-m})\leq (0^{n-k}\omega\mu)
\end{smallmatrix}}\sum_{\begin{smallmatrix}P\in
SSYT_n\\sh(P)=\mu^{+}\\K_+(P)=key(\beta)\end{smallmatrix}}y^{P}\right)\\
 &=&
\sum_{\begin{smallmatrix}\mu \in
\mathbb{N}^k\end{smallmatrix}}\widehat{\kappa}_{\mu}(x)\sum_{\begin{smallmatrix}P\in \mathfrak{B}_{(0^{n-k},\omega\mu)}\\ entries \,\le m\end{smallmatrix}}y^P. \label{explain}
\end{eqnarray}
 Given $\mu\in \mathbb{N}^k$, since $m\ge k$,  put $\nu:=(\mu,0^{m-k}, 0^{n-m})$. Then $\omega\nu=(0^{m-k},0^{n-m},\omega\mu)$. Recall that $\mathfrak{B}_{(0^{m-k},\omega\mu^{+},0^{n-m})}=\mathfrak{B}^{(\mu^{+},0^{m-k})}$
 is the crystal graph consisting of all  SSYTs  with shape $(\mu^{+},0^{m-k})$ and entries less or equal than $m$. (The size of the longest permutation $\omega$ should be understood from the context.) Therefore the arrows are $P\overset{i}\rightarrow P'$ for each crystal operator $f_i$ such that $f_i(P)=P'$, $1\le i<m$.  Henceforth, one has
 \begin{equation}\label{keypol}\displaystyle\sum_{\begin{smallmatrix}P\in \mathfrak{B}_{\omega\nu}\\ entries \,\le m\end{smallmatrix}}y^P=\displaystyle\sum_{P\in \mathfrak{B}_{\omega\nu}\cap \,\mathfrak{B}_{(0^{m-k},\omega\mu^{+},0^{n-m})}}y^P,\end{equation}
the weight polynomial  of all  SSYTs in the
$\mathfrak{B}_{\omega\nu}$ with entries less or
equal than $m,$ equivalently, of all  SSYTs with entries $\le m$ and shape $\mu^+\in\mathbb{N}^k$
 whose right key is bounded by $key(0^{n-k},\omega\mu)$. It is also equivalent to consider all SSAFs
such that the shape  has at most $k$ nonzero entries with
zeros in the  last  $n-m$ entries, and is bounded by $\omega \nu$.  At this point we can say that a SSYT, in the intersection of the two Demazure crystals $\mathfrak{B}_{(0^{n-m},0^{m-k},\omega\mu)}\cap \,\mathfrak{B}_{(0^{m-k},\omega\mu^{+},0^{n-m})}$, has shape  in $\mathbb{N}^k$, entries $\le m$, and necessarily its right key is specified by a vector $\zeta\in \mathbb{N}^m$ with at most $k$ non zero entries and satisfying  $\zeta\le (0^{m-k},\omega\mu^+)$. On the other hand,  one also has $(\zeta,0^{n-m})\le (0^{n-m},0^{m-k},\omega\mu)$ despite that $\omega\mu\le \omega\mu^+$. Indeed  $\mathfrak{B}_{\omega\nu}\cap \,\mathfrak{B}_{(0^{m-k},\omega\mu^{+},0^{n-m})}\subseteq  \mathfrak{B}_{(0^{m-k},\alpha,0^{n-m})}$ for some $\alpha\in\mathbb{N}^k$ and $\alpha\le \omega\mu^+$.
Next, we determine the optimal $\alpha\in\mathbb{N}^k$ so  that the Demazure crystal $\mathfrak{B}_{(0^{m-k},\alpha,0^{n-m})}=\mathfrak{B}_{\omega\nu}\cap \,\mathfrak{B}_{(0^{m-k},\omega\mu^{+},0^{n-m})} $.  This shows that \eqref{keypol} is a key polynomial and  simultaneously describes its indexing weak composition. See Example~\ref{intersection}.
\begin{lem}\label{keyintersection} Let
$\gamma\in\mathbb{N}^n$ such that $\gamma^+=(\eta, 0^{n-m})$ is a partition of length $\le m\le n$.  Consider  a  sequence of positive integers $1\le i_M,\ldots,i_1<n$ $($not necessarily a reduced word of $\mathfrak{S}_n$$)$ such that $\kappa_\gamma(y)=\pi_{i_M}\cdots\pi_{i_1}y^{(\eta,0^{n-m})}$. If $j_s,\ldots,j_1$ is the subsequence consisting of  all elements $\ge m$, it holds \begin{eqnarray}{\label{intersec}\displaystyle\sum_{\begin{smallmatrix}P\in \mathfrak{B}_{\gamma}\\ entries\, \le m\end{smallmatrix}}y^P=\displaystyle\sum_{P\in \mathfrak{B}_{\gamma}\cap \mathfrak{B}_{(\omega\eta,0^{n-m})}}y^P=\pi_{i_M}\cdots\tilde\pi_{j_s}\cdots\tilde\pi_{j_1}\cdots\pi_{i_1}y^{(\eta,0^{n-m})},}\end{eqnarray}
 where the tilde ``$\;\;\widetilde{} \;$" means  omission, and $\omega $ is the longest permutation of $\mathfrak{S}_m$.
\end{lem}
\begin{proof}
   Notice that  from the recursive definition of key polynomial or \eqref{dempro},   $\mathfrak{B}_{\gamma}$ $\subseteq $ $\mathfrak{B}^{\gamma^+}$. Also $\mathfrak{B}_{(\omega\eta,0^{n-m})}=  \mathfrak{B}^{\eta}$ $\subseteq \mathfrak{B}^{\gamma^+}$, and $\mathfrak{B}_{\gamma}\cap \mathfrak{B}_{(\omega\eta,0^{n-m})}=$$ \mathfrak{B}_{\gamma}\cap \mathfrak{B}^{\eta}$. If $n=m$ or $\gamma$ has the last $n-m$ entries equal to zero, then $\gamma\le (\omega\eta,0^{n-m})$,  $ \mathfrak{B}_{\gamma}\subseteq $ $\mathfrak{B}^{\eta}$, and $1\le i_M,\ldots,$ $i_1< m$. Henceforth, in this case, $\displaystyle\sum_{\begin{smallmatrix}P\in \mathfrak{B}_{\gamma}\\ entries\, \le m\end{smallmatrix}}y^P=\displaystyle\sum_{P\in \mathfrak{B}_{\gamma}}y^P=\kappa_\gamma(y)$. Otherwise, the intersection of the two graphs $ \mathfrak{B}_{\gamma}\cap \mathfrak{B}^{\eta}$
   is  the graph obtained from $\mathfrak{B}_{\gamma}$ by deleting all the vertices consisting of  SSYTs with entries $>m$, and, therefore, all $i$-edges incident on them (either getting in or out), in particular, those  with $i\ge m$.   This means that all $i$-strings, with $i\ge m$, in $\mathfrak{B}_{\gamma}$, are deleted, while just the heads
remain, in the case of $i=m$. Furthermore, every $i$-string with $i<m$ whose head has an entry $>m$ is ignored.  In conclusion, $ \mathfrak{B}_{\gamma}\cap \mathfrak{B}^{\eta}$ consists of the $i$-strings in $ \mathfrak{B}_{\gamma}$ with $i<m$  whose heads have entries $\le m$.
From the combinatorial interpretation of Demazure operators $\pi_i$ in terms of  the $i$-strings of a crystal graph, \eqref{demazuregraph}, \eqref{demazureoperator}, \eqref{demazureoperator2}, this means
   we are deleting in $\pi_{i_M}\cdots$$\pi_{i_2}$$\pi_{i_1}y^{(\eta,0^{n-m})}$  the action of the Demazure operators $\pi_i$ for $i\ge m$,
and,  thanks to
 \eqref{dempro}, one still has  a key polynomial, precisely, \eqref{intersec}.
\end{proof}
We now  calculate the indexing weak composition of the key polynomial \eqref{intersec} in the case $\eta=(\mu^+, 0^{m-k})$ and $\gamma=\omega\nu=$$\omega(\mu,0^{m-k},0^{n-m})$, and, therefore, the key polynomial \eqref{keypol}.
\begin{prop} \label{keyvector} Let $1\leq
k\leq m\leq n,$ and $n-m+1\leq k$. Given $\mu \in\mathbb{N}^k$,
 let $\alpha=(\alpha_1,\ldots,\alpha_k)\in \mathbb{N}^k$ such that for each $ i=k,\ldots, 1$, the entry $\alpha_i$ is the maximum element among the last  $\min\{i,n-m+1\}$ entries of $\omega\mu$  after deleting $\alpha_j$, for $i<j\le k$.    Then, if  $\nu=(\mu,0^{m-k},0^{n-m})$,
\begin{enumerate}

\item
 \begin{eqnarray}
\sum_{\begin{smallmatrix}P\in \mathfrak{B}_{\omega\nu}\\ entries\, \le m\end{smallmatrix}}y^P&=&
\sum_{P\in \mathfrak{B}_{\omega\nu}\cap \,\mathfrak{B}_{(0^{m-k},\omega\mu^{+},0^{n-m})}}y^P
=\sum_{\begin{smallmatrix}P\in\mathfrak{B}_{(0^{m-k},\alpha,0^{n-m})}\end{smallmatrix}}  y^P\nonumber\\
&=&\pi_{\sigma(\lambda,SE)}\kappa_{(\omega\mu,0^{n-k})}(y)=\kappa_{(0^{m-k},\alpha,0^{n-m})}(y).\end{eqnarray}

\item $\mathfrak{B}_{\omega(\mu,0^{m-k},0^{n-m})}\cap \,\mathfrak{B}_{(0^{m-k},\omega\mu^{+},0^{n-m})}=\mathfrak{B}_{(0^{m-k},\alpha,0^{n-m})}$ and $ \omega\mu\le \alpha\le\omega\mu^+$.

    In particular, when $m=n$, then $\alpha=\omega\mu$; and  when $m+k= n+1$, then $\alpha=\omega\mu^{+}$ and $\kappa_{(0^{m-k},\omega\mu^{+},0^{n-m})}(y)=s_{(\mu^+,0^{m-k})}(y_1,\dots,y_m)$ is a Schur polynomial.
\end{enumerate}
\end{prop}
\begin{proof}
 $1$. Recalling the action  of  Demazure operators  $\pi_i$ on key polynomials  via bubble sorting operators on their indexing weak compositions \eqref{dempro}, and since  $\omega\nu=(0^{n-k},\omega\mu)$, one may write,
 \begin{eqnarray}\label{eqpre} \kappa_{\omega\nu}(y)&=&\prod_{i=1}^{k}(\pi_{i+n-k-1}\dots \pi_i)\kappa_{(\omega\mu,0^{n-k})}(y)
 \end{eqnarray}
 \begin{eqnarray}
&=&
\prod_{i=1}^{k-(n-m)-1}(\pi_{i+n-k-1}\dots \pi_i)\label{eq0}
\\ \label{eq}
&&\textbf{\LARGE.}\prod_{i=0}^{n-m}(\pi_{m-1+i}\dots \pi_{k-(n-m)+i})
\kappa_{(\omega\mu,0^{n-k})}(y).
\end{eqnarray}
The Demazure operators in \eqref{eqpre}  act as bubble sorting operators on the weak composition
$(\omega\mu,0^{n-k})$, shifting $n-k$ times to the right each of the $k$ entries of $\omega\mu$.  
This can be done by shifting, first, the last $n-m+1$ entries of $\omega\mu$ \eqref{eq} and then \eqref{eq0} the remaining first $k-(n-m)-1\ge 0$ entries.
From Lemma~\ref{keyintersection}, with $\eta=(\mu^+, 0^{m-k})$ and $\gamma=\omega\nu$, omitting in
 \eqref{eq} the operators with indices $\ge m$, one has
\begin{align}\nonumber\displaystyle\sum_{\begin{smallmatrix}P\in \mathfrak{B}_{\omega\nu}\\ entries \,\le m\end{smallmatrix}}y^P &=\displaystyle\sum_{P\in \mathfrak{B}_{\omega\nu}\cap \mathfrak{B}^{(\mu^{+},0^{m-k})}}y^P=\pi_{\sigma(\lambda,SE)}\kappa_{(\omega\mu,0^{n-k})}(y)
\\
&=\prod_{i=1}^{k-(n-m)-1}(\pi_{i+n-k-1}\dots \pi_i) \label{keypol10}\\
&\kern3cm
\textbf{\LARGE.}\prod_{i=0}^{n-m}(\pi_{m-1}\dots \pi_{k-(n-m)+i})\label{keypol0}
\kappa_{(\omega\mu,0^{n-k})}(y)\\
&=\kappa_{(0^{m-k},\alpha,0^{n-m})}(y).\label{keypol2}
\end{align}
The Demazure operators in  \eqref{keypol0} act as bubble sorting operators on the weak composition $(\omega\mu,0^{m-k},0^{n-m})$, shifting $m-k$ times to the right the last $n-m+1$ entries of $\omega\mu$, and sorting them in ascending order. Next, the  operators \eqref{keypol10} act similarly on the resulting vector ignoring the entry $m$, then ignoring the entry $m-1$, and so on.
Thus the weak composition indexing  the new key polynomial  $\kappa_{(0^{m-k},\alpha,0^{n-m})}$ \eqref{keypol2} is such that $\alpha=(\alpha_1,\dots,\alpha_k)$, where for each $i=k,\ldots,1$,
$\alpha_i$ is the maximum element of the last  $\min\{i,n-m+1\}$ entries of $\omega\mu$  after deleting $\alpha_j$, for $i<j\le k$.
(After some point, the number of remaining entries in $\omega\mu$ is less than $n-m+1$ and   just the $i$ remaining entries are considered.)

$2$.  It is a consequence of $1$,  recalling that, in Section~\ref{sec:Bruhat}, the left Bruhat order (implies Bruhat order) in $\mathfrak{S}_k\mu$ is described by bubble sorting operators. An alternative proof comes from the construction of $\alpha$, provided $\omega\mu$, and using the Bruhat order
 characterization  \eqref{inducedbruhat} in an orbit.
Start with $\alpha^0:=\omega\mu$. Next put, for $i=0,\dots, k-1$,
 $\alpha^{i+1}$   equal to the result of swapping in $\alpha^i$  the $i+1$-th last entry of $\alpha^i$  with the maximum among the last $min\{k-i,n-m+1\}$ entries in $\alpha^i$, after ignoring the $i$ last entries. Eventually, one obtains $\alpha$. In
each step, one has  $\alpha^i \leq \alpha^{i+1}$, for $i\ge 0$, and finally
$\omega\mu\le \alpha$.
\end{proof}
Example~\ref{intersection} illustrates this proposition.
\begin{thm}\label{minmin} Let $\lambda=(m^{n-m+1}, m-1,\dots,n-k+1),$ where $1\leq
k, m\leq n,$ and $n+1\leq m+k$, be the Ferrers shape in Figure~\ref{fig:trunc}. Then we have the following explicit expansions in the SSYTs in  the Demazure crystal
\begin{enumerate}
\item If $1\leq k\leq m$,
\begin{align} \nonumber  \prod_{\begin{smallmatrix}(i,j)\in
\lambda\\
k\le m\end{smallmatrix}}(1-x_iy_j)^{-1}
&=\sum_{\mu\in \mathbb{N}^k}
\widehat{\kappa}_{\mu}(x)\pi_{\sigma(\lambda,SE)}
\kappa_{\omega\mu}(y)\\
&=\sum_{\begin{smallmatrix}\mu \in
\mathbb{N}^k
 \end{smallmatrix}}\widehat{\kappa}_{\mu}(x)\kappa_{(0^{m-k},\alpha)}(y),
\label{last111}
\end{align}
\noindent where $\alpha\in\mathbb{N}^k$ is defined in Proposition~\ref{keyvector} for each $\mu\in \mathbb{N}^k$.
\item If $1\leq m\leq k$,
\begin{align} \nonumber\label{lastone}
\prod_{\begin{smallmatrix}(i,j)\in
\lambda\\
m\le k\end{smallmatrix}}(1-x_iy_j)^{-1}&= \sum_{\mu\in \mathbb{N}^m}
\pi_{\sigma(\lambda,NW)}\hat\kappa_\mu(x)
\kappa_{\omega\mu}(y)\\
\nonumber
&=\sum_{\mu\in \mathbb{N}^m}
\widehat{\kappa}_{\mu}(y)\pi_{\sigma(\lambda,NW)}
\kappa_{\omega\mu}(x)\\
&=\sum_{\begin{smallmatrix}\mu \in
\mathbb{N}^m
\end{smallmatrix}}\kappa_{(0^{k-m},\alpha')}(x)\widehat{\kappa}_{\mu}(y),
\end{align}
 \noindent  where
 $\alpha'\in\mathbb{N}^m$  is defined similarly,
swapping $k$ with $m$ in Proposition~\ref{keyvector}, for each $\mu\in\mathbb{N}^m$.
\end{enumerate}
\end{thm}
\begin{proof} $1$. Identity~\eqref{last111} follows from \eqref{explain} and Proposition~\ref{keyvector}.

$2$. Considering $\overline\lambda$, the conjugate of $\lambda$, and the expansion \eqref{last111}, one has
\begin{align} \nonumber  \prod_{\begin{smallmatrix}(i,j)\in
\lambda\\
m\le k\end{smallmatrix}}(1-x_iy_j)^{-1}&=\prod_{\begin{smallmatrix}(j,i)\in
\overline\lambda\\
m\le k\end{smallmatrix}}(1-x_iy_j)^{-1}=\sum_{\mu\in \mathbb{N}^m}
\widehat{\kappa}_{\mu}(y)\pi_{\sigma(\overline\lambda,SE)}
\kappa_{\omega\mu}(x) \\ \label{lefthand}
& =\sum_{\mu\in \mathbb{N}^m}
\widehat{\kappa}_{\mu}(y)\pi_{\sigma(\lambda,NW)}
\kappa_{\omega\mu}(x) =
 \sum_{\begin{smallmatrix}\mu \in
\mathbb{N}^m
\end{smallmatrix}}\widehat{\kappa}_{\mu}(y)\kappa_{(0^{k-m},\alpha')}(x),
\end{align}
where $\alpha'\in\mathbb{N}^m$  is defined,
swapping $k$ with $m$, in Proposition~\ref{keyvector}, for each $\mu\in\mathbb{N}^m$.
On the other hand, using the change of basis \eqref{bruhatdec2}, one has \eqref{conjugate}, which together with \eqref{lefthand} gives
\begin{align}\nonumber\prod_{\begin{smallmatrix}(i,j)\in
\lambda\\
m\le k\end{smallmatrix}}(1-x_iy_j)^{-1}&=\sum_{\mu\in \mathbb{N}^m}\pi_{\sigma(\lambda,NW)}\widehat\kappa_{\mu}(x)\kappa_{\omega\mu}(y)\\ \nonumber
&=\sum_{\begin{smallmatrix}\mu \in
\mathbb{N}^m
\end{smallmatrix}}\widehat{\kappa}_{\mu}(y)\kappa_{(0^{k-m},\alpha')}(x).
\end{align}
\end{proof}
In Figure~\ref{fig:trunc}, if $m=n$,
  ${\lambda}=(n,n-1,\dots,n-k+1),$ with $1\leq k\leq n,$ and from  identity \eqref{last111} and Proposition~\ref{keyvector}, one has
\begin{eqnarray}\label{refine} \nonumber\prod_{(i,j)\in \lambda
}(1-x_iy_j)^{-1}=
\sum_{\begin{smallmatrix}\mu\in \mathbb{N}^{k}\\
\nu=(\mu,0^{n-k})\end{smallmatrix}}\widehat{\kappa}_{\nu}(x)\kappa_{\omega\nu}(y).
\end{eqnarray}
(Similarly, for $k=n$, in identity \eqref{lastone}.) In particular, if $m=n=k$ ($\lambda=\overline\lambda$), we recover \eqref{cau21} from  both previous identities.
When $n+1=m+k$, from Proposition~\ref{keyvector}, identity \eqref{last111} becomes \eqref{last},
and hence we recover identity \eqref{aaa1} with $k\le m$. Similarly, \eqref{lastone} leads to \eqref{aaa1} with $m\le k$.
\begin{ex} \label{intersection} Let $n=5,~ m=4$ $\ge  k=3$, $\mu=(1,1,2)$, and $\nu=(1,1,2,0,0)$.
The black and blue tableaux constitute  the vertices of the Demazure crystal
$\mathfrak{B}_{\omega\nu}=\mathfrak{B}_{(0,0,2,1,1)}=$ $\mathfrak{B}_{s_2s_1s_3s_2s_4s_3(2,1,1,0,0)}$.
One has $\pi_2\pi_1\pi_3\pi_2\pi_3x^{(2,1,1,0,0)}$$=\pi_2\pi_1$$\pi_2\pi_3x^{(2,1,1,0,0)}$$=\kappa_{(0,1,2,1,0)}(x)$.
{\em(}The  shortest element in the coset $s_2s_1s_3s_2s_3<s_2>$ is $s_2s_1s_2s_3$.{\em )}  The black and the red tableaux are the vertices of  the
crystal $
\mathfrak{B}_{(0,\omega\mu^{+},0)}=$ $\mathfrak{B}_{(0,1,1,2,0)}$ $=$ $\mathfrak{B}_{s_1s_2s_3s_2s_1\nu^+}$.
The intersection  $\mathfrak{B}_{\omega\nu}\cap
\mathfrak{B}_{(0,\omega\mu^{+},0)}$ consists of the black
tableaux which constitute the vertices of the Demazure crystal
$\mathfrak{B}_{(0,\alpha,0)}=$ $\mathfrak{B}_{(0,1,2,1,0)}=$ $\mathfrak{B}_{s_2s_1s_2s_3(2,1,1,0,0)}$, with $\alpha=(1,2,1)$ defined in Proposition~\ref{keyvector}. $($Note that the crystal graph does not have  all the edges represented. Only those referring to the words under consideration.$)$

\begin{tikzpicture}[scale=0.75]
\draw[line width=1pt] (4,17.5) rectangle (4.5,18);
\draw[line width=1pt] (4.5,17.5) rectangle (5,18);
\draw[line width=1pt] (4,18) rectangle (4.5,18.5);
\draw[line width=1pt] (4,18.5) rectangle (4.5,19);
\node at (4.25,17.75){\small 1};\node at (4.75,17.75){\small 1};\node at
(4.25,18.25){\small 2};\node at (4.25,18.75){\small 3};
\draw[arrows=->,line width=1pt]  (3.9, 17.3)--(3.1,16) ;\node at
(3.96,16.8){\scriptsize$1$};
\draw[arrows=->,line width=1pt] (5.1,17.3) -- (5.9, 16);\node at
(5.7,16.8){\scriptsize$3$};
\draw[line width=1pt] (2,15) rectangle (2.5,15.5);
\draw[line width=1pt] (2.5,15) rectangle (3,15.5);
\draw[line width=1pt] (2,15.5) rectangle (2.5,16);
\draw[line width=1pt] (2,16) rectangle (2.5,16.5);
\node at (2.25,15.25){\small 1};\node at (2.75,15.25){\small 2};\node at
(2.25,15.75){\small 2};\node at (2.25,16.25){\small 3};

\draw[line width=1pt] (6,15) rectangle (6.5,15.5);
\draw[line width=1pt] (6.5,15) rectangle (7,15.5);
\draw[line width=1pt] (6,15.5) rectangle (6.5,16);
\draw[line width=1pt] (6,16) rectangle (6.5,16.5);
\node at (6.25,15.25){\small 1};\node at (6.75,15.25){\small 1};\node at
(6.25,15.75){\small 2};\node at (6.25,16.25){\small 4};

\draw[line width=1pt, color=blue] (8,15) rectangle (8.5,15.5);
\draw[line width=1pt, color=blue] (8.5,15) rectangle (9,15.5);
\draw[line width=1pt, color=blue] (8,15.5) rectangle (8.5,16);
\draw[line width=1pt, color=blue] (8,16) rectangle (8.5,16.5);
\node at (8.25,15.25){\small 1};\node at (8.75,15.25){\small 1};\node at
(8.25,15.75){\small 2};\node at (8.25,16.25){\small 5};

\draw[arrows=->,line width=1pt] (2.5,14.9) -- (2.5, 14.1);\node at
(2.85,14.5){\scriptsize$2$};
 \draw[arrows=->,line width=1pt, color=red] (3.1,14.9) -- (3.9,
14.1);\node at
(3.8,14.7){{\color{red}\scriptsize$3$}};
 \draw[arrows=->,line width=1pt]  (5.9, 14.8)-- (5.1,14);\node at
(5.2,14.7){\scriptsize$1$};
\draw[arrows=->,line width=1pt] (6.5, 14.9)--(6.5,14.1) ;\node at
(6.85,14.5){\scriptsize$2$};
\draw[arrows=->,line width=1pt,color=blue] (8.5, 14.9)-- (8.5,14.1);\node at
(8.85,14.5){{\color{blue}\scriptsize$2$}};
\draw[arrows=->,line width=1pt,color=blue] (9.1,14.8) -- (9.9, 14.1);\node at
(9.7,14.7){{\color{blue}\scriptsize$1$}};
\draw[arrows=->,line width=1pt, color=blue] (7.1,15.75) -- (7.9,
15.75);\node at
(7.6,16.1){{\color{blue}\scriptsize$4$}};

\draw[line width=1pt] (2,12.5) rectangle (2.5,13);
\draw[line width=1pt] (2.5,12.5) rectangle (3,13);
\draw[line width=1pt] (2,13) rectangle (2.5,13.5);
\draw[line width=1pt] (2,13.5) rectangle (2.5,14);
\node at (2.25,12.75){\small 1};\node at (2.75,12.75){\small 3};\node at
(2.25,13.25){\small 2};\node at (2.25,13.75){\small 3};

\draw[line width=1pt] (4,12.5) rectangle (4.5,13);
\draw[line width=1pt] (4.5,12.5) rectangle (5,13);
\draw[line width=1pt] (4,13) rectangle (4.5,13.5);
\draw[line width=1pt] (4,13.5) rectangle (4.5,14);
\node at (4.25,12.75){\small 1};\node at (4.75,12.75){\small 2};\node at
(4.25,13.25){\small 2};\node at (4.25,13.75){\small 4};

\draw[line width=1pt] (6,12.5) rectangle (6.5,13);
\draw[line width=1pt] (6.5,12.5) rectangle (7,13);
\draw[line width=1pt] (6,13) rectangle (6.5,13.5);
\draw[line width=1pt] (6,13.5) rectangle (6.5,14);
\node at (6.25,12.75){\small 1};\node at (6.75,12.75){\small 1};\node at
(6.25,13.25){\small 3};\node at (6.25,13.75){\small 4};

\draw[line width=1pt,color=blue] (8,12.5) rectangle (8.5,13);
\draw[line width=1pt,color=blue] (8.5,12.5) rectangle (9,13);
\draw[line width=1pt,color=blue] (8,13) rectangle (8.5,13.5);
\draw[line width=1pt,color=blue] (8,13.5) rectangle (8.5,14);
\node at (8.25,12.75){\small 1};\node at (8.75,12.75){\small 1};\node at
(8.25,13.25){\small 3};\node at (8.25,13.75){\small 5};

\draw[line width=1pt,color=blue] (10,12.5) rectangle (10.5,13);
\draw[line width=1pt,color=blue] (10.5,12.5) rectangle (11,13);
\draw[line width=1pt,color=blue] (10,13) rectangle (10.5,13.5);
\draw[line width=1pt,color=blue] (10,13.5) rectangle (10.5,14);
\node at (10.25,12.75){\small 1};\node at (10.75,12.75){\small 2};\node at
(10.25,13.25){\small 2};\node at (10.25,13.75){\small 5};

\draw[arrows=->,line width=1pt, color=red]  (2.5, 12.4)--(2.5,11.6) ;\node at
(2.85,12){{\color{red}\scriptsize$3$}};
\draw[arrows=->,line width=1pt]   (4.5, 12.4)--(4.5,11.6);\node at
(4.85,12){\scriptsize$2$};
\draw[arrows=->,line width=1pt] (6.5, 12.4)--(6.5,11.6) ;\node at
(6.85,12){\scriptsize$1$};
\draw[arrows=->,line width=1pt, color=blue]   (8.5,
12.4)--(8.5,11.6);\node at
(8.85,12){{\color{blue}\scriptsize$3$}};
\draw[arrows=->,line width=1pt,color=blue] (9.1,12.4) -- (9.9, 11.4);\node at
(9.7,12.3){{\color{blue}\scriptsize$1$}};
\draw[arrows=->,line width=1pt,color=blue] (11.1,12.4) -- (11.9,
11.4);\node at
(11.7,12.3){{\color{blue}\scriptsize$2$}};


\draw[line width=1pt,color=red] (2,10) rectangle (2.5,10.5);
\draw[line width=1pt,color=red] (2.5,10) rectangle (3,10.5);
\draw[line width=1pt,color=red] (2,10.5) rectangle (2.5,11);
\draw[line width=1pt,color=red] (2,11) rectangle (2.5,11.5);
\node at (2.25,10.25){\small 1};\node at (2.75,10.25){\small 4};\node at
(2.25,10.75){\small 2};\node at (2.25,11.25){\small 3};

\draw[line width=1pt] (4,10) rectangle (4.5,10.5);
\draw[line width=1pt] (4.5,10) rectangle (5,10.5);
\draw[line width=1pt] (4,10.5) rectangle (4.5,11);
\draw[line width=1pt] (4,11) rectangle (4.5,11.5);
\node at (4.25,10.25){\small 1};\node at (4.75,10.25){\small 3};\node at
(4.25,10.75){\small 2};\node at (4.25,11.25){\small 4};

\draw[line width=1pt] (6,10) rectangle (6.5,10.5);
\draw[line width=1pt] (6.5,10) rectangle (7,10.5);
\draw[line width=1pt] (6,10.5) rectangle (6.5,11);
\draw[line width=1pt] (6,11) rectangle (6.5,11.5);
\node at (6.25,10.25){\small 1};\node at (6.75,10.25){\small 2};\node at
(6.25,10.75){\small 3};\node at (6.25,11.25){\small 4};

\draw[line width=1pt,color=blue] (8,10) rectangle (8.5,10.5);
\draw[line width=1pt,color=blue] (8.5,10) rectangle (9,10.5);
\draw[line width=1pt,color=blue] (8,10.5) rectangle (8.5,11);
\draw[line width=1pt,color=blue] (8,11) rectangle (8.5,11.5);
\node at (8.25,10.25){\small 1};\node at (8.75,10.25){\small 1};\node at
(8.25,10.75){\small 4};\node at (8.25,11.25){\small 5};

\draw[line width=1pt,color=blue] (10,10) rectangle (10.5,10.5);
\draw[line width=1pt,color=blue] (10.5,10) rectangle (11,10.5);
\draw[line width=1pt,color=blue] (10,10.5) rectangle (10.5,11);
\draw[line width=1pt,color=blue] (10,11) rectangle (10.5,11.5);
\node at (10.25,10.25){\small 1};\node at (10.75,10.25){\small 2};\node at
(10.25,10.75){\small 3};\node at (10.25,11.25){\small 5};

\draw[line width=1pt,color=blue] (12,10) rectangle (12.5,10.5);
\draw[line width=1pt,color=blue] (12.5,10) rectangle (13,10.5);
\draw[line width=1pt,color=blue] (12,10.5) rectangle (12.5,11);
\draw[line width=1pt,color=blue] (12,11) rectangle (12.5,11.5);
\node at (12.25,10.25){\small 1};\node at (12.75,10.25){\small 3};\node at
(12.25,10.75){\small 2};\node at (12.25,11.25){\small 5};

\draw[arrows=->,line width=1pt, color=red] (2.5,9.9) -- (2.5, 9.1);\node at
(2.85,9.5){{\color{red}\scriptsize$3$}};
\draw[arrows=->,line width=1pt] (4.5,9.9) -- (4.5, 9.1);\node at
(4.85,9.5){\scriptsize$2$};
\draw[arrows=->,line width=1pt] (6.5,9.9) -- (6.5, 9.1);\node at
(6.85,9.5){\scriptsize$1$};
\draw[arrows=->,line width=1pt, color=blue] (8.5,9.9) -- (8.5, 9.1);\node at
(8.85,9.5){{\color{blue}\scriptsize$1$}};
\draw[arrows=->,line width=1pt,color=blue] (11.1,9.9) -- (11.9, 9.1);\node at
(11.7,9.7){{\color{blue}\scriptsize$1$}};
\draw[arrows=->,line width=1pt,color=blue] (13.1,9.9) -- (13.9, 9.1);\node at
(13.7,9.7){{\color{blue}\scriptsize$2$}};

\draw[line width=1pt,color=red] (2,7.5) rectangle (2.5,8);
\draw[line width=1pt,color=red] (2.5,7.5) rectangle (3,8);
\draw[line width=1pt,color=red] (2,8) rectangle (2.5,8.5);
\draw[line width=1pt,color=red] (2,8.5) rectangle (2.5,9);
\node at (2.25,7.75){\small 1};\node at (2.75,7.75){\small 4};\node at
(2.25,8.25){\small 2};\node at (2.25,8.75){\small 4};

\draw[line width=1pt] (4,7.5) rectangle (4.5,8);
\draw[line width=1pt] (4.5,7.5) rectangle (5,8);
\draw[line width=1pt] (4,8) rectangle (4.5,8.5);
\draw[line width=1pt] (4,8.5) rectangle (4.5,9);
\node at (4.25,7.75){\small 1};\node at (4.75,7.75){\small 3};\node at
(4.25,8.25){\small 3};\node at (4.25,8.75){\small 4};

\draw[line width=1pt] (6,7.5) rectangle (6.5,8);
\draw[line width=1pt] (6.5,7.5) rectangle (7,8);
\draw[line width=1pt] (6,8) rectangle (6.5,8.5);
\draw[line width=1pt] (6,8.5) rectangle (6.5,9);
\node at (6.25,7.75){\small 2};\node at (6.75,7.75){\small 2};\node at
(6.25,8.25){\small 3};\node at (6.25,8.75){\small 4};

\draw[line width=1pt,color=blue] (8,7.5) rectangle (8.5,8);
\draw[line width=1pt,color=blue] (8.5,7.5) rectangle (9,8);
\draw[line width=1pt,color=blue] (8,8) rectangle (8.5,8.5);
\draw[line width=1pt,color=blue] (8,8.5) rectangle (8.5,9);
\node at (8.25,7.75){\small 1};\node at (8.75,7.75){\small 2};\node at
(8.25,8.25){\small 4};\node at (8.25,8.75){\small 5};

\draw[line width=1pt,color=blue] (12,7.5) rectangle (12.5,8);
\draw[line width=1pt,color=blue] (12.5,7.5) rectangle (13,8);
\draw[line width=1pt,color=blue] (12,8) rectangle (12.5,8.5);
\draw[line width=1pt,color=blue] (12,8.5) rectangle (12.5,9);
\node at (12.25,7.75){\small 2};\node at (12.75,7.75){\small 2};\node at
(12.25,8.25){\small 3};\node at (12.25,8.75){\small 5};

\draw[line width=1pt,color=blue] (14,7.5) rectangle (14.5,8);
\draw[line width=1pt,color=blue] (14.5,7.5) rectangle (15,8);
\draw[line width=1pt,color=blue] (14,8) rectangle (14.5,8.5);
\draw[line width=1pt,color=blue] (14,8.5) rectangle (14.5,9);
\node at (14.25,7.75){\small 1};\node at (14.75,7.75){\small 3};\node at
(14.25,8.25){\small 3};\node at (14.25,8.75){\small 5};

\draw[arrows=->,line width=1pt, color=red]   (2.5, 7.4)--(2.5,6.6);\node at
(2.85,7){{\color{red}\scriptsize$2$}};
\draw[arrows=->,line width=1pt, color=red]  (5.1, 7.4)--(5.9,6.6) ;\node at
(5.1,6.9){{\color{red}\scriptsize$1$}};
\draw[arrows=->,line width=1pt] (6.5, 7.4)-- (6.5,6.6) ;\node at
(6.85,7){\scriptsize$2$};
\draw[arrows=->,line width=1pt, color=blue]  (8.5, 7.4)--(8.5,6.6) ;\node at
(8.85,7){{\color{blue}\scriptsize$1$}};
\draw[arrows=->,line width=1pt,color=blue] (9.1,7.4) -- (9.9, 6.4);\node at
(9.7,7.2){{\color{blue}\scriptsize$2$}};
\draw[arrows=->,line width=1pt,color=blue] (13.1,7.4) -- (13.9, 6.4);\node at
(13.7,7.2){{\color{blue}\scriptsize$2$}};

\draw[line width=1pt,color=red] (2,5) rectangle (2.5,5.5);
\draw[line width=1pt,color=red] (2.5,5) rectangle (3,5.5);
\draw[line width=1pt,color=red] (2,5.5) rectangle (2.5,6);
\draw[line width=1pt,color=red] (2,6) rectangle (2.5,6.5);
\node at (2.25,5.25){\small 1};\node at (2.75,5.25){\small 4};\node at
(2.25,5.75){\small 3};\node at (2.25,6.25){\small 4};

\draw[line width=1pt] (6,5) rectangle (6.5,5.5);
\draw[line width=1pt] (6.5,5) rectangle (7,5.5);
\draw[line width=1pt] (6,5.5) rectangle (6.5,6);
\draw[line width=1pt] (6,6) rectangle (6.5,6.5);
\node at (6.25,5.25){\small 2};\node at (6.75,5.25){\small 3};\node at
(6.25,5.75){\small 3};\node at (6.25,6.25){\small 4};

\draw[line width=1pt,color=blue] (8,5) rectangle (8.5,5.5);
\draw[line width=1pt,color=blue] (8.5,5) rectangle (9,5.5);
\draw[line width=1pt,color=blue] (8,5.5) rectangle (8.5,6);
\draw[line width=1pt,color=blue] (8,6) rectangle (8.5,6.5);
\node at (8.25,5.25){\small 2};\node at (8.75,5.25){\small 2};\node at
(8.25,5.75){\small 4};\node at (8.25,6.25){\small 5};

\draw[line width=1pt,color=blue] (10,5) rectangle (10.5,5.5);
\draw[line width=1pt,color=blue] (10.5,5) rectangle (11,5.5);
\draw[line width=1pt,color=blue] (10,5.5) rectangle (10.5,6);
\draw[line width=1pt,color=blue] (10,6) rectangle (10.5,6.5);
\node at (10.25,5.25){\small 1};\node at (10.75,5.25){\small 3};\node at
(10.25,5.75){\small 4};\node at (10.25,6.25){\small 5};

\draw[line width=1pt,color=blue] (14,5) rectangle (14.5,5.5);
\draw[line width=1pt,color=blue] (14.5,5) rectangle (15,5.5);
\draw[line width=1pt,color=blue] (14,5.5) rectangle (14.5,6);
\draw[line width=1pt,color=blue] (14,6) rectangle (14.5,6.5);
\node at (14.25,5.25){\small 2};\node at (14.75,5.25){\small 3};\node at
(14.25,5.75){\small 3};\node at (14.25,6.25){\small 5};

\draw[arrows=->,line width=1pt, color=red] (2.5,4.9) -- (2.5, 4.1);\node at
(2.85,4.5){{\color{red}\scriptsize$1$}};

\draw[arrows=->,line width=1pt, color=blue] (8.5,4.9) -- (8.5, 4.1);\node at
(8.85,4.5){{\color{blue}\scriptsize$2$}};

\draw[line width=1pt,color=red] (2,2.5) rectangle (2.5,3);
\draw[line width=1pt,color=red] (2.5,2.5) rectangle (3,3);
\draw[line width=1pt,color=red] (2,3) rectangle (2.5,3.5);
\draw[line width=1pt,color=red] (2,3.5) rectangle (2.5,4);
\node at (2.25,2.75){\small 2};\node at (2.75,2.75){\small 4};\node at
(2.25,3.25){\small 3};\node at (2.25,3.75){\small 4};

\draw[line width=1pt,color=blue] (8,2.5) rectangle (8.5,3);
\draw[line width=1pt,color=blue] (8.5,2.5) rectangle (9,3);
\draw[line width=1pt,color=blue] (8,3) rectangle (8.5,3.5);
\draw[line width=1pt,color=blue] (8,3.5) rectangle (8.5,4);
\node at (8.25,2.75){\small 2};\node at (8.75,2.75){\small 3};\node at
(8.25,3.25){\small 4};\node at (8.25,3.75){\small 5};

\draw[arrows=->,line width=1pt, color=blue] (8.5, 2.4)--(8.5,1.6) ;\node at
(8.85,2){{\color{blue}\scriptsize$2$}};


\draw[line width=1pt,color=blue] (8,0) rectangle (8.5,0.5);
\draw[line width=1pt,color=blue] (8.5,0) rectangle (9,0.5);
\draw[line width=1pt,color=blue] (8,0.5) rectangle (8.5,1);
\draw[line width=1pt,color=blue] (8,1) rectangle (8.5,1.5);
\node at (8.25,0.25){\small 3};\node at (8.75,0.25){\small 3};\node at
(8.25,0.75){\small 4};\node at (8.25,1.25){\small 5};

\end{tikzpicture}
\end{ex}

\section*{Acknowledgments}
 We thank  Alain Lascoux for letting us know \cite{lascouxcrystal}, and his  paper with Amy M.~Fu~\cite{fulascoux};
  Vic Reiner for suggesting to us to extend  our main theorem to  truncated staircase shapes; and  Viviane Pons for letting us
 know about the implementation
 of key polynomials for Sage-Combinat \cite{pons}.

This work was partially supported by the Centro de Matem\'atica da
Universidade de Coimbra (CMUC), funded by the European Regional
Development Fund through the program COMPETE and by the Portuguese
Government through the FCT --- Funda\c c\~ao para a Ci\^encia e a Tecnologia
under the project PEst-C/MAT/UI0324/2011.
The second author was  also supported by  Funda\c c\~ao para a Ci\^encia e a Tecnologia (FCT) through the Grant SFRH / BD / 33700 / 2009.

 \textit{Email addresses:}\\ 
 
 oazenhas@mat.uc.pt (Olga Azenhas), CMUC, Department of Mathematics, University of Coimbra, 3001-501 Coimbra, Portugal.\\
 
 aramee@mat.uc.pt (Aram Emami) CMUC, Department of Mathematics, University of Coimbra, 3001-501 Coimbra, Portugal.

\end{document}